\documentclass[11pt]{article}

\usepackage{graphicx} 
\usepackage[pdftex]{pict2e}
\usepackage[dvipsnames]{xcolor}
\usepackage{stix}

\input{epsf.sty}
\usepackage{epsfig}
\headheight\baselineskip\headsep2\baselineskip\textwidth210mm
\advance\textwidth -2in\textheight277mm\advance\textheight-2in
\label{seq.sub.gdn}
\usepackage[utf8]{inputenc}
\usepackage{tikz}
\usepackage{makecell}
\usetikzlibrary{shapes.geometric, positioning, shapes, arrows}
\usepackage{amsfonts}
\usepackage{amsthm}
\usepackage{float}
\usepackage{amssymb}
\usepackage{graphicx}
\usepackage{amsmath}

\usepackage{lscape}
\usepackage{color}
\usepackage{multirow}
\usepackage[pdftex]{pict2e}
\usepackage[dvipsnames]{xcolor}
\usepackage{stix}

\definecolor{c30}{rgb}{0,0,1}

\newcommand{\sign}{\text{sign}}

\newtheorem{theorem}{Theorem}[section]
\newtheorem{lemma}[theorem]{Lemma}
\newtheorem{proposition}[theorem]{Proposition}
\newtheorem{corollary}[theorem]{Corollary}
\theoremstyle{definition}

\newtheorem{remark}[theorem]{Remark}
\numberwithin{equation}{section}

\usepackage{color}

\title{
On Analytically Tractable Multidimensional Diffusions via Doob $h$-Transforms, with Resetting and Applications to Wiener and Ornstein--Uhlenbeck Processes   
}

\author{
 Antonio \ {\bf Di Crescenzo}\thanks{
 Email: adicrescenzo@unisa.it \ -- \ ORCID: 0000-0003-4751-7341} 
 \qquad 
 Verdiana \ {\bf  Mustaro}\thanks{
 Email: vmustaro@unisa.it \ -- \ ORCID: 0000-0003-4583-2612} 
 \qquad
 Serena \ {\bf  Spina}\thanks{
 Email: sspina@unisa.it \ -- \ ORCID: 0000-0001-6408-7596}
 \\[4mm]
\normalsize  Dipartimento di Matematica, Universit\`a degli Studi di Salerno,\\
\normalsize Via Giovanni Paolo II, 132, 84084 Fisciano (SA), Italy
}

\begin{document}

\maketitle

\begin{abstract}
We investigate a class of drift transformations of multidimensional diffusion processes generated through Doob $h$-transforms. These transformations provide a systematic approach for constructing analytically tractable stochastic models with prescribed probabilistic and statistical properties. We derive sufficient conditions under which the transformed diffusion admits an explicit transition density expressed through a product form involving a strictly positive harmonic function. Particular choices of this function lead to mixture representations of the transition density and to bimodality. We further analyze the effects of the transformation on stochastic ordering, diffusions in potential landscapes, and Poissonian resetting dynamics. In particular, we show that the product-form relation is preserved under resetting, enabling explicit characterization of the corresponding stationary distributions. Two multidimensional examples based on Wiener and Ornstein--Uhlenbeck processes illustrate the theory, providing closed-form expressions for transition densities, weight functions, and effective potentials. The two-dimensional setting is explored in detail, including symmetry effects and absorbing boundaries.
 


\end{abstract}



\section{Introduction and background}
Transition probability density functions (p.d.f.'s) of multidimensional diffusion processes play a 
fundamental role in the stochastic modeling of dynamical systems. However, explicit closed-form expressions for transition p.d.f.'s are generally unavailable in multidimensional settings, except for a limited number of special cases. In this paper, we 
extend to the multidimensional setting the 
framework for generating new multidimensional transition p.d.f.'s induced by the Doob h-transforms. Then, we 
investigate the properties and interrelationships of the resulting diffusion processes. 

\subsection{Recent advances on multidimensional diffusion processes and transforms}
Prior to this work, several recent studies have significantly advanced the theory and applications 
of multidimensional diffusion processes, addressing a wide range of problems concerning 
theoretical, methodological, and applied aspects. 
\par
Diffusion processes provide a suitable framework for modeling complex systems 
in a variety of applications, including neural activity (see Ditlevsen and L\"ocherbach \cite{Ditlevsen2017}),  geophysics systems (cf.\ R\"ockner et al.\ \cite{RoZhuZhu}), and, more recently, 
denoising diffusion probabilistic models for generative learning (see Bre\v{s}ar and Mijatovi\'c \cite{BreMij}, for instance). 
On the theoretical side, for example, Trevisan \cite{Trevisan} investigates the well-posedness of martingale solutions to stochastic differential equations, establishing a general equivalence between different 
formulations of multidimensional diffusions. 
Eberle et al.\ \cite{Eberle} develop a probabilistic coupling approach to quantify convergence to equilibrium for (kinetic) Langevin processes.  
\par
Recent research has addressed metastability and stochastic resonance in weakly time-inhomogeneous diffusions (Herrmann et al.\ \cite{Herrmann}), recurrence and transience in random environments driven by $\alpha$-stable L'evy processes (Kusuoka et al.\ \cite{Kusuoka}), and conditional independence structures for stationary diffusions with sparse drift functions (Boege et al.\ \cite{Boege}). Related developments include diffusions in bounded domains with boundary resetting (Grigorescu and Kang \cite{Grigorescu}) or killing via branching particle systems (Del Moral and Villemonais \cite{DelMoral}). Advances on 
multidimensional diffusions have been oriented to  methodologies for Fokker-Planck-Kolmogorov equations (Athanassoulis et al.\ \cite{Athanassoulis}), also in the presence of time delay (Loos and Klapp \cite{Loos}).
%
\par
Among the various existing methods for transforming stochastic processes, attention has been given in the literature to the Doob $h$-transform. The latter is a classical and powerful tool for constructing new Markov and diffusion processes by reweighting the transition probabilities of an original process through a positive harmonic function. It provides a unified framework for studying conditioned stochastic processes and has found  applications in probability theory and stochastic analysis. To mention some recent developments, we recall that for one-dimensional diffusions, Doob $h$-transforms have been used to characterize conditioned processes and reveal geometric connections with inversion mappings and time changes (cf.\ Alili et al.\ \cite{Alili}). They also provide a natural description of transient diffusions conditioned to exhibit a different asymptotic behavior, where the scale function acts as the transforming function (cf.\ Hening \cite{Hening}). 
More recently, Kuznetsov and Yuan \cite{Kuznetsov} combined Doob $h$-transforms with Siegmund duality to develop a Darboux transformation for diffusion processes, deriving explicit relationships between the transition densities of the original and transformed diffusions. Beyond the diffusion setting, generalized Doob $h$-transforms have been employed to represent Markov processes conditioned on rare large-deviation events through equivalent driven dynamics, with applications to nonequilibrium statistical mechanics and stochastic control (cf.\ Chetrite and Touchette \cite{Chetrite}). 
These developments illustrate the versatility of Doob $h$-transforms as a general framework for constructing conditioned stochastic processes, generating analytically tractable diffusion models, and modifying the long-term behavior of stochastic dynamics. Some applications can also be found in Borodin and Salminen \cite{Borodin}. 
\subsection{Aims of the paper}
In this paper, we develop a solvable-model construction framework for diffusion processes in $\mathbb{R}^n$ based on Doob $h$-transforms. 
The proposed approach extends the use of product-form transformations of transition p.d.f.'s to the multidimensional setting and provides a systematic procedure for generating new analytically tractable diffusion models. 
The transformation modifies only the drift component while preserving the diffusion matrix of the underlying process. 
As in the one-dimensional case, it yields a product-form relation between the transition p.d.f.'s of non-singular diffusion processes through an auxiliary weight function $h$. 
Compared with other transformation techniques, including monotone and space-time transformations (see 
\c{C}etin \cite{Cetin}, Forman and Sørensen \cite{Forman}, 
Giorno et al.\ \cite{Giorno2006}, 
Nagasawa \cite{Nagasawa}, and Ricciardi \cite{Ricciardi}), 
the present framework generates a distinct class of multidimensional transformed diffusions.
\par
We prove the well-posedness of the proposed transformation in the multidimensional setting. 
Unlike the one-dimensional case, where explicit representations of the weight function may be available, there is no universal form of $h$ in higher dimensions. 
Nevertheless, suitable choices of this function generate a broad family of transformed diffusion processes, provided
that the initial process has at least one attractive endpoint. 
We derive the infinitesimal moments of the transformed process and show that its transition density admits a mixture representation involving the original diffusion.
\par
The general results are illustrated through transformations of the Wiener and Ornstein--Uhlenbeck (OU) processes in $\mathbb{R}^n$. 
In particular, the transformed Wiener process is shown to admit a representation as a mixture of two Wiener processes, leading to transition densities exhibiting at most bimodal behavior. 
This feature provides a mechanism for constructing analytically tractable multimodal diffusion models.
\par
Following previous studies on diffusion processes with reset (see, for instance, 
Magdziarz and Ta\'zbierski \cite{Magdziarz} and Dharmaraja et al.\ \cite{Dharmaraja}), 
we extend the proposed framework to diffusion dynamics with Poissonian resetting, thereby generating new classes of reset diffusion processes with explicit characterizations. 
Finally, motivated by applications, we investigate first-passage problems and show that the combination of the product-form transformation with symmetry properties yields explicit expressions for the associated taboo transition densities.
\par
Although Doob $h$-transforms are classical objects, their systematic exploitation as a tool for constructing multidimensional diffusion models with explicit transition densities, with multimodal behavior, and analytically tractable resetting dynamics appears to have been relatively less explored. 
%
\subsection{Notations}
In the following, we introduce some useful notation that will be recalled throughout the paper. 
\par
Henceforth, multidimensional matrices and vectors will be denoted, respectively, by capital and bold lowercase letters. 
Let also $A=(a_{ij})\in \mathbb{R}^{m\times n}$ be a real matrix, with $m,n \geq 1$. 
We will express the transpose matrix as $A^\top:=(a_{ji})\in \mathbb{R}^{n\times m}$. 
The row and column vectors of $A$ will be represented 
respectively as
\begin{equation}
A_{i,*}:=(a_{i1}\ a_{i2} \ \ldots \ a _{in}),
\qquad  
A_{*,j}:=(a_{1j}\ a_{2j} \ \ldots \ a _{mj})^\top,
\qquad 
i=1,\ldots,m,\;\; j=1,\ldots,n.
\label{eq:Ai*A*j}
\end{equation}
Let $C=(c_{ij})\in \mathbb{R}^{n\times n}$ be an $n$-dimensional square matrix. Then, its determinant and trace will be represented as
$|C|=\det(C)$ and ${\rm Tr}(C)=\sum_{i=1}^n c_{ii}$, respectively.
\par
Let us now introduce the notation for the usual matrix operations. Let $A$ and $B=(b_{ij}) \in \mathbb{R}^{n \times r}$ be two real matrices, with $r \geq 1$ and $c\in \mathbb{R}$. We set $c\ A=(c\ a_{i j})$, $A \cdot B=\left(\sum_{k=1}^n a_{i k}b_{k j}\right)$, $A^d=\left(a_{i j}^d\right)$ for any $d\in\mathbb{N}$. Moreover, assuming that $C\in \mathbb{R}^{n\times n}$ is invertible (i.e. $|C|\neq0$), we will set its inverse as $C^{-1}\in \mathbb{R}^{n\times n}$, such that $C\cdot C^{-1}=C^{-1}\cdot C=\mathbb{1},$
where $\mathbb{1}$ is the $n$-dimensional unit matrix. 
\par
The Euclidean norm of ${\bf v}:=(v_1\ v_2\ \ldots\ v_n)\in \mathbb{R}^n$ will be indicated as
$$
 \lvert\lvert {\bf v} \lvert \lvert:=({\bf v}\cdot {\bf v})^{1/2}=(v_1^2+v_2^2+\ldots+v_n^2)^{1/2}.
$$
Moreover, the Frobenius matrix norm of an $n$-dimensional square matrix $C$ 
will expressed as
$$
\lvert\lvert C \lvert \lvert:= [{\rm Tr}(C^\top \cdot C)]^{1/2}=\left(\sum_{i=1}^n\sum_{j=1}^n (c_{ij})^2\right)^{1/2}.
$$
Let us now assume that $f:[0,T]\times\mathbb{R}^n\to \mathbb{R}$, with $T\in(0,\infty)$, is a suitably differentiable function. For differential calculus we will use:
\begin{align*}
\partial_t f& = \frac{\partial f}{\partial t}, 
  \quad \partial_{x_i} f = \frac{\partial f}{\partial x_i}, 
  \quad \partial_{x_i,x_j} f = \frac{\partial^2 f}{\partial x_i\partial x_j}, \quad \nabla f = (\partial_{x_i} f)_{i=1,\ldots,n},\qquad t\in[0,T],\; i,j=1,\ldots,n.
\end{align*}
\subsection{Mathematical background}\label{sect:backgr}
Let $(\Omega,{\cal F},\mu)$ be a probability space endowed with the (completed) natural filtration ${\cal F}_{t \geq 0}$ of the standard $n$-dimensional Brownian motion ${\bf W}_t$.
Let $\widehat {\bf X}_t=\big(\widehat {\bf X}_t\big)_{t\geq 0}=(\widehat{X}^1_t,\widehat{X}^2_t,\ldots,\widehat{X}^n_t)\in \mathbb{R}^n$, be a time-homogeneous diffusion process defined on the domain ${\cal D}=I_1\times 
I_2\times\ldots\times I_n\subseteq\mathbb{R}^n$, where $I_i=(\alpha_i,\beta_i)$ 
$(i=1,2,\ldots,n)$. 
The process $\widehat {\bf X}_t$ satisfies the following stochastic differential equation (SDE):
$$
{\rm d} \widehat {\bf X}_t
=\widehat{  {\bf b}}\big(\widehat {\bf X}_t\big){\rm d}t
+\widehat{\Sigma}\big(\widehat {\bf X}_t\big){\rm d}{\bf W}_t,
\qquad 
\widehat {\bf X}_0=\widehat{\bf x}_0\in {\cal D},
$$
where $\widehat{{\bf b}}({\bf x})=\left(\widehat{b}_i({\bf x})\right)\in \mathbb{R}^{n}$ is an $n$-dimensional vector, $\widehat{\Sigma}({\bf x})=\left(\widehat{\sigma}_{ij}({\bf x})\right)\in \mathbb{R}^{n \times n}$ an $n \times n$ matrix, $\bf{x} \in \mathbb{R}^{n}$. 
\par
The vector $\widehat{{\bf b}}({\bf x})$ is the  drift vector, whereas $\widehat{\Sigma}\big(\widehat {\bf X}_t\big)$ 
is the dispersion matrix. The  matrix $\widehat A({\bf x})=\big(\widehat a_{ij}({\bf x})\big)\in \mathbb{R}^{n \times n}$ 
is named diffusion matrix, and its elements are given by 
$$
 \widehat a_{ij}({\bf x}) := \sum_{k=1}^n \widehat{\sigma}_{ik} ({\bf x}) \,\widehat{\sigma}_{jk} ({\bf x}),
$$
i.e.\ $\widehat A({\bf x})= \widehat{\Sigma}({\bf x})\, \widehat{\Sigma}({\bf x})^\top$. 
The elements of $\widehat{{\bf b}}({\bf x})$ and  ${\widehat{A}}({\bf x})$ are the infinitesimal moments of 
$\widehat X_t$ given by 
\begin{equation}\label{inf_mom_X_t_hat}
\begin{array}{ll}
 \widehat{b}_i({\bf x})
 = \lim_{\Delta t \to 0^+} \frac{1}{\Delta t}\mathsf E\left[ \widehat{\bf{X}}^i_{t+\Delta t}
 - \widehat{\bf{X}}^i_t\, \Big\lvert\, {\bf \widehat{X}}_t={\bf x}\right], 
 & i=1,2,\ldots,n, \\[3mm]
 \widehat{a}_{ij}({\bf x})
= \lim_{\Delta t \to 0^+} \frac{1}{\Delta t}\mathsf E\left\{\left[\widehat{\bf{X}}^i_{t+\Delta t}-\widehat{\bf{X}}^i_t\right]\left[\widehat{\bf{X}}^j_{t+\Delta t}-\widehat{\bf{X}}^j_t\right] \, \Big\lvert \, \widehat{\bf{X}}_t={\bf x}\right\},
& i,j=1,2,\ldots,n,
\end{array}
\end{equation}
where the matrix  $\widehat{A}({\bf x})$  
is symmetric and  positive definite, and $\widehat{b}_i({\bf x})$ and $\widehat{\sigma}_{ij}({\bf x})$ are $C^1(\mathbb{R}^n)$ bounded measurable functions. Moreover, we suppose that $\widehat {\bf X}_t$ is globally Lipschitz (see, for instance,  Eq. (3.38) of Pavliotis \cite{Pavliotis}), i.e. there exists a positive constant $L>0 $ such that
%
%
\begin{equation} \label{cond_lipschitz_widehatxt}
    \big\lvert\big\lvert \widehat{\bf b}({\bf x})-\widehat{\bf b}({\bf y})\big\lvert\big\lvert+\big\lvert \big\lvert\widehat{\Sigma}({\bf x})-\widehat{\Sigma}({\bf y})\big\lvert\big\lvert\leq L \lvert\lvert{\bf x}-{\bf y}\lvert\lvert, \qquad {\bf x},{\bf y}\in \mathbb{R}^n.
\end{equation}
%
%
%
The infinitesimal generator  associated to $\widehat{\bf{X}}_t$ is
%
%
\begin{equation*}
 \widehat{\cal L} \psi({\bf x}) 
 = \sum_{1\leq i\leq n} \widehat{b}_i({\bf x})\partial_{x_i}\psi({\bf x})
    + \frac{1}{2} \sum_{1\leq i,j\leq n}  \widehat{a}_{ij}({\bf x}) 
\partial_{x_i,x_j}\psi({\bf x}),
\end{equation*}
for smooth test functions $\psi({\bf x})$, $\widehat{b}({\bf x})$ and 
$\widehat{a}_{ij}({\bf x})$. The adjoint operator is
\begin{equation*}
 \widehat{\cal L}^* \psi({\bf y}) 
 = -\sum_{1\leq i\leq n} \partial_{y_i}\big[\widehat{b}_i({\bf y})\psi({\bf y})\big]
    + \frac{1}{2} \sum_{1\leq i,j\leq n}  \partial_{y_i,y_j}\big[ \widehat{a}_{ij}({\bf y}) 
\psi({\bf y})\big].
\label{inf_gen_adj}
\end{equation*}
%
 Furthermore, for all 
$\bf{x},\,\bf{y}\in  {\cal D}$, $\Gamma \in {\cal B(\mathbb{R}^n)}$
and for all $t>\tau\geq 0$, with ${\cal F}_{\tau}^{\widehat{X}}$ the  filtration generated by the stochastic process  $\widehat{\bf X}_t$, one can associate to $\widehat{\bf X}_t$ the transition function
$$
\widehat{F}(\Gamma,t\,|\,\widehat{\bf X}_{\tau},\tau)
:=\mathsf P\big(\widehat{\bf X}_t \in \Gamma\,|\,{\cal F}_{\tau}^{\widehat{\bf X}}\big),
$$
and its density with respect to the Lebesgue measure is:
\begin{equation}
\widehat{F}(\Gamma,t\,|\,{\bf y},\tau)=\int_\Gamma \widehat{f}({\bf x},t\,|\,{\bf y},\tau)\ {\rm d}{\bf x},
\label{1.1}
\end{equation}
where 
 $\widehat{f}({\bf x},t\,|\,{\bf y},\tau)$ is the transition p.d.f.\ of $\widehat {\bf X}_t$,
%
%
given $\widehat{\bf X}_{\tau}={\bf y}$. 
This function satisfies the Kolmogorov and Fokker--Planck equations, i.e. 
\begin{equation}
{\partial_\tau}\widehat{f}({\bf x},t\,|\,{\bf y},\tau)+\widehat{\cal L} \widehat{f}({\bf x},t\,|\,{\bf y},\tau)=0,
\qquad  
{\partial_t}\widehat{f}({\bf x},t\,|\,{\bf y},\tau)=\widehat{\cal L}^* \widehat{f}({\bf x},t\,|\,{\bf y},\tau),
\label{1.3}
\end{equation}	
%
%
%
respectively, and the following initial condition: 
\begin{equation}
\lim_{t\downarrow\tau}\widehat{f}({\bf x},t\,|\,{\bf y},\tau)=
\lim_{\tau\uparrow t}\widehat{f}({\bf x},t\,|\,{\bf y},\tau)=\prod_{i=1}^n
\delta(x_i-y_i),
\label{1.5}
\end{equation}
where $\delta$ is the Dirac-delta function.
\subsection{Plan of the paper}
This paper is organized as follows. 
Section \ref{sect:2} presents the main theoretical results about the product-form relationship between the transition p.d.f.'s of non-singular diffusion processes in $\mathbb{R}^n$ $(n>1)$
induced by the multidimensional Doob $h$-transform. 
Particular attention is devoted to the conditions on the auxiliary weight function $h$ that ensure the well-posedness of the transformed process. 
We also investigate general properties of the transformation, including stochastic ordering, applications to diffusion processes with Poissonian resetting, and a potential-theoretic representation.
\par
Section \ref{casestudies} illustrates the methodology through two case studies: 
Section \ref{sect:Wiener} and \ref{sect:OUtransf} are devoted to transformed 
Wiener and Ornstein--Uhlenbeck processes, respectively. 
For the Wiener case, the transformed transition p.d.f.\ is shown to admit a representation as 
a two-component Gaussian mixture. 
\par
Section \ref{boundaries} deals with a special case about the transformation of a symmetric two-dimensional process in the presence of an absorbing boundary. An expression of the taboo transition p.d.f.\ of the transformed process is obtained. 
Section \ref{sect:5} specializes the general theory to two-dimensional linear diffusion processes, providing an explicit expression for the weight function and discussing in detail the transformed Wiener and 
Ornstein--Uhlenbeck models in Sections \ref{wiener_bidim} and \ref{subs:tOU}, respectively. 
Finally, Section \ref{sect:6} concludes the paper with some remarks and directions for future research.
 
\section{Main results}\label{sect:2}
\subsection{The multidimensional Doob $h$-transform}
Consider a time-homogeneous diffusion process ${\bf X}_t=\big({\bf X}_t\big)_{t\geq 0}=({X}^1_t,{X}^2_t,\ldots,{X}^n_t)\in \mathbb{R}^n$, defined on the domain ${\cal D}$, 
with infinitesimal moments ${\bf b}({\bf x})$ and  $A({\bf x})$, 
and satisfying the SDE
\begin{equation} \label{sde_x_trasf}
    {\rm d} {\bf X}_t={\bf b}\big({\bf X}_t\big){\rm d}t+{\Sigma}\big( {\bf X}_t\big){\rm d}{\bf W}_t,
    \qquad {\bf X}_0={\bf x}_0,
\end{equation}
with ${\bf x}_0\in {\cal D}$. 
The dispersion matrix ${\Sigma}({\bf x})=\big( \sigma_{ij}({\bf x})\big)\in \mathbb{R}^{n \times n}$ 
is related to the diffusion matrix $A({\bf x})=\big( a_{ij}({\bf x})\big)\in \mathbb{R}^{n \times n}$ by 
\begin{equation}
  a_{ij}({\bf x}) := \sum_{k=1}^n {\sigma}_{ik} ({\bf x}) \,{\sigma}_{jk} ({\bf x}).
 	\label{eq:defaij}
\end{equation}
In the following theorem we provide conditions such that the transition p.d.f.\ of ${\bf X}_t$, say 
$f({\bf x},t\,|\,{\bf y},\tau)$, is related to the transition p.d.f.\ of $\widehat{\bf X}_t$, given in (\ref{1.1}), through the following product-form identity:
\begin{equation}
f({\bf x},t\,|\,{\bf y},\tau)={h({\bf x})\over h({\bf y})}\,
\widehat{f}({\bf x},t\,|\,{\bf y},\tau),
\label{2.1}
\end{equation}
for any ${\bf x},{\bf y}\in {\cal D}$ and  $t> \tau\geq 0$, 
where $h({\bf x})$, ${\bf x} \in {\cal D}$, is a suitable weight function.
\begin{theorem} \label{teo_1}
Let $h({\bf x})\in C^{2}(\mathcal{D})$ be a strictly positive function, and such that 
\begin{equation}
\widehat{\cal L}^* h({\bf x})\equiv
\sum_{i=1}^n \widehat{b}_i({\bf x})\partial_{x_i}h({\bf x})
+{1\over 2}\sum_{i,j=1}^n  \widehat{a}_{ij}({\bf x})
\partial_{x_i,x_j}h({\bf x})=0, \qquad \boldsymbol{x} \in \mathcal{D.}
\label{2.5}
\end{equation}
Moreover, let us assume that the function
\begin{equation} \label{cond_lipschitz_w}
    \Omega({\bf x}):=\frac{\nabla h({\bf x})\cdot  A({\bf x})}{h({\bf x})}, 
    \qquad {\bf x}\in \mathcal{D},
\end{equation}
possesses the property of global Lipschitz continuity. Hence, the time-homogeneous diffusion process ${\bf X}_t$, with infinitesimal moments 
\begin{equation}
\begin{array}{ll}
	b_i({\bf x})=\widehat{b}_i({\bf x})
    +\displaystyle{1\over h({\bf x})}\sum_{j=1}^n \widehat{a}_{ij}({\bf x})
	\ \partial_{x_j}h({\bf x}) & (i=1,\ldots,n), 
    \\
	a_{ij}({\bf x})=  \widehat{a}_{ij}({\bf x})  & (i,j=1,\ldots,n),
\end{array}
\label{2.5bc} 
\end{equation}
has the transition p.d.f.\ $f({\bf x},t\,|\,{\bf y},\tau)$ given in (\ref{2.1}), 
for $\bf{x},\,\bf{y}\in  {\cal D}$ and $t>\tau\geq  0$, 
which satisfies the initial delta condition 
\begin{equation}
\lim_{t\downarrow\tau}f({\bf x},t\,|\,{\bf y},\tau)=
\lim_{\tau\uparrow t}f({\bf x},t\,|\,{\bf y},\tau)=\prod_{i=1}^n
\delta(x_i-y_i),
\label{2.4}
\end{equation}
and the Kolmogorov and Fokker--Planck equations  
\begin{equation}
{\partial_\tau}{f}({\bf x},t\,|\,{\bf y},\tau)+{\cal L} {f}({\bf x},t\,|\,{\bf y},\tau)=0,
\qquad  
{\partial_t}{f}({\bf x},t\,|\,{\bf y},\tau)={\cal L}^* {f}({\bf x},t\,|\,{\bf y},\tau),
	\label{2.3}
\end{equation}
where ${\cal L}$ is the infinitesimal generator  associated to $\bf{X}_t$ and ${\cal L}^*$ its adjoint operator.
\end{theorem}
\begin{proof}
We first verify that the infinitesimal moments of the   diffusion process ${\bf X}_t$ satisfy the Lipschitz condition. 
Making use of Eq.\ \eqref{2.5bc}  we obtain, for some positive $L$, 
\begin{align}
    \big\lvert\big\lvert {\bf b}({\bf x})-{\bf b}({\bf y})\big\lvert\big\lvert
    +\big\lvert \big\lvert{\Sigma}({\bf x})-{\Sigma}({\bf y})\big\lvert\big\lvert
    &=
    \big\lvert\big\lvert \widehat{\bf b}({\bf x})+\frac{\nabla h({\bf x})\cdot A({\bf x})}{h({\bf x})} 
    -\widehat{\bf b}({\bf y})-\frac{\nabla h({\bf y})\cdot A({\bf y})}{h({\bf y})}\big\lvert\big\lvert
    +\big\lvert \big\lvert\widehat{\Sigma}({\bf x})-\widehat{\Sigma}({\bf y})\big\lvert\big\lvert 
    \nonumber\\
     &\leq \big\lvert\big\lvert \widehat{\bf b}({\bf x})-\widehat{\bf b}({\bf y})\big\lvert\big\lvert
     +\big\lvert \big\lvert\widehat{\Sigma}({\bf x})-\widehat{\Sigma}({\bf y})\big\lvert\big\lvert 
     + \left\lvert \left\lvert \frac{\nabla h({\bf x})\cdot A({\bf x})}{h({\bf x})}-\frac{\nabla h({\bf y})\cdot A({\bf y})}{h({\bf y})}\right\lvert \right\lvert 
     \nonumber \\ 
    &\leq L \lvert\lvert {\bf x}-{\bf y} \lvert \lvert+\left\lvert \left\lvert \frac{\nabla h({\bf x})\cdot A({\bf x})}{h({\bf x})}
    -\frac{\nabla h({\bf y})\cdot A({\bf y})}{h({\bf y})}\right\lvert \right\lvert,
    \nonumber
\end{align}
where the first inequality is given by the triangle inequality and the second follows immediately from Eq.\ \eqref{cond_lipschitz_widehatxt}. Hence, the claimed 
Lipschitz condition on the infinitesimal moments of ${\bf X}_t$ follows from the hypothesis of global Lipschitz continuity of  the function (\ref{cond_lipschitz_w}).
\par
Let us now prove that  $f({\bf x},t\,|\,{\bf y},\tau)$ verifies the equations in \eqref{2.3}.
By hypothesis, the transition p.d.f. $\widehat{f}({\bf x},t\,|\,{\bf y},\tau)$ is solution of Kolmogorov and 
Fokker--Planck equations (\ref{1.3}). 
By taking into account that (\ref{2.1}) holds, the Kolmogorov equation in \eqref{2.3} becomes
%
\begin{eqnarray*}
	&&\left\{{1\over 2}\sum_{i,j=1}^n a_{ij}({\bf y})
	\Bigl[{2\over h({\bf y})}\partial_{y_i}h({\bf y})
	\partial_{y_j} h({\bf y})-\partial_{y_i,y_j}h({\bf y})\Bigr]-\sum_{i=1}^n b_i({\bf y})\partial_{y_i}h({\bf y})\right\}\frac{\widehat{f}({\bf x},t\,|\,{\bf y},\tau)}{h({\bf y})}\\
	&&\qquad +\sum_{i=1}^n\Bigl[b_i({\bf y})-\widehat{b}_i({\bf y})-{1\over h({\bf y})}
	\sum_{j=1}^n a_{ij}({\bf y})\partial_{y_j}h({\bf y})\Bigr] \partial_{y_i}
    \widehat{f}({\bf x},t\,|\,{\bf y},\tau)\\ 
    && \qquad +{1\over 2}
	\sum_{i,j=1}^n \bigl[ a_{ij}({\bf y})- \widehat{a}_{ij}({\bf y})\bigr]\partial_{y_i,y_j}\widehat{f}({\bf x},t\,|\,{\bf y},\tau)=0,
\end{eqnarray*}
which is an identity due to Eqs.\ (\ref{2.5}) and \eqref{2.5bc}. 
Similarly, using (\ref{2.1}) in the Fokker-Planck equation of (\ref{2.3}) one has: 
\begin{eqnarray*}
	&&\bigg\{\sum_{i=1}^n\Bigl[h({\bf x})\partial_{x_i}\bigl(
	\widehat{b}_i({\bf x})-b_i({\bf x})\bigr)-b_i({\bf x})\partial_{x_i}h({\bf x})\Bigr]\\
	&& \qquad +{1\over 2}\sum_{i,j=1}^n\Bigr[a_{ij}({\bf x})\partial_{x_i,x_j}h({\bf x})+h({\bf x})\partial_{x_i,x_j}
	[a_{ij}({\bf x})- \widehat a_{ij}({\bf x})]
	+2\partial_{x_i} a_{ij}({\bf x})\partial_{x_j}h({\bf x})\Bigr]\bigg\}\widehat{f}({\bf x},t\,|\,{\bf y},\tau)\\
	&& \qquad +\Bigl\{h({\bf x})\sum_{i=1}^n\bigl(\widehat{b}_i({\bf x})-b_i({\bf x})\bigr)+
	\sum_{i,j=1}^n\Bigl[h({\bf x}){\partial\over\partial x_j}
	\bigl(a_{ij}({\bf x})-\widehat{a}_{ij}({\bf x})\bigr)
	+ a_{ij}({\bf x})\partial_{x_j}h({\bf x})\Bigr]\Bigr\}\\
	&& \qquad \times 
	\partial_{x_i}\widehat{f}({\bf x},t\,|\,{\bf y},\tau)+
	{1\over 2}h({\bf x})\sum_{i,j=1}^n\bigl[ a_{ij}({\bf x})-\widehat{a}_{ij}({\bf x})
	\bigr]\partial_{x_i,x_j}\widehat{f}({\bf x},t\,|\,{\bf y},\tau)=0,
\end{eqnarray*}
which is an identity due to relations (\ref{2.5}) and \eqref{2.5bc}. Finally, recalling (\ref{1.5}) 
and (\ref{2.1}), we see that the initial condition (\ref{2.4}) holds. 
\end{proof}
\begin{remark}
Due to Eq.\ (\ref{2.5}), the weight $h$ is an harmonic function for $\widehat{\bf X}_t$. 
It represents a conditional mean and expresses a martingale property, since Eq.\ \eqref{2.1} yields 
\begin{equation}
 h({\bf y})=\int_{\mathbb{R}^n}h({\bf x})\,
 \widehat{f}({\bf x},t\,|\,{\bf y},\tau)\,{\rm d}{\bf x}
 =\mathsf{E}\left[h\big(\widehat{{\bf X}}_t\big)\,\big|\,\widehat{{\bf X}}_{\tau}
    ={\bf y}\right],
    \qquad {\bf y}\in {\cal D},
    \label{eq:wasmean}
\end{equation}
for all $h$'s satisfying the conditions of Theorem \ref{teo_1} 
and $t>\tau\geq 0$. 
\end{remark}
%
\begin{corollary}\label{corol:HHgenerale}
For any  initial state ${\bf y}\in {\cal D}$, let us define the set 
\begin{equation}
    H_{\bf y}=\{{\bf x} \in {\cal D}:\;h({\bf x})=h({\bf y})\}.
\label{hyperplaneD}
\end{equation}
From Eq.\ (\ref{2.1}), for $\bf{y}\in  {\cal D}$ and $t>\tau\geq 0$, one has
\begin{equation*}
f({\bf x},t\,|\,{\bf y},\tau)=\widehat{f}({\bf x},t\,|\,{\bf y},\tau),\qquad \forall \,{\bf x} \in H_{\bf y}. 
\end{equation*}
In addition, by defining the regions  
\begin{equation}
{\cal H}^-_{\bf y}=\left\{ {\bf x} \in {\cal D}:\;h({\bf x})<h({\bf y}) \right\} \quad \text{and} 
\quad 
{\cal H}^+_{\bf y}=\left\{{\bf x} \in {\cal D}:\;h({\bf x})>h({\bf y})\right\},
\label{spacesD}
\end{equation}
we have
\begin{equation} %
f({\bf x},t\,|\,{\bf y},\tau)<    (>)\, \widehat{f}({\bf x},t\,|\,{\bf y},\tau),
\qquad \forall \, {\bf x} \in {\cal H}^-_{\bf y}\; ({\bf x} \in {\cal H}^+_{\bf y}), 
\label{f_rel_fhat}
\end{equation}
and thus  
\begin{equation}
\begin{array}{l}
  {F}({\cal H}^-_{\bf y},t\,|\,{\bf y},\tau)
 = \widehat{F}({\cal H}^-_{\bf y},t\,|\,{\bf y},\tau)- \Pi (t\,|\,{\bf y},\tau),
\\
 {F}({\cal H}^+_{\bf y},t\,|\,{\bf y},\tau)
 =\widehat{F}({\cal H}^+_{\bf y},t\,|\,{\bf y},\tau)+ \Pi (t\,|\,{\bf y},\tau),
\end{array}
 \label{eq:relazugH}
\end{equation}
where ${F}$ is defined similarly as 
$\widehat{F}$ in (\ref{1.1}), and  
$$
 \Pi(t\,|\,{\bf y},\tau):=\int_{{\cal H}^-_{\bf y}} 
 \left|\widehat{f}({\bf x},t\,|\,{\bf y},\tau)- {f}({\bf x},t\,|\,{\bf y},\tau) \right|
    \,{\rm d}{\bf x}
    =\int_{{\cal H}^+_{\bf y}} 
 \left|\widehat{f}({\bf x},t\,|\,{\bf y},\tau)- {f}({\bf x},t\,|\,{\bf y},\tau) \right|
    \,{\rm d}{\bf x}.
$$
\end{corollary}
\par
The result in (\ref{eq:relazugH}) shows that the mapping from $\widehat{\bf X}_t$ to ${\bf X}_t$, induced by the infinitesimal moments (\ref{2.5bc}), moves the probability mass $\Pi(t\,|\,{\bf y},\tau)$ from ${\cal H}^-_{\bf y}$ (associated with $\widehat{\bf X}_t$) to  ${\cal H}^+_{\bf y}$ (associated with ${\bf X}_t$). 
This is useful to construct an upper bound for the Wasserstein distance between the probability measures of the two processes. 
Specifically, for each fixed $t>\tau\geq 0$, we can consider two probability measures, namely $\widehat\mu$ and $\mu$, associated with $\widehat{\bf X}_{t}$ and ${\bf X}_{t}$. 
We also assume that $\widehat\mu,\mu\in\mathcal{P}_p(\mathbb{R}^n)$,  which is the set of probability measures over $\mathbb{R}^n$, equipped with the Euclidean norm, whose $p$-th moments are finite, for $p\geq 1$.  Then, with reference to Definition 1.2 and Theorem 1.6 of Chewi et al.\ \cite{chewi}, if the supports of $\widehat\mu$ and $\mu$ are both included in the same ball of diameter $d$, then the $p$-Wasserstein distance $W_p(\widehat\mu,\mu)$, $p\geq 1$, satisfies the inequality 
$$
 \left(W_p(\widehat\mu,\mu)\right)^p
\leq d^p \,\frac{1}{2}\, \Pi(t\,|\,{\bf y},\tau), 
\qquad t>\tau\geq 0, \;\; {\bf y}\in {\cal D}. 
$$
We recall that $\frac{1}{2}\,\Pi(t\,|\,{\bf y},\tau)$ also represents the total variation distance between $\widehat\mu$ and $\mu$ at time $t$. 
\par
Another application of Corollary \ref{corol:HHgenerale} can be given in the context of couplings. Indeed,
denoting by $(\widehat{\bf X}_t^{\rm c},{\bf X}_t^{\rm c})$ the maximal coupling of the pair of processes 
$(\widehat{\bf X}_t,{\bf X}_t)$, for all $t>\tau \geq 0$ and ${\bf y}\in \cal D$ one has (see, for instance, B\"ottcher \cite{Bottcher}) 
\begin{equation} \label{maximal_coupling_def}
\mathsf{P}\big[\widehat{\bf X}_t^{\rm c} ={\bf X}_t^{\rm c}\big]
=\int_{{\cal H}^-_{\bf y}} {f}({\bf x},t\,|\, {\bf y}, \tau) \ {\rm d}{\bf x}+\int_{H_{\bf y}\cup{\cal H}^+_{\bf y}} \widehat{f}({\bf x},t\,|\, {\bf y}, \tau) \ {\rm d}{\bf x},
\end{equation}
with $H_{\bf y}$, ${\cal H}^-_{\bf y}$ and ${\cal H}^+_{\bf y}$ defined in \eqref{hyperplaneD} and \eqref{spacesD}. This property can be further explicited when the structure of the weight function $h$ is known, as we will investigate in the following sections. 
\par
\begin{remark}\label{rem_indip}
If the process $\widehat{\bf X}_t$ 
has independent components, i.e.\ the infinitesimal moments satisfy 
$\widehat{b}_i({\bf x})\equiv \widehat{b}_i(x_i)$, 
$\widehat{a}_{ii}({\bf x})\equiv \widehat{a}_{ii}(x_i)$ 
and $\widehat{a}_{ij}({\bf x})=0$, for  $i,j=1,\ldots,n$; $i\neq j$, then 
a special solution of (\ref{2.5}) is given by: 
\begin{equation*}
h({\bf x})=1+c\,\prod_{i=1}^n\int_{u_i}^{x_i}\exp\Bigl\{-2\int_{u_i}^z
{\widehat{b}_i(v)\over  \widehat{a}_{ii}(v)}\,{\rm d}v\Bigr\}\,{\rm d}z,
\qquad {\bf x}\in \cal D,
\label{2.7}
\end{equation*}
for any  ${\bf u}\in{\cal D}$ and   $c\in \mathbb{R}$ such that $h({\bf x})>0$. In this case  the Lipschitz condition in \eqref{cond_lipschitz_w} still holds because of the conditions on the infinitesimal moments of $\widehat{\bf X}_t.$
\end{remark}
%
\par
Remark \ref{rem_indip} suggests looking for general expressions of the weight function of the form $h({\bf x})=1+c \,k({\bf x})$. Indeed, in the following theorem we prove that for similar forms of $h({\bf x})$ 
the transition p.d.f.\ of the process ${\bf X}_t$, even with non-independent components, can be expressed as a suitable proper mixture.
\begin{theorem} \label{prop_w}
Let $h({\bf x})=h_c({\bf x})\in C^{2}(\mathcal{D})$ be a function satisfying the conditions of Theorem \ref{teo_1}, with 
\begin{equation}
   h_c({\bf x}):=1+c \,k({\bf x}), 
   \qquad {\bf x}\in \mathcal{D},
   \label{w_c}
\end{equation}
where $k({\bf x})\in C^{2}(\mathcal{D})$ is a positive function, and $c \in (c_m,+\infty)$,  with $c_m=\displaystyle-\inf_{{\bf x} \in {\cal D}} \frac{1}{k({\bf x})}$.
Then, the transition p.d.f.\ given in (\ref{2.1}), for all   $t>\tau\geq  0$ and 
$\bf{x},\,\bf{y}\in  {\cal D}$ can be expressed as  
\begin{equation*} 
    f({\bf x}, t \,|\,{\bf y}, \tau)
    ={1+c \,k({\bf x})\over 1+c \,k({\bf y})}\,
    \widehat{f}({\bf x},t\,|\,{\bf y},\tau),
\end{equation*}
and it can be decomposed as the following mixture
\begin{equation}
    f({\bf x},t\,|\,{\bf y},\tau)
    =\theta_c({\bf y})\,
    \widehat{f}({\bf x},t\,|\,{\bf y},\tau)
    +(1-\theta_c({\bf y}))\,
    \widetilde{f}({\bf x},t\,|\,{\bf y},\tau),
    \label{rel_mixture}
\end{equation}
with  
 \begin{equation}
   \theta_c({\bf y})=\frac{1}{h_c({\bf y})},  
   \label{theta_gen}
 \end{equation}
and where 
  \begin{equation} \label{densita_fc}
      \widetilde{f}({\bf x}, t\,|\,{\bf y}, \tau)
      =\frac{k({\bf x})}{k({\bf y})}\,
      \widehat{f}({\bf x},t\,|\,{\bf y}, \tau)
       \end{equation}
 is the transition p.d.f.\ of the diffusion process, say 
 $\widetilde{\bf X}_t=\big(\widetilde{\bf X}_t\big)_{t\geq 0}$,  having the following infinitesimal moments:
   \begin{equation}
   \begin{array}{ll}
	\widetilde{b}_i({\bf x})
	=\widehat{b}_i({\bf x})+\displaystyle{1\over k({\bf x})}\sum_{j=1}^n \widehat{a}_{ij}({\bf x})
	\partial_{x_j}k({\bf x})\qquad &(i=1,\ldots,n),
    \\
	\widetilde{a}_{ij}({\bf x})=\widehat{a}_{ij}({\bf x}) 
	\qquad &(i,j=1,\ldots,n), 
\end{array}
\label{mominf_h} 
\end{equation}
with $\widehat{b}_i({\bf x})$ and $\widehat{a}_{ij}({\bf x})$   defined in  (\ref{inf_mom_X_t_hat}).
\end{theorem}
\begin{proof}
We refer to the non-trivial case when $c\neq 0$, i.e.\ $h_c({\bf x})\not \equiv 1$. Firstly, we observe that the density in \eqref{densita_fc} is always 
well posed. Indeed, $\widetilde{f}({\bf x}, t\,|\,{\bf y}, \tau)$ is positive due to   definition \eqref{densita_fc} and the positiveness of the function $k$.  Moreover, 
from (\ref{eq:wasmean}) and \eqref{w_c} we have 
\begin{eqnarray} 
\int_{{\cal D}}\widetilde{f}({\bf x},t\,|\,{\bf y},\tau)\ {\rm d}{\bf x}
\!\!\!\!  &=& \!\!\!\!   \int_{\cal D}\frac{k({\bf x})}{k({\bf y})}\, \widehat{f}({\bf x},t\,|\,{\bf y},\tau)\  {\rm d}{\bf x}
\nonumber\\
&=& \!\!\!\!   \frac{1}{h_c({\bf y})-1}\int_{{\cal D}}h_c({\bf x})\, \widehat{f}({\bf x},t\,|\,{\bf y},\tau)\ {\rm d}{\bf x}-\frac{1}{h_c({\bf y})-1}\int_{{\cal D}} \widehat{f}({\bf x},t\,|\,{\bf y},\tau)\ {\rm d}{\bf x}
\nonumber\\
&=&\!\!\!\!   \frac{h_c({\bf y})}{h_c({\bf y})-1}-\frac{1}{h_c({\bf y})-1}=1.
\nonumber
\end{eqnarray}
Eq.\ \eqref{w_c} ensures that 
the Lipschitz condition in \eqref{cond_lipschitz_w} is met for $k({\bf x})$.
Furthermore, $\widetilde{f}({\bf x},t\,|\,{\bf y},\tau)$ verifies Eqs.\ \eqref{2.3} for the infinitesimal moments given in \eqref{mominf_h} using the same strategy of Theorem \ref{teo_1} for $k({\bf x})$ instead of $h({\bf x})$. 
Finally, due to  relations (\ref{w_c}) and (\ref{theta_gen}), 
the right-hand-side  of (\ref{rel_mixture}) is expressed as:
\begin{eqnarray}
\theta_c({\bf y})\widehat{f}({\bf x},t\,|\,{\bf y},\tau)+(1-\theta_c({\bf y}))\frac{k({\bf x})}{k({\bf y})}\widehat{f}({\bf x},t\,|\,{\bf y},\tau)  
\!\!\!\! &=& \!\!\!\!\frac{1}{h_c({\bf y})}\widehat{f}({\bf x},t\,|\,{\bf y},\tau)+\frac{h_c({\bf y})-1}{h_c({\bf y})}\frac{h_c({\bf x})-1}{h_c({\bf y})-1}\widehat{f}({\bf x},t\,|\,{\bf y},\tau)
\nonumber\\
 &=& \!\!\!\!\frac{h_c({\bf x})}{h_c({\bf y})}\widehat{f}({\bf x},t\,|\,{\bf y},\tau)=f({\bf x},t\,|\,{\bf y},\tau),\nonumber
\end{eqnarray}
this completing the proof.
\end{proof}
%

\begin{remark}\label{rem:mixt}
Taking into account Theorem \ref{prop_w}, for the mixture on the right-hand-side of (\ref{rel_mixture}) one has:
\begin{itemize}
    \item[(i)] 
    the mixture is proper, since $\theta_c({\bf y})\in [0,1]$;
    \item[(ii)] $\displaystyle\lim_{c\to 0} f({\bf x},t\,|\,{\bf y},\tau)=\widehat{f}({\bf x},t\,|\,{\bf y},\tau),$ in this case $\theta_c({\bf y})=1$;
    \item[(iii)] if $c=\displaystyle\frac{1}{k({\bf y})}\in (c_m,+\infty)$, then $\theta_c({\bf y})=1/2$;
    \item[(iv)] $\displaystyle\lim_{c\to +\infty} f({\bf x},t\,|\,{\bf y},\tau)=\widetilde{f}({\bf x},t\,|\,{\bf y},\tau),$ in this case $\theta_c({\bf y})=0$. 
\end{itemize}
\end{remark}
\par
We also remark that a mixture representation analogous to \eqref{rel_mixture} was established by Giorno and Nobile \cite{Giorno2019} in the context of one-dimensional Gauss--Markov processes and time-inhomogeneous diffusion processes with transition p.d.f.'s connected by product-form representations.
\par
The remainder of this section is devoted to some features and applications of the previous results.  
\subsection{Comparisons based on the usual stochastic order}\label{sect:usord}
In the framework introduced in this section, it is interesting to analyze the stochastic comparison of the diffusion processes under investigation. In the following, the term `increasing' is used in non-strict sense. 
Aiming to compare the marginals, we say that $\widehat {\bf X}_t$ and ${\bf X}_t$ are ordered in the usual stochastic order, and write $\widehat{\bf X}_t\leq_{\rm st}{\bf X}_t$, if 
\begin{equation*}
 \mathsf{E}\big[\phi\big( \widehat{\bf X}_t\big) \big]
 \leq \mathsf{E}\big[\phi\big( {\bf X}_t  \big)\big], 
 \qquad t\geq 0,
\end{equation*}
for any increasing function $\phi$ for which the expectations are well defined (cf.\ Shaked and Shanthikumar \cite{Shaked}). 
Then, due to Theorem \ref{teo_1}, the validity of the product-form \eqref{2.1} yields a useful result in this field. To this aim, we recall that an $n$-dimensional random vector $\widehat {\bf X}_t$, $t\geq 0$, is said to be positively associated if ${\rm Cov}\big[\psi_1(\widehat {\bf X}_t),\psi_2(\widehat {\bf X}_t)\big]\geq 0$ for all increasing functions $\psi_1$, $\psi_2$ for which the covariance exists. Moreover, we say that $h({\bf x})$ is increasing in ${\cal D}$ if it is increasing componentwise, for ${\bf x}\in {\cal D}$. 
\par
Now we can provide sufficient conditions for the marginals of $\widehat{\bf X}_t$ and ${\bf X}_t$ to be ordered in the usual stochastic order. 
\begin{proposition} \label{prop_ord}
Under the assumptions of Theorem \ref{teo_1}, if $\widehat {\bf X}_t$ is positively associated for all $t\geq 0$, and if $h({\bf x})$ is increasing in ${\cal D}$, then $\widehat{\bf X}_t\leq_{\rm st} {\bf X}_t$ for all $t\geq 0$.
\end{proposition}
    \begin{proof} 
        Recalling Theorem 6.B.8 of \cite{Shaked}, $\widehat{\bf X}_t\leq_{\rm st}{\bf X}_t$ holds as long as $\widehat{\bf X}_t$ is associated and the p.d.f.'s ratio $f({\bf x},t\,|\, {\bf y}, \tau)/\widehat{f}({\bf x},t\,|\, {\bf y}, \tau)$ is increasing in ${\bf x}$, for any ${\bf y}\in {\cal D}$ and $t>\tau\geq 0$. Hence, the result follows immediately by recalling Eq.\ \eqref{2.1}.
    \end{proof}
\par
\subsection{An application to processes with Poissonian resetting}\label{subs:generalreset}
Let us analyze the diffusion processes under investigation in the presence of a resetting mechanism. Thanks to Theorem \ref{teo_1}, hereafter we show that a product-form relation similar to Eq.\ (\ref{2.1}) can also be obtained for the transition p.d.f.'s of the two diffusion processes considered so far when both processes are subject to resetting to a fixed state $\boldsymbol{y}^{\rm R}$ occurring at the events of two identically distributed Poisson processes. 
\par
We consider a diffusion process subject to an underlying resetting scheme, led by a time-nonhomogeneous Poisson process $N_t=\big({N}_t\big)_{t\geq 0}$, so that the process returns to a resetting state $\boldsymbol{y}^{\rm R}\in {\cal D}$ at any occurrence of $N_t$. 
Therefore, denoting by $\widehat{\bf X}_t^{\rm R}=\big(\widehat{\bf X}_t^{\rm R}\big)_{t\geq 0}$
the resulting diffusion process, it satisfies the SDE
\begin{equation} \label{sde_x_r}
    {\rm d} \widehat{\bf X}_t^{\rm R}
    =\widehat{{\bf b}}\big(\widehat{\bf X}_t^{\rm R}\big){\rm d}t
    +\widehat{{\Sigma}}\big( \widehat{{\bf X}}_t^{\rm R}\big){\rm d}{\bf W}_t
    +\left(\boldsymbol{y}^{\rm R}-\widehat{\bf X}_t^{\rm R}\right)\,{\rm d}N_t,
    \qquad \widehat{\bf X}_0^{\rm R}=\widehat{\bf x}_0,
\end{equation}
with $\widehat{\bf x}_0\in {\cal D}$, and $N_t$ independent of ${\bf W}_t$. 
Hence, in absence of resets $\widehat{\bf X}_t^{\rm R}$ behaves as the diffusion process $\widehat{\bf X}_t$ considered so far. 
In the following proposition, we show that a product-form relation similar to Eq.\ \eqref{2.1} also holds  for the transition p.d.f.'s of $\widehat{\bf X}_t^{\rm R}$ and ${\bf X}_t^{\rm R}=\left({\bf X}_t^{\rm R}\right)_{t \geq 0}$, the latter being a diffusion process obtained by superposing the same resetting mechanism considered above to the process ${\bf X}_t$. 
Specifically, the process ${\bf X}_t^{\rm R}$ satisfies the SDE 
\begin{equation} \label{sde_x_r_trasf}
     {\rm d} {\bf X}_t^{\rm R}
    ={\bf b}\big({\bf X}_t^{\rm R}\big){\rm d}t
    +{\Sigma}\big({\bf X}_t^{\rm R}\big){\rm d}{\bf W}_t
    +\left(\boldsymbol{y}^{\rm R}-{\bf X}_t^{\rm R}\right)\,{\rm d}N_t',
    \qquad {\bf X}_0^{\rm R}={\bf x}_0,
\end{equation}
where ${\bf x}_0\in {\cal D}$ and $\boldsymbol{y}^{\rm R}\in {\cal D}$ are identical to the same terms as given in (\ref{sde_x_r}), and $N_t'$ is an independent copy of $N_t$. 
\par
The considered Doob $h$-transform leads to a 
product-form identity, similar to (\ref{2.1}), between the transition p.d.f.'s of the processes ${\bf X}_t^{\rm R}$ and $\widehat{\bf X}_t^{\rm R}$, valid only for   transitions starting from the reset state $\boldsymbol{y}^{\rm R}$. 
\begin{proposition}\label{prop_reset}
Let $\widehat{\bf X}_t^{\rm R}$ and ${\bf X}_t^{\rm R}$ be $n$-dimensional diffusion processes satisfying the SDE's \eqref{sde_x_r} and \eqref{sde_x_r_trasf}, respectively. Both processes reset to an arbitrary state $\boldsymbol{y}^{\rm R}\in {\cal D}$ at any occurrence of the independent nonhomogeneous Poisson processes $N_t$ and $N_t'$, both having intensity $\lambda(t)$, and cumulative intensity 
$\Lambda(t)=\int_{0}^t\lambda(s)\, {\rm d}s$. If the conditions of Theorem \ref{teo_1} are satisfied, then the transition p.d.f.\ of ${\bf X}_t^{\rm R}$ is expressed in terms of the transition p.d.f.\ of $\widehat{\bf X}_t^{\rm R}$ via the following product-form relation: 
\begin{equation} \label{rel_f}
    f^{\rm R}({\bf x},t\,|\,{\bf y}^{\rm R},\tau)={h({\bf x})\over h({\bf y}^{\rm R})}\,
\widehat{f}^{\rm R}({\bf x},t\,|\,{\bf y}^{\rm R},\tau),
\end{equation}
for any $\bf{x},\bf{y}\in  {\cal D}$ and $t>\tau\geq  0$. 
\end{proposition}
\begin{proof}
Recalling the reasoning brought on for one-dimensional processes under resetting (see, for instance, Giorno et al.\ \cite{GiornoNobileSpina} and references therein), by the Markov property of the involved processes, the transition p.d.f.\ of the process introduced in \eqref{sde_x_r} can be expressed as 
\begin{equation} \label{densita_f_r}
   \widehat{f}^{\rm R}(\boldsymbol{x},t\,|\,\boldsymbol{y}^{\rm R},\tau)
=e^{-\left[\Lambda(t)-\Lambda(\tau)\right]}\;\widehat{f}(\boldsymbol{x},t\,|\,\boldsymbol{y}^{\rm R},\tau)
+\int_{\tau}^t \lambda(s)\,e^{-\left[\Lambda(t)-\Lambda(s)\right]}\,
\widehat{f}(\boldsymbol{x},t\,|\,\boldsymbol{y}^{\rm R},s)\, {\rm d} s, 
\end{equation}
where $s$ is the last instant before $t$ in which a reset occurs. 
Making use of \eqref{2.1} and choosing the weight function $h$ as specified in Theorem \ref{teo_1}, a simple algebraic manipulation leads to
$$ 
 \frac{h(\boldsymbol{x})}{h(\boldsymbol{y}^{\rm R})}\,
 \widehat{f}^{\rm R}(\boldsymbol{x},t\,|\,\boldsymbol{y}^{\rm R},\tau)
 =e^{-\left[\Lambda(t)-\Lambda(\tau)\right]}\,
 f(\boldsymbol{x},t\,|\,\boldsymbol{y}^{\rm R},\tau)
 +\int_{\tau}^t \lambda(s)\,
 e^{\left[\Lambda(t)-\Lambda(s)\right]}\,
 f(\boldsymbol{x},t\,|\,\boldsymbol{y}^{\rm R},s)\, {\rm d} s,
$$
justifying the result in \eqref{rel_f}, since $f^{\rm R}$ satisfies an equation analogous to \eqref{densita_f_r}.
\end{proof}
Let us now consider a special case, in which the processes introduced in \eqref{sde_x_r} and 
 (\ref{sde_x_r_trasf}) are subject to a resetting scheme led by homogeneous Poisson processes. In this case we define the stationary distributions of $\widehat{{\bf X}}_t^{\rm R}$ and ${\bf X}_t^{\rm R}$ as
\begin{equation} \label{dens_stazionaria}
    \widehat{f}^{\rm R}({\bf x}\,|\,{\bf y}^{\rm R}):=\lim_{t\to \infty}  \widehat{f}^{\rm R}({\bf x},t\,|\,{\bf y}^{\rm R},\tau), \qquad f^{\rm R}({\bf x}\,|\,{\bf y}^{\rm R}):=\lim_{t\to \infty}  f^{\rm R}({\bf x},t\,|\,{\bf y}^{\rm R},\tau),
\end{equation}
respectively. By proceeding similarly as in Proposition \ref{prop_reset}, the following result yields the stationary distribution of the transformed process ${\bf X}_t^{\rm R}$ governed by Eq.\ \eqref{sde_x_r_trasf}. 
\begin{proposition} \label{prop_reset_2}
    Let $\widehat{\bf X}_t^{\rm R}$ and ${\bf X}_t^{\rm R}$ be the $n$-dimensional diffusion processes satisfying the assumptions of Proposition \ref{prop_reset}, so their transition p.d.f.'s satisfy the relation in \eqref{rel_f}. If the intensity of the resetting processes $N_t$ and $N_t'$ is a positive constant $\lambda$, then the stationary distribution of ${\bf X}_t^{\rm R}$ is expressed in terms
    of the stationary distribution of $\widehat{{\bf X}}_t^{\rm R}$, given in \eqref{dens_stazionaria}, via the following product-form relation:
    \begin{equation} \label{dens_stazionaria_trasf}
        f^{\rm R}({\bf x}\,|\,{\bf y}^{\rm R})=\frac{h({\bf x})}{h({\bf y}^{\rm R})} \widehat{f}^{\rm R}({\bf x}\,|\,{\bf y}^{\rm R}).
    \end{equation}
\end{proposition}
\begin{proof}
Recalling Eqs.\ \eqref{densita_f_r} and \eqref{dens_stazionaria}, one has that the stationary distribution of $\widehat{{\bf X}}_t^{\rm R}$ when $\lambda(t)=\lambda>0$, and thus $\Lambda(t)=\lambda t$, is given by
   $$
        \widehat{f}^{\rm R}({\bf x}\,|\,{\bf y}^{\rm R})=\int_0^\infty\lambda \, e^{-\lambda\theta}\, \widehat{f}({\bf x},\theta\,|\,{\bf y}^{\rm R})\, {\rm d}\theta.
   $$
   The proof of (\ref{dens_stazionaria_trasf}) then follows immediately by making use of \eqref{2.1} and choosing the weight function $h$ as specified in Theorem \ref{teo_1}.
\end{proof}
\par
A special case related to Proposition \ref{prop_reset_2} will be studied in Section \ref{sect:Wiener}. 
\subsection{Diffusion in a potential}
In the following sections, we investigate some applications of the previous results to particular choices of the diffusion process $\widehat{\bf X}_t$ concerning the Wiener process, the OU process and a linear process. For such cases, it is also useful to adopt a description in which the transformed diffusion process is viewed as moving in a potential $U({\bf x})$ in $C_1({\cal D})$. Hence, when the 2nd-order infinitesimal moments of ${\bf X}_t$ are constant, that is, when 
${\Sigma}( {\bf x})\equiv{\Sigma}$, the SDE (\ref{sde_x_trasf}) for ${\bf X}_t$ can be expressed as 
\begin{equation} \label{NEWsde_x_trasf}
    {\rm d} {\bf X}_t=-  \nabla U\big({\bf X}_t\big){\rm d}t+{\Sigma} \,{\rm d}{\bf W}_t,\qquad {\bf X}_0={\bf y}.
\end{equation}
This representation based on the potential $U$ is useful for analyzing dual diffusions (i.e., diffusions driven by opposite force fields) and related passage problems, see Comtet et al.\ \cite{Comtet_etal}. 
Moreover, the potential U has also been employed in one-dimensional settings, when featuring two (or more) minima, to model bimodal (or multimodal) diffusions (cf.\ Forman and S\o rensen \cite{Forman}, Hongler et al.\ \cite{Hongler}, and references therein), as well as in monotone cases, to represent dynamics influenced by attractive natural endpoints (see Di Crescenzo et al.\ \cite{Di_Crescenzo_Giorno_Nobile}). In the multidimensional framework, in Li et al.\ \cite{Lietal} the authors study the relationship between the local minima of a multidimensional potential and the distributions of coupling times for two stochastic processes, i.e.\ overdamped Langevin dynamics. Moreover, Lavenant et al.\ \cite{Lavenant} analyze  
a multidimensional model in which the single cell RNA-sequencing dynamics, described through a diffusion process, are subjected to a potential landscape. 
%
\section{Case studies for transformed $n$-dimensional processes} \label{casestudies}
\subsection{Transformed Wiener process}\label{sect:Wiener}
Let $\widehat{\bf X}_t$ be the $n$-dimensional Wiener process 
defined in ${\cal D}\equiv{\mathbb R}^n$, with infinitesimal moments 
%
\begin{equation} \label{mom_infinitesimali_y}
\begin{array}{ll}
 \widehat{b}_i({\bf x})=m_i & (i=1,\ldots,n), \\
 \widehat{a}_{ij}({\bf x})= a_{ij} & (i,j=1,\ldots,n),
\end{array}    
\end{equation}
where $m_i\in\mathbb{R}\;(i=1,\ldots,n)$, $a_{ij}\in\mathbb{R}$ $(i,j=1,\ldots,n\,;i\not= j)$, 
$a_{ii}>0$ $(i=1,\ldots,n)$, with $\Delta:=|A|>0$, since, recalling Eq.\ (\ref{eq:defaij}), 
the matrix $A=(a_{ij})$ is symmetric and positive definite. 
As well known, the transition p.d.f.\ of $\widehat{\bf X}_t$ for all $t>\tau\geq 0$ and 
${\bf x},\,{\bf y}\in{\mathbb R}^n$ is given by 
\begin{equation}
\widehat{f}({\bf x},t\,|\,{\bf y},\tau)
 = {\cal N} \Big({\bf y}+(t-\tau)\,{\bf m}, (t-\tau)\,A \Big),
	\label{3.2}
\end{equation}
with ${\bf m}=(m_1,m_2,\ldots,m_n)$, and where ${\cal N} ( {\bf z},  \Upsilon )$ denotes an $n$-dimensional Gaussian density having mean ${\bf z}$ and covariance matrix $\Upsilon $. 
%
\par
If not all the components of ${\bf m}$ are vanishing, then we consider a family of positive functions $h_c({\bf x})\in C^{2}(\mathbb{R}^n)$ defined as 
\begin{equation}
h_c({\bf x})=1+c\exp\bigl\{-2\,{\bf r} \cdot {\bf x}^\top\bigr\}, \qquad c\in (0,+\infty).
\label{3.3generica}
\end{equation}
Here, in order that condition (\ref{2.5}) in Theorem \ref{teo_1} holds, it is required that the vector ${\bf r}\equiv(r_1,r_2,\ldots,r_n)\in \mathbb{R}^n$ satisfies the following relation: 
\begin{equation}
	{\bf r}\cdot A \cdot {\bf r}^\top-{\bf r} \cdot{\bf m}^\top=0.
	\label{3.4generica}
\end{equation}
Since $A$ is symmetric and positive definite,  the quadratic form 
$q({\bf r})={\bf r} \cdot A\cdot {\bf r}^\top$ is positive for all choices of ${\bf r}$. Thus, 
assumption ${\bf m}\neq {\bf 0}$  ensures that Eq.\ (\ref{3.4generica}) has non-trivial solutions. 
Hence, being ${\bf r}\neq {\bf 0}$ we have that $h_c({\bf x})$ is not a constant. 
\par
Moreover, recalling Feller's theory of boundary classification, the endpoints $-\infty$ and $+\infty$ are both natural and unattainable, for any component $\widehat{X}^k_t$, $k=1,2,\ldots, n$, where $-\infty$ is attractive (nonattractive) if $m_k<0$ ($m_k\geq 0$), whereas 
$+\infty$ is nonattractive (attractive) if $m_k\leq 0$ ($m_k> 0$).
\par
We note that the case considered in Eq.\ (\ref{3.3generica}) is quite different from the special solution proposed in Remark \ref{rem_indip}, even when the process $\widehat{\bf X}_t$ has independent components.
\begin{remark}
Without loss of generality, the function in (\ref{3.3generica}) 
is taken such that $h_0({\bf x})= 1$, ${\bf x}\in \mathbb{R}^n$. 
Moreover, the Lipschitz condition given in Eq.\ \eqref{cond_lipschitz_w} holds.
Specifically, from (\ref{3.3generica}) we can see that 
$$
 \Omega({\bf x})=-\frac{2 c\ {\bf r}\cdot A}{\exp\{2{\bf r}\cdot {\bf x}^\top\}+c},
 \qquad \bf{x} \in\mathbb{R}^n,
$$
satisfies the mentioned condition for $c\in (0,+\infty)$. 
We need to prove that there exists a positive quantity $L_1$ such that 
$\lvert\lvert\Omega({\bf x})-\Omega({\bf y})\lvert\lvert
<L_1 \lvert\lvert{\bf x}-{\bf y}\lvert\lvert$ for all ${\bf x},{\bf y}\in\mathbb{R}^n$, and hence that
$$
|2c|\;\lvert \lvert {\bf r}\cdot A \lvert\lvert \;\left\lvert \frac{1}{\exp\{2{\bf r}\cdot {\bf x}^\top\}+c}-\frac{1}{\exp\{2{\bf r}\cdot {\bf y}^\top\}+c} \right\lvert<L_1 \lvert\lvert{\bf x}-{\bf y}\lvert\lvert.
$$
This is equivalent to show that the gradient
$$
 \nabla\left(\displaystyle \frac{1}{\exp\{2{\bf r}\cdot {\bf x}^\top\}+c}\right)
 = \frac{\exp\{{\bf r}\cdot{\bf x}^\top\}\ {\bf r}}{(\exp\{2{\bf r}\cdot {\bf x}^\top\}+c)^2}
 $$
is bounded. Taking into account that
$$\left\lvert\frac{\exp\{{\bf r}\cdot{\bf x}^\top\}}{(\exp\{2{\bf r}\cdot {\bf x}^\top\}+c)^2}  \right\lvert=\left\lvert \frac{1}{\exp\{2{\bf r}\cdot {\bf x}^\top\}+c}-\frac{c}{(\exp\{2{\bf r}\cdot {\bf x}^\top\}+c)^2}\right\lvert<\frac{2}c,$$
the result thus follows.
\end{remark}
\par
We recall that the case ${\bf m}={\bf 0}$ has been excluded. 
Hence, for $h_c({\bf x})$ given in (\ref{3.3generica}) and ${\bf r}$ satisfying the condition (\ref{3.4generica}), with ${\bf m}\not ={\bf 0}$ and $c\in (0,+\infty)$, due to Eq.\ (\ref{2.5bc}) the infinitesimal moments of the transformed diffusion process ${\bf X}_t$, obtained from the Wiener process with infinitesimal moments (\ref{mom_infinitesimali_y}), are given by
\begin{equation}
\begin{array} {ll}
    b_i({\bf x})=m_i-\displaystyle\frac{2c\, {\bf r}\cdot  a_{*,i}}
		{\exp\bigl\{2\,{\bf r} \cdot{\bf x}^\top\bigr\}+c} \quad
	&(i=1,\ldots,n),
 \\
	 a_{ij}({\bf x})=\widehat{a}_{ij} &(i,j=1,\ldots,n),
\end{array}
\label{mom_infinitesimali_fx}
\end{equation}
where  $ a_{*,i}$ is defined as in (\ref{eq:Ai*A*j}). 
Note that $a_{*,i}= a_{i,*}$ for all $i=1,\ldots,n$ since $A$ is symmetric. 
Further, recalling (\ref{3.2}), using Theorem \ref{teo_1} we obtain the transition p.d.f.\ of ${\bf X}_t$ for the weight function \eqref{3.3generica}.
\begin{remark}
Due to Eqs.\  \eqref{mom_infinitesimali_y} and \eqref{mom_infinitesimali_fx}, 
a bound for the Euclidean distance between the drift vectors of the 
processes $\widehat{{\bf X}}_t$ and ${\bf X}_t$ is
    $$
    \lvert\lvert \widehat{{\bf b}}
    ({\bf x})-{\bf b}
    ({\bf x})\rvert \rvert 
    =2\ \frac{c}{\exp\bigl\{2\,{\bf r} \cdot{\bf x}^\top\bigr\}+c} \ \lvert\lvert {\bf r}\cdot A\rvert \rvert
    <2\ \lvert\lvert {\bf r}\cdot A\rvert \rvert,
    $$
and is independent on $c$.  
\end{remark}
%
\begin{remark} \label{remark_coupling}
When choosing the weight function $h_c({\bf x})$ as in (\ref{3.3generica}), 
the condition given in (\ref{hyperplaneD}) leads to the hyperplane $H_{{\bf y}}=\{{\bf x}\in \mathbb{R}^n :\;{\bf r} \cdot {\bf x}^\top={\bf r} \cdot {\bf y}^\top\}$. 
Clearly, for any fixed initial point ${\bf y}\in\mathbb{R}^n$, for ${\bf r}\neq {\bf 0}$ 
verifying condition (\ref{3.4generica}), the hyperplane $H_{{\bf y}}$ passes through the point ${\bf y}$, so that 
${\bf b}({\bf x})={\bf b}({\bf y})$ $ \forall\, {\bf x}\in H_{{\bf y}}$, due to (\ref{mom_infinitesimali_fx}). 
Moreover, the half-spaces ${\cal H}^-_{{\bf y}}$ and ${\cal H}^+_{{\bf y}}$, introduced in (\ref{spacesD}), are 
defined through the 
conditions ${\bf r} \cdot {\bf x}^\top<{\bf r} \cdot {\bf y}^\top$ and 
${\bf r} \cdot {\bf x}^\top>{\bf r} \cdot {\bf y}^\top$, respectively. 
Hence, in this case the probability of the maximal coupling given in Eq.\ \eqref{maximal_coupling_def}, 
for fixed $t>\tau\geq 0$ and ${\bf y}\in \mathbb{R}^n$ can be expressed as 
\begin{equation} \label{maximal_coupling_w}
\mathsf{P}\big[\widehat{\bf X}_t^{\rm c} ={\bf X}_t^{\rm c}\big]
=\int_{\{{\bf r} \cdot {\bf x}^\top<{\bf r} \cdot {\bf y}^\top\}} {f}({\bf x},t\,|\, {\bf y}, \tau) \ {\rm d}{\bf x}+\int_{\{{\bf r} \cdot {\bf x}^\top\geq{\bf r} \cdot {\bf y}^\top\}} \widehat{f}({\bf x},t\,|\, {\bf y}, \tau) \ {\rm d}{\bf x}.
\end{equation}
\end{remark}
\begin{remark}\label{remark_picchi}
(i) 
The transition p.d.f.\ \eqref{2.1} of the process ${\bf X}_t$, obtained from the $n$-dimensional Wiener process $\widehat{\bf X}_t$ by making use of function $h_c$ as in \eqref{3.3generica}, with ${\bf r}$ satisfying 
assumption \eqref{3.4generica}, can be written as the mixture in (\ref{rel_mixture}) due to Theorem \ref{prop_w}, so that 
    \begin{equation}\label{h_teta_y}
      k({\bf x})= \exp\{-2\ {\bf r}\cdot{\bf x}^\top \}, 
      \qquad 
      \theta_c({\bf y})=\frac{1}{h_c({\bf y})}
      =\frac{1}{1+c\exp\{-2\ {\bf r}\cdot {\bf y}^\top\}},
    \end{equation}
for $c\in (0,+\infty)$, since 
$c_m=-\inf_{{\bf x}\in \mathbb{R}^n} \frac{1}{k({\bf x})}=0$.
The transition p.d.f.\ $\widehat{f}({\bf x},t\,|\,{\bf y},\tau)$ is given in \eqref{3.2}, 
whereas $\widetilde{f}({\bf x}, t\,|\,{\bf y}, \tau)$ is the transition p.d.f.\ of a Wiener process $\widetilde{\bf X}_t$ in $\mathbb{R}^n$, having infinitesimal moments (cf.\ Eq.\ \eqref{mominf_h})
\begin{equation*}
   \widetilde{b}_i({\bf x})
   =m_i-2{\bf r}\cdot  a_{*i},
   \qquad 
   \widetilde{a}_{ij}=a_{ij},
   \qquad i,j=1,\ldots,n.
\end{equation*}
(ii) 
Since $\widehat{f}({\bf x},t\,|\,{\bf y},\tau)$ and $\widetilde{f}({\bf x}, t\,|\,{\bf y}, \tau)$ are both Gaussian densities, they have an unique maximum point in 
\begin{equation*}
 {\bf \widehat{x}}_{M}={\bf y}+(t-\tau)\,{\bf m},
\qquad
\widetilde{{\bf x}}_{M}={\bf y}+(t-\tau)({\bf m}-2\,{\bf r} \cdot A)
=\widehat{{\bf x}}_{M}-(t-\tau)\,2\,{\bf r} \cdot A,
\end{equation*}
respectively, and thus 
$\widehat{f}({\bf \widehat{x}}_{M},t\,|\,{\bf y},\tau)=\widetilde{f}(\widetilde{{\bf x}}_{M}, t\,|\,{\bf y}, \tau)
=  \bigl[2\pi(t-\tau)\bigr]^{-n/2} \Delta^{-1/2}$.    
It follows that the transition p.d.f.\ $f({\bf x},t\,|\,{\bf y},\tau)$, due to (\ref{2.1}), can have two maxima, at most. The signs of ${\bf x}-{\bf y}$ and ${\bf m}$ play an essential role in the determination of the mode(s), which can be accomplished via numerical methods due to the transcendency of the relevant equations. 
Also, numerical evaluations may show that for small (large) $t-\tau$, the transition p.d.f.\ of ${\bf X}_t$ exhibits one (two) mode(s), in general. A similar behavior is exhibited by transformed Skellam processes (See Di Crescenzo and Martinucci \cite{Di_Crescenzo2009}). 
\\
(iii) Recalling (\ref{rel_mixture}) and case (iii) of Remark \ref{rem:mixt}, the transition p.d.f.\ $f({\bf x},t\,|\,{\bf y},\tau)$ is 
expressed as a mixture with equal weights if and only if $\theta_c({\bf y})=1/2$, i.e.\ for 
\begin{equation}
 c=\exp\{2\ {\bf r}\cdot{\bf y}^\top \},
\label{eq:condsuc}
\end{equation}
due to \eqref{3.3generica}.   
The values of $f({\bf x},t\,|\,{\bf y},\tau)$ in the maximum points of $\widehat{f}$ and $\widetilde{f}$ are given respectively by 
$$
    f({\bf \widehat{x}}_M,t\,|\,{\bf y},\tau)
    =\frac{\theta_c({\bf y})\,e^{-2 (t-\tau){\bf r}\cdot{\bf m}^\top}+(1-\theta_c({\bf y}))}{{\bigl[2\pi(t-\tau)\bigr]^{n/2}
		\Delta^{1/2}}},
        \qquad
    f(\widetilde{{\bf x}}_M,t\,|\,{\bf y},\tau)
    = \frac{\theta_c({\bf y})+(1-\theta_c({\bf y}))\,e^{-2 (t-\tau){\bf r}\cdot{\bf m}^\top}}{{\bigl[2\pi(t-\tau)\bigr]^{n/2}
		\Delta^{1/2}}}, 
$$
and clearly they are identical when condition (\ref{eq:condsuc}) holds. 
Moreover, if condition (\ref{eq:condsuc}) is fulfilled, from Eq.\ (\ref{mom_infinitesimali_fx}) 
one obtains that the drift vector of ${\bf X}_t$ evaluated at the initial state ${\bf y}$ is constant, i.e.  
\begin{equation}
    {\bf b}({\bf y})={\bf m}- {\bf r}\cdot A.
    \label{eq:driftWc1/2}
\end{equation}
\end{remark}
\begin{remark}\label{rem:r2choices}
Further suitable choices of ${\bf r}$ that satisfy the condition  in Eq\. (\ref{3.4generica}) are given by:
\\
(i) ${\bf r}={\bf m} \cdot A^{-1}$,
\\
(ii) ${\bf r}= \frac{m_i}{a_{ii}}\cdot {\bf e}_i$ 
for any $i=1,\ldots,n$, where ${\bf e}_i$ is the unit vector with respect to the $i$-th coordinate.
\par
These cases will be treated in more detail in Section \ref{wiener_bidim} in the two-dimensional case. 
Moreover, if $c$ is taken as in (\ref{eq:condsuc}) then from Eq.\ (\ref{eq:driftWc1/2}) we have:
\\
(i) if ${\bf r}={\bf m} \cdot A^{-1}$, then ${\bf b}({\bf y})={\bf 0}$; 
\\
(ii) if ${\bf r}= \frac{m_i}{ a_{ii}}\cdot {\bf e}_i$ 
with $i=1,\ldots,n$, then $b_i({\bf y})=0$ and 
$b_j({\bf y})= m_j- {m_i}\,\frac{ a_{ij}}{a_{ii}}$ for $j\neq i$. 
\end{remark}
\par
As done in Section \ref{sect:usord}, let us now present a result on the usual stochastic ordering between $\widehat{\bf X}_t$ and ${\bf X}_t$ for the marginals of the diffusion processes considered in this section. 
Hereafter, the inequality between vectors is taken component-wise. 
\begin{proposition}\label{prop_W_ord}
Consider the $n$-dimensional diffusion processes whose infinitesimal moments are described in Eqs.\ \eqref{mom_infinitesimali_y} and \eqref{mom_infinitesimali_fx}. 
If $a_{ij}\geq 0$ for all $1\leq i,j\leq n$ and ${\bf r} \cdot {\bf x}^\top$ is decreasing for ${\bf x}\in \mathbb{R}^n$, then $\widehat{\bf X}_t\leq_{\rm st}{\bf X}_t$ for all $t\geq 0$.
\end{proposition}
\begin{proof}
From the theorem given in Pitt \cite{Pitt}, $\widehat{\bf X}_t$ is positively associated for all $t\geq 0$ thanks to the assumption that $\sigma_{ij}\geq 0$ 
for all $1\leq i,j\leq n$. Moreover, it follows from Eq.\ (\ref{3.3generica}) that $h_c({\bf x})$ is increasing for ${\bf x}\in \mathbb{R}^n$ if and only if ${\bf r} \cdot {\bf x}^\top$ is decreasing for ${\bf x}\in \mathbb{R}^n$. 
The thesis thus follows from Proposition \ref{prop_ord}.  
\end{proof}
\par
Let us now analyze an application of the results given in Section \ref{subs:generalreset}. We assume that an $n$-dimensional Wiener process $\widehat{\bf X}^{\rm R}_t$ is subject to resets to a known state ${\bf y}^{\rm R}\in \mathcal{D}$ modeled after the occurrences of a time-homogeneous Poisson process $N_t$. Indeed, the corresponding transition p.d.f. is easily obtained by taking Eq.\ \eqref{densita_f_r}, choosing $\Lambda(t)=\lambda t$ with $\lambda>0$ and making use of \eqref{3.2}. Moreover, as shown in Proposition \ref{prop_reset}, one can obtain a transformed resetting process ${\bf X}^{\rm R}_t$ by means of any weight function $h_c({\bf x})$ in \eqref{w_c}, whose p.d.f.\ can be recovered by Eq.\ \eqref{rel_f}.
The following result yields a closed expression for the stationary p.d.f.\  $f({\bf x}\,|\,{\bf y}^R)$ of ${\bf X}^{\rm R}_t$, whose general form is given in \eqref{dens_stazionaria_trasf}. 
\begin{proposition}\label{propWienerreset}
     Let $\widehat{\bf X}_t^{\rm R}$ and ${\bf X}_t^{\rm R}$ be the $n$-dimensional diffusion processes introduced in Proposition \ref{prop_reset}, whose transition p.d.f.'s satisfy the relation in \eqref{rel_f}. If the intensity of the resetting processes $N_t$ and $N_t'$ is a positive constant $\lambda$, then the stationary p.d.f.\ of ${\bf X}_t^{\rm R}$ is given by
     \begin{align}
        f({\bf x}\,|\,{\bf y}^R)
        & =\displaystyle \frac{h_c({\bf x})}{h_c({\bf y}^{\rm R})}\,
        \frac{2\lambda e^{\Phi\left({\bf x,y}^{\rm R}\right)}}{(2\pi)^{n/2} {\Delta^{1/2}}}  
        \left(\frac{\Xi\left({\bf x,y}^{\rm R}\right)}
        {\lambda+\frac{1}{2}{\bf m}\cdot A^{-1}\cdot{\bf m}^{\top}}\right)^{\frac{2-n}{4}}
        \nonumber
        \\
        & \times\,  K_{1-\frac{n}{2}}\left(2
        \left(\Xi\left({\bf x,y}^{\rm R}\right) 
        \left(\lambda+\frac{1}{2}{\bf m}\cdot A^{-1}\cdot{\bf m}^{\top}\right)\right)^{1/2}\right),
     \label{staz_wiener}
\end{align}
     where the weight function $h_c$ is explicited in \eqref{3.3generica}, 
     $\Delta=| A|$, ${A^{-1}}:=(\gamma_{ij})$,
     \begin{align*}
         \Phi\left({\bf x,y}^{\rm R}\right)
         &=\sum_{i=1}^n\gamma_{ii}\; m_i(x_i-y_i^{\rm R})
         +\sum_{i,j=1}^n\gamma_{ij}\left[ m_j (x_i-y^{\rm R}_i)+ m_i (x_j-y^{\rm R}_j)\right], \\ 
         \Xi\left({\bf x,y}^{\rm R}\right)
         &=\frac{1}{2} \, ({\bf x}-{\bf y}^{\rm R})\cdot A^{-1}  \cdot({\bf x}-{\bf y}^{\rm R})^{\top},
     \end{align*}
     and $K_{h}$ denotes the  modified second-kind Bessel function of index $h$ 
     (see, for instance, Gardiner \cite{Gardiner}). 
\end{proposition}
\begin{proof}
Making use of Eqs.\ \eqref{dens_stazionaria} and \eqref{3.2}, and recalling the relation among the stationary p.d.f.'s of $\widehat{\bf X}_t^{\rm R}$ and ${\bf X}_t^{\rm R}$ highlighted in Proposition \ref{prop_reset_2}, one has  %
 \begin{align*}
     f({\bf x}\,|\,{\bf y}^{\rm R})&=\frac{h_c(\bf x)}{h_c({\bf y}^{\rm R})}
     \lim_{t \to \infty} \widehat{f}(x,t\,|\,{\bf y}^{\rm R},\tau) 
     =\frac{h_c(\bf x)}{h_c({\bf y}^{\rm R})}\,\frac{\lambda}{(2\pi)^{n/2}{\Delta^{1/2}}} 
     \nonumber \\
     &\times\int_0^{\infty} \theta^{-n/2}\;\exp\left\{-\lambda\theta-\frac{({\bf x}-{\bf y}-\theta\;{\bf{m}})\cdot A^{-1}\cdot({\bf x}-{\bf y}-\theta\;{\bf{m}})^{\top}}{2\theta}\right\}\;{\rm d}\theta,
 \end{align*}
 where $\theta:=t-\tau$. Through algebraic manipulations, the previous expression becomes
 \begin{align*}
     f({\bf x}\,|\,{\bf y}^{\rm R})&=
     \frac{h_c(\bf x)}{h_c({\bf y}^{\rm R})}\,
     \frac{\lambda\, e^{ \Phi\left({\bf x,y}^{\rm R}\right)}}{(2\pi)^{n/2}{\Delta^{1/2}}} 
     \int_0^{\infty} \theta^{-n/2}\;\exp\left\{-\frac{1}{\theta}\,\Xi\left({\bf x,y}^{\rm R}\right)-\theta\left[\lambda+\frac{1}{2}{\bf m}\cdot A^{-1}\cdot{\bf m}^{\top}\right]\right\}{\rm d}\theta.
 \end{align*}
 The result in \eqref{staz_wiener} thus follows by making use of Eq.\  ET II 82(23)a of Gradshteyn and Ryzhik \cite{Gradstein} with a suitable choice of the parameters.
\end{proof}
We note that the stationary p.d.f.\ in Eq.\ \eqref{staz_wiener} is an extension of the result presented in Eq.\ (27) of Evans and Majumdar \cite{Evans} about the stationary solution of a standard $d$-dimensional Wiener process 
subject to resets occurring with constant intensity $r$. 
\par
We conclude this section by investigating the representation of the process ${\bf X}_t$ in the presence of a potential $U({\bf x})$, as seen in Eq.\ (\ref{NEWsde_x_trasf}). 
Here, and in the following, $\log$ denotes the natural logarithm. 
%
\begin{theorem}\label{theor_pot_wiener}
The diffusion process having infinitesimal moments (\ref{mom_infinitesimali_fx}) satisfies the SDE (\ref{NEWsde_x_trasf}) with potential 
   \begin{equation}
   U({\bf x})=-{\bf m}\cdot {\bf x}^\top-\frac{\log(h_c({\bf x}))}{r_1}\ {\bf r}\cdot  a_{*,1} +k, 
   \qquad {\bf x} \in \mathbb{R}^n,\;\; k \in \mathbb{R},
   \label{potential_wiener_n}
   \end{equation}
with $h_c({\bf x})$ given in (\ref{3.3generica}), provided that ${\bf r}$
satisfies  condition (\ref{3.4generica}), such that 
${\bf r}\neq {\bf 0}$ and 
   \begin{equation}
       r_1\ {\bf r}\cdot A-\left({\bf r}\cdot  {a}_{*,1}\right) \  {\bf r}={\bf 0}.
      \label{cond_pot}
   \end{equation}
\end{theorem}
\begin{proof}
Let us consider the differential 1-form defined by $\omega({\bf x}):=-\sum_{i=1}^n b_i({\bf x})\ {\rm d}x_i$, where the first order infinitesimal moments are given in \eqref{mom_infinitesimali_fx}. 
To obtain the potential $U({\bf x})$, we see that $\omega({\bf x})$ is an exact differential form. Indeed, the condition $\partial_{x_k}b_i({\bf x})=\partial_{x_i}b_k({\bf x})$ is satisfied for all $i<k$. For $i=1$ one obtains Eq.\ (\ref{cond_pot}), and for $i=2,\ldots,n$ the resulting equations 
   \begin{equation}
       r_i\ {\bf r}\cdot A-\left({\bf r}\cdot  {a}_{*,i}\right) \  {\bf r}={\bf 0} 
      \label{cond_poti}
   \end{equation}
are equivalent to (\ref{cond_pot}).
Furthermore, since the potential is a primitive of the differential form $\omega$, we compute $U({\bf x})$ such that:
\begin{equation} \label{syst_pot}
    \nabla U=-{\bf m}+\frac{2c}{1+c\, e^{2 {\bf r}\cdot{\bf x}^\top}}\,{\bf r}\cdot A.
\end{equation}
By integrating the first element of $\nabla U$ with respect to $x_1$ one has:
\begin{equation}
  U({\bf x})=-m_1 x_1- \frac{1}{r_1} \, \log(1+c\,e^{-2 {\bf r}\cdot{\bf x}^\top})\ 
  {\bf r}\cdot  {a}_{*,1}+\gamma_1(x_{2..n}),
  \label{proof_pot1}
\end{equation}
where $\gamma_1(x_{2..n})$ is an arbitrary function of $x_2,x_3,\ldots,x_n$. By equating the derivative of  (\ref{proof_pot1}) with respect to $x_2$ to the corresponding element of the gradient in (\ref{syst_pot}), we obtain:
$$
\partial_{x_2}\gamma_1(x_{2..n})
=-m_2+\frac{2ce^{-2 {\bf r}\cdot{\bf x}^\top}}{1+ce^{-2 {\bf r}\cdot{\bf x}^\top}}
\left({\bf r}\cdot  {a}_{*,2}-\frac{r_2}{r_1}\ {\bf r}\cdot {a}_{*,1}\right).
$$
The right-hand side of the above equation depends only on $x_2,x_3,\ldots,x_n$ if and only if 
\begin{equation}\label{ass1_pot}
{\bf r}\cdot  {a}_{*,2}-\frac{r_2}{r_1}\ {\bf r}\cdot  {a}_{*,1}=0,
\end{equation}
so that
$$
\gamma_1(x_{2..n})=-m_2x_2+\gamma_2(x_{3..n}), 
$$
where $\gamma_2(x_{3..n})$ is an arbitrary function of $x_3,x_4,\ldots,x_n$ and thus 
\begin{equation*}
  U({\bf x})=-m_1 x_1-m_2 x_2-\frac{1}{r_1}\,\log(1+c\,e^{-2 {\bf r}\cdot{\bf x}^\top}) \ 
  {\bf r}\cdot  {a}_{*,1}+\gamma_2(x_{3..n}).
\end{equation*}
However, condition \eqref{ass1_pot} is verified because of Eq.\ \eqref{cond_pot}. By applying the same procedure again, for the $k$-th element of the gradient in (\ref{syst_pot}), for $k>2$, one obtains the expression of the potential in (\ref{potential_wiener_n}) by using \eqref{cond_pot}. 
Finally,  we note that $U({\bf x})$ can be recovered through identical steps if one begins by integrating the $j$-th element of the gradient in \eqref{syst_pot} with respect to $x_j$, for $j=2,\ldots,n$. Thus, by making use of \eqref{cond_poti} for $i=j$, the potential function will have the following expression:
$$  
 U({\bf x})=-{\bf m}\cdot {\bf x}^\top-\frac{1}{r_j}\, 
 \log(h_c({\bf x}))\,{\bf r}\cdot  {a}_{*,j}+k, 
 \qquad {\bf x} \in \mathbb{R}^n,\;\; k \in \mathbb{R}, \;\; j\in\{2,\ldots,n\},
$$
which is equivalent to \eqref{potential_wiener_n} for any $j$.
\end{proof}   
%
%
%

\subsection{Transformed Ornstein--Uhlenbeck process}\label{sect:OUtransf}
Let $\widehat{\bf X}_t$ be the $n$--dimensional OU process defined 
in ${\cal D}\equiv \mathbb{R}^n$, with infinitesimal moments 
\begin{equation}
\begin{array}{ll} 
	\widehat{b}_i({\bf x})=\eta\,x_i & (i=1,\ldots,n), 
    \\
	\widehat{a}_{ij}({\bf x})=a_{ij}
    & (i,j=1,\ldots,n), 
\end{array}
\label{mom_infinitesimali_OU}
\end{equation}
where $\eta>0$, and where 
$a_{ij}\in\mathbb{R}$ $(i,j=1,\ldots,n\,;i\not= j)$, $a_{ii}>0$ $(i=1,\ldots,n)$, with $\Delta:=|A|>0$, 
the matrix  $A$ being symmetric and positive definite. 
For all $t>\tau\geq 0$ and ${\bf x}, {\bf y} \in \mathbb{R}^n$, the transition p.d.f.\ of $\widehat{\bf X}_t$ is given by: 
\begin{equation}
\widehat{f}({\bf x},t\,|\,{\bf y},\tau)=
{\cal N} \Big({\bf y}e^{\eta\,(t-\tau)}, \left(e^{2\eta\,(t-\tau)}-1\right) A \Big).
    \label{3.12}
\end{equation}
%
%
We consider the family of positive functions $h_c({\bf x}) \in C^2(\mathbb{R}^n)$ defined by
\begin{equation}
h_c({\bf x})=1+c\,{\rm Erf}\bigl\{{\bf r}\cdot{\bf x}^\top\bigr\},
\qquad  
c \in (-1,0)\cup (0,1),
\label{3.13}
\end{equation}
where ${\rm Erf}(x)={2\over\sqrt{\pi}}\int_0^x\exp\{-z^2\}\ {\rm d}z$ denotes the error function, so the condition \eqref{2.5} of Theorem \ref{teo_1} is satisfied provided that  ${\bf r}\equiv(r_1,r_2,\ldots,r_n)$ is such that  
\begin{equation}
{\bf r}\cdot A \cdot{\bf r}^\top-\eta=0.
\label{3.15}
\end{equation}
Condition $\eta>0$ and the hypothesis on $A$ ensure that Eq.\ \eqref{3.15} certainly admits solutions for ${\bf r}$ in $\mathbb{R}^n$. 
\par
We remark also that assumption $\eta>0$ defines a diffusion process that does not show the usual mean-reverting behavior of OU processes. In this case, the process tends to drift away from ${\bf 0}$, which becomes an unstable equilibrium point, so that $\widehat{\bf X}_t$ can be  viewed as an ``explosive'' OU process. Note that the SDE defining the OU process still holds when $\eta>0$ (see Maller et al.\ \cite{maller}), but the process $\widehat{{\bf X}}_t$ is not stationary. 
Moreover, from the Feller's theory of boundary classification, we can state that the endpoints $-\infty$ and $+\infty$ are both natural, attractive, and unattainable, for any component $\widehat{{X}}^k_t$, $k=1,2,\ldots, n$. \par
We note that, in the case $\rho=0$, i.e.\ when the components of the process are independent, another possible choice of the weight function can be obtained from Remark \ref{rem_indip}. By this procedure, we obtain a different expression than Eq.\ (\ref{3.13}). 
\begin{remark}
The function $h_c({\bf x})$ satisfies the Lipschitz condition in Eq.\ \eqref{cond_lipschitz_w}.
Indeed, from Eq.\ \eqref{3.13} we can see that the mentioned condition is satisfied by 
$$
\Omega({\bf x})
=\frac{2c}{{\pi^{1/2}}}\frac{\exp\{-({\bf r}\cdot {\bf x}^\top)^2\}\ {\bf r}\cdot A}{1+c \ {\rm Erf}\{{\bf r}\cdot {\bf x}^\top\}}, 
\qquad {\bf x} \in \mathbb{R}^n.
$$
Hence, we need to prove that, for a certain $L_2>0$, one has $\lvert\lvert\Omega({\bf x})-\Omega({\bf y})\lvert\lvert<L_2 \lvert\lvert{\bf x}-{\bf y}\lvert\lvert$ 
for all ${\bf x},{\bf y} \in \mathbb{R}^n$. 
In this case, this inequality becomes
$$
 \Bigg\lvert\frac{2c}{{\pi^{1/2}}}\Bigg\lvert\ \lvert \lvert {\bf r}\cdot A\lvert\lvert 
 \left\lvert \frac{\exp\{-({\bf r}\cdot {\bf x}^\top)^2\}}{1+c \ {\rm Erf}\{{\bf r}\cdot {\bf x}^\top\}} \right\lvert
 <L_2 \lvert\lvert{\bf x}-{\bf y}\lvert\lvert,
$$
which is equivalent to show that the gradient
$$
 \nabla\left(\displaystyle \frac{\exp\{-({\bf r}\cdot {\bf x}^\top)^2\}}{1+c \ {\rm Erf}\{{\bf r}\cdot {\bf x}^\top\}} \right)
 = -2\left[\frac{c}{{\pi^{1/2}}}\frac{\exp\{-2({\bf r}\cdot{\bf x}^\top)^2\}}{(1+c \ {\rm Erf}\{{\bf r}\cdot {\bf x}^\top\})^2}
 +\frac{\exp\{-({\bf r}\cdot{\bf x}^\top)^2\}\ {\bf r}\cdot{\bf x}^\top}{1+c \ {\rm Erf}\{{\bf r}\cdot {\bf x}^\top\}}\right]\ {\bf r}
$$
is bounded. Taking into account that
$\displaystyle\frac{1}{1+|c|}<\frac{1}{1+c\ {\rm Erf}\{{\bf r}\cdot {\bf x}^\top\}}<\frac{1}{1-|c|}$ always holds when $|c|<1$, 
the bound is easily obtained by recalling that $\exp\{-({\bf r}\cdot{\bf x}^\top)^2\}$ is also bounded.
\end{remark}
\par
Recalling Theorem \ref{teo_1}, the infinitesimal moments of the transformed OU process ${\bf X}_t$ are given by
\begin{equation} \label{drift_OU}
\begin{array}{ll}
   b_i({\bf x})
   =\eta x_i  -\displaystyle{\frac{2}{{\pi^{1/2}}}}
   \frac{c\,\exp\{-({\bf r}\cdot{\bf x}^\top)^2\}}{1+c\,{\rm Erf} \{{\bf r}\cdot {\bf x}^\top\}}\ {\bf r}\cdot a_{*,i} 
&(i=1,\ldots,n),   \\
    a_{ij}({\bf x})=\widehat{a}_{ij}  &(i,j=1,\ldots,n).
\end{array}
\end{equation}
%
We can observe that $b_i({\bf 0})= -\frac{2c}{{\pi^{1/2}}}\,{\bf r}\cdot a_{*,i}$ for any $i=1,\ldots,n$.  
Moreover, for $k=1,\ldots,n$, 
$$
\lim_{x_k\to  \pm\infty }b_i({\bf x})= 
\left\{
\begin{array}{ll}
 \pm\infty,    & i=k, \\
  \eta x_i\equiv \widehat b_i({\bf x}),   & i\neq k,
\end{array}
\right.
$$
so  that the transformation performed by means of $h_c({\bf x})$ in \eqref{3.13} mostly affects the OU process $\widehat{\bf{X}}_t$ for small values of ${\bf x}$.
\begin{remark}
    Recalling Eqs.\ \eqref{mom_infinitesimali_OU} and \eqref{drift_OU}, 
we can obtain a bound for the Euclidean distance between the drift vectors of the 
processes $\widehat{{\bf X}}_t$ and ${\bf X}_t$, i.e.
$$
 \lvert\lvert \widehat{{\bf b}}
    ({\bf x})-{\bf b}
    ({\bf x})\rvert \rvert 
    =\frac{2}{{\pi^{1/2}}}\  \left\lvert\frac{c\,\exp\{-({\bf r}\cdot{\bf x}^\top)^2\}}{1+c\,{\rm Erf}({\bf r}\cdot{\bf x}^\top)}\right\rvert \ \lvert\lvert {\bf r}\cdot A\rvert \rvert.
$$
Indeed, by recalling the well-known bounds
$$
 \left(1-\exp\{-x^2\}\right)^{1/2}
 \leq{\rm Erf}(x)\leq
 \left(1-\exp\{-4x^2/\pi\}\right)^{1/2},
 \qquad x\geq 0,
$$
one gets  
\begin{equation*}
  \lvert\lvert \widehat{{\bf b}}
    ({\bf x})-{\bf b}
    ({\bf x})\rvert \rvert<\frac{2}{{\pi^{1/2}}}\ 
    \frac{\lvert\lvert {\bf r}\cdot  A\rvert \rvert}{1+\sign(c\ {\bf r}\cdot{\bf x}^\top)|c|
    \left(
    {1-\exp\left\{-\zeta_{\bf x} \left( {\bf r}\cdot{\bf x}^\top\right)^2\right\}}\right)^{1/2}},
\end{equation*}
    where
    $$
    \zeta_{\bf x}=\begin{cases}
        1, & {\rm if }\; c\ {\bf r}\cdot{\bf x}^\top\geq 0, \\ 
        {4}/{\pi}, & {\rm otherwise.}
    \end{cases}
    $$
Hence, in conclusion, for any $c \in (-1,0)\cup (0,1)$ we have
    $$
    \lvert\lvert \widehat{{\bf b}}
    ({\bf x})-{\bf b}
    ({\bf x})\rvert \rvert
    <\frac{2}{{\pi^{1/2}}}\,\frac{\lvert\lvert {\bf r}\cdot A\rvert \rvert}{\ell_{\bf x}}, 
    \qquad \hbox{for \ }
    \ell_{\bf x}=\begin{cases}
        1,  \qquad &{\rm if }\;\; c\ {\bf r}\cdot{\bf x}^\top\geq 0
        \\ 
        1-|c|,  
        \qquad &{\rm otherwise.}
    \end{cases}
$$
\end{remark}
%
\begin{remark}
If the weight function is chosen as in (\ref{3.13}), the condition given in (\ref{hyperplaneD}) leads to the hyperplane $H_{{\bf y}}=\{{\bf x}\in \mathbb{R}^n :\;{\bf r} \cdot {\bf x}^\top={\bf r} \cdot {\bf y}^\top\}$, which passes through the fixed initial point ${\bf y} \in \mathbb{R}^n$.
Moreover, the half-spaces ${\cal H}^-_{{\bf y}}$ and ${\cal H}^+_{{\bf y}}$ introduced in (\ref{spacesD}) are 
defined through the 
conditions ${\bf r} \cdot {\bf x}^\top<(>)\;{\bf r} \cdot {\bf y}^\top$ and 
${\bf r} \cdot {\bf x}^\top>(<)\;{\bf r} \cdot {\bf y}^\top$, respectively, for $c>0$ ($c<0$). 
Similarly to the case treated in Remark \ref{remark_coupling}, 
the probability of the maximal coupling \eqref{maximal_coupling_def} for the OU process and the OU-transformed process is analogous to Eq.\ \eqref{maximal_coupling_w}, with a possible slight modification on the regions depending on the sign of $c$ for $h_c$ given in \eqref{3.13}. 
%
\end{remark}
\par 
Due to Theorem \ref{teo_1}, the transition p.d.f.\ of ${\bf X}_t$ can be expressed as in \eqref{2.1} and recalling Eqs.\ (\ref{3.12}) and (\ref{3.13}). 
Moreover, in this case Theorem \ref{prop_w} cannot be applied, since $k({\bf x})= {\rm Erf}\bigl\{{\bf r}\cdot{\bf x}^\top\bigr\}$ is not positive for all ${\bf x}\in {\mathbb{R}}^n$. 
\par
For the process having infinitesimal moments (\ref{drift_OU}), it is not possible to determine analytically the number of maximum points of $f({\bf x},t\,|\,{\bf y},\tau)$. However, in Section \ref{subs:tOU} we provide some results for the transformed OU process ${\bf X}_t$ in the case $n=2$. 
%
%
\begin{remark} 
A suitable choice of ${\bf r}$ that satisfies the condition in Eq.\ (\ref{3.15}) is given by  
${\bf r}= \left( {\eta}/{a_{ii}}\right)^{1/2}\cdot {\bf e}_i$ 
for any $i=1,\ldots,n$. 
\end{remark}
\par
Similarly to Proposition \ref{prop_W_ord}, we are able to provide conditions such that the marginals of the diffusion processes considered in this section are stochastically ordered.   
\begin{proposition} 
Consider the $n$-dimensional diffusion processes whose infinitesimal moments are described in Eqs.\ \eqref{mom_infinitesimali_OU} and \eqref{drift_OU}. 
If $a_{ij}\geq 0$ for all $1\leq i,j\leq n$, and $c\, ({\bf r} \cdot {\bf x}^\top)$ is increasing for ${\bf x}\in \mathbb{R}^n$, then $\widehat{\bf X}_t\leq_{\rm st}{\bf X}_t$ for all $t\geq 0$.
\end{proposition}
\par
The proof proceeds along the same line of Proposition \ref{prop_W_ord}, and thus is omitted. 
\par
As done in Theorem \ref{theor_pot_wiener}, hereafter we study the representation of the process ${\bf X}_t$ in the presence of a potential $U({\bf x})$ (cf. Eq.\ (\ref{NEWsde_x_trasf})). 
\begin{theorem}\label{theor_pot_OU}
  The diffusion process having infinitesimal moments (\ref{drift_OU}) satisfies the SDE (\ref{NEWsde_x_trasf}) with potential 
   \begin{equation*}
   U({\bf x})=-\frac{\eta}{2}\,||{\bf x}||^{2}-\frac{1}{r_1}\,\log(h_c({\bf x}))\ {\bf r}\cdot {a}_{*,1}+k, 
   \qquad {\bf x} \in \mathbb{R}^n,\;\; k \in \mathbb{R},
   \end{equation*}
   with $h_c({\bf x})$ given in (\ref{3.13}),  provided that ${\bf r}$
verifies  conditions (\ref{cond_pot}) and (\ref{3.15}), with 
${\bf r}\neq {\bf 0}$. 
\end{theorem}
\begin{proof}
    Let us consider the differential 1-form $\omega({\bf x}):=-\sum_{i=1}^n b_i({\bf x})\ {\rm d}x_i$, where the first order infinitesimal moments are given in \eqref{drift_OU}. The exactness of $\omega({\bf x})$ is ensured, since relation \eqref{cond_pot} implies that the condition $\partial_{x_k}b_i({\bf x})=\partial_{x_i}b_k({\bf x})$ is fulfilled for all $i<k$. The rest of the proof follows similarly as seen in Theorem \ref{theor_pot_wiener} for the transformed Wiener process case.
\end{proof}    
\par
Since the proofs of both Theorems \ref{theor_pot_wiener} and \ref{theor_pot_OU} are based on the same differential 1-form, the corresponding regularity condition on the derivatives leads in both cases to the same equation, i.e.\ Eq.\ (\ref{cond_pot}). 
%
\section{Two-dimensional symmetric processes and absorbing boundaries} \label{boundaries}
In this section we investigate a symmetry property and dynamics in the presence of absorbing boundaries that can be envisaged for the processes investigated so far. For better tractability, we refer to the two-dimensional case, even if the general setting can be extended to higher dimensions.

\subsection{Symmetry properties}
With reference to Section \ref{sect:backgr}, for $n=2$ let $\widehat{\bf X}_t$ be the diffusion process in the domain ${\cal D}=I_1\times I_2\subseteq\mathbb{R}^2$, where $I_i=(\alpha_i,\beta_i)$ $(i=1,2)$, having infinitesimal moments (\ref{inf_mom_X_t_hat}). 
We consider the dynamics of $\widehat{\bf X}_t$ in the presence of the time-varying absorbing boundary 
\begin{equation}
 C(t)=\{{\bf x}\in {\cal D}:\,x_1=h(x_2,t)\},
 \qquad t\geq 0,
\label{eq:defCt}
\end{equation}
where the function $h= h(x_2,t):I_2 \times [0,+\infty)\rightarrow I_1$ is continuous and has continuous derivatives with respect to $x_2$ and $t$. 
For all $t \in [0,+\infty)$, we consider the disjoint regions ${\cal D}_1(t)=\{{\bf x}\in {\cal D}:x_1<h(x_2,t)\}$ and ${\cal D}_2(t)=\{{\bf x}\in {\cal D}:x_1>h(x_2,t)\}$. 
Now we focus on a transformation ${\bf\Psi}({\bf x},t)\equiv (\psi_1({\bf x},t),\psi_2({\bf x},t)):{\cal D}\times [0,+\infty)\to {\cal D}$, such that there exists a one-to-one correspondence between ${\bf x}$ and ${\bf\Psi}({\bf x},t)$ for all $t \in [0,+\infty)$, and let
$$
J({\bf x},t)
=\partial_{x_1}\psi_1({\bf x},t)\,\partial_{x_2}\psi_2({\bf x},t)-\partial_{x_1}\psi_2({\bf x},t)\,\partial_{x_2}\psi_1({\bf x},t)\neq 0
$$  
be the Jacobian determinant of the transformation.
Then, given a continuous function $\varphi: {\cal D}\times [0,+\infty)\to \mathbb R$ such that $\varphi({\bf x},t)>0$ for all ${\bf x}\in{\cal D}$ and $t\in[0,+\infty)$, in the following we shall assume the symmetry relation of the transition p.d.f.\ of $\widehat{\bf X}_t$, i.e.
\begin{equation} \label{cond_simm}
    \widehat{f}({\bf x},t\,|\,{\bf y},\tau)
    =\frac{\varphi({\bf x},t)}{\varphi({\bf y},\tau)}\,|J({\bf x},t)|\,
    \widehat{f}({\bf \Psi} ({\bf x},t),t\,|\,{\bf \Psi} ({\bf y},\tau),\tau),
    \qquad \forall \;{\bf x},{\bf y}\in{\cal D}, 
    \quad t>\tau\geq 0,
\end{equation}
to hold.
Let us now consider the diffusion process ${\bf X}_t$ satisfying the SDE (\ref{sde_x_trasf}), and whose transition p.d.f.\  is related to the transition p.d.f.\ of $\widehat{\bf X}_t$ through the product-form identity given in (\ref{2.1}). 
\begin{lemma} \label{lemma_simm}
Under the assumptions of Theorem \ref{teo_1}, let ${\bf X}_t$ be the diffusion process having infinitesimal moments (\ref{2.5bc}). 
Let the underlying diffusion process $\widehat{\bf X}_t$ have symmetric transition p.d.f.\ satisfying Eq.\ \eqref{cond_simm}, and let the transformation ${\bf\Psi}({\bf x},t)$ and the function $\varphi({\bf x},t)$ fulfill the conditions stated above. 
Then, the transition p.d.f.\ of ${\bf X}_t$ satisfies the following symmetry relation:
\begin{equation}
 f({\bf x},t\,|\,{\bf y},\tau)
 =\frac{h({\bf x})}{h({\bf y})}\,
 \frac{h({\bf \Psi}({\bf y},\tau))}{h({\bf \Psi} ({\bf x},t))}
 \,
 \frac{\varphi({\bf x},t)}{\varphi({\bf y},\tau)}\,|J({\bf x},t)|\,
    {f}({\bf \Psi} ({\bf x},t),t\,|\,{\bf \Psi} ({\bf y},\tau),\tau),
    \qquad \forall \;{\bf x},{\bf y}\in{\cal D}, \quad t>\tau\geq 0.
 \label{eq:symmperf}
\end{equation}
\end{lemma}
\begin{proof}
Eq.\ (\ref{eq:symmperf}) follows by using the product-form relation (\ref{2.1}) in both sides of Eq.\ (\ref{cond_simm}). 
\end{proof}
\subsection{First-crossing time and absorbing boundaries}
Let us now introduce the first-crossing time of $\widehat{\bf X}_t$ through the curve given in (\ref{eq:defCt}). Specifically, for all $\tau \geq 0$ and ${\bf y}\in {\cal D}\setminus C(\tau)$, we define :
\begin{equation}
    \widehat T_{C} =\left\{
       \begin{array}{ll}
        \inf \{t\geq \tau : \widehat{\bf X}_t\in {\cal D}_2(t)\}, \quad 
       \widehat{\bf X}_{\tau}={\bf y} \in {\cal D}_1(\tau),
       \\[5pt]
        \inf \{t\geq \tau : \widehat{\bf X}_t\in {\cal D}_1(t)\}, \quad 
       \widehat{\bf X}_{\tau}={\bf y} \in {\cal D}_2(\tau).
\end{array}
    \right.
    \label{eq:defFCThatT}
\end{equation}
Moreover, for all $t>\tau \geq 0$, and ${\bf y}\in {\cal D}\setminus C(\tau)$ we introduce the transition p.d.f.\ of $\widehat{\bf X}_t$ in the presence of the absorbing boundary $C(t)$, also known as {\em taboo} transition p.d.f.:
\begin{equation}
\widehat{f}^A({\bf x},t\,|\,{\bf y},\tau)\,{\rm d}{\bf x}
=\mathsf P\big(\widehat{\bf X}_t\in {\rm d}{\bf x}, \widehat T_{C}>t \,|\,\widehat{\bf X}_{\tau}={\bf y}\big), 
\qquad {\bf x} \in {\cal D}.
\label{eq:defhatfA}
\end{equation}
Obtaining closed-form expressions for the density of $\widehat T_{C}$ and for the taboo transition p.d.f.\ in general is not an easy task. For instance, in the case of two-dimensional Wiener and OU processes the first-exit time through time-varying ellipses is studied by means of computational and asymptotic results (see Di Crescenzo et al.\ \cite{Di_Crescenzo2024}). 
However, hereafter we show that an approach based on the symmetry specified in Eq.\ (\ref{cond_simm}) and on the product-form relation in (\ref{2.1}) allows us to obtain net and tractable expressions for the taboo transition p.d.f.\ of the transformed process ${\bf X}_t$. 
\par 
Under the assumption that the symmetry relation (\ref{cond_simm}) holds for all ${\bf x},{\bf y}\in{\cal D}$ and $t>\tau\geq 0$, and assuming that the transformation ${\bf\Psi}({\bf x},t)$ 
and the function $\varphi({\bf x},t)$ are such that 
\begin{align}\label{cond_teo_1995}
& {\bf x}\in{\cal D}_i \quad \Longleftrightarrow \quad 
    {\bf\Psi}({\bf x},t)\in {\cal D}_j \qquad 
    \forall \; i,j=1,2,  \; i\neq j,
    \nonumber
\\
    & \psi_1(h(x_2,t), x_2, t)=h(x_2,t), \qquad
    \psi_2(h(x_2,t), x_2, t)=x_2, 
\\    
    & \varphi(h(x_2,t), x_2, t)=1,
    \qquad \forall x_2\in I_2, \; t\geq 0,
 \nonumber    
\end{align}  
then the taboo transition p.d.f.\ (\ref{eq:defhatfA}) can be expressed as 
\begin{equation}
 \widehat{f}^A({\bf x},t\,|\,{\bf y},\tau)
 =\widehat{f}({\bf x},t\,|\,{\bf y},\tau)
 -\varphi({\bf x},t)\,\widehat{f}({\bf \Psi} ({\bf x},t),t\,|\,{\bf y},\tau), 
\label{fA_proc_orig}
\end{equation}
for all $t>\tau \geq 0$ and 
$({\bf x},{\bf y})\in D_1(t)\times D_1(\tau)$ or 
$({\bf x},{\bf y})\in D_2(t)\times D_2(\tau)$ (see Theorem 3.1 of Di Crescenzo et al.\ \cite{Di_Crescenzo95} for details).
%
%
%
%
Hereafter, we show that the symmetry identity expressed in 
(\ref{cond_simm}) not only leads to the relation for the p.d.f.\ $\widehat{f}^A$ given in (\ref{fA_proc_orig}), but it also allows us to develop similar results for the transformed process ${\bf X}_t$. 
Similarly as Eq.\ (\ref{eq:defhatfA}), we can introduce the transition p.d.f.\ of ${\bf X}_t$ in the presence of the absorbing boundary $C(t)$, i.e.
\begin{equation}
{f}^A({\bf x},t\,|\,{\bf y},\tau)\,{\rm d}{\bf x}
=\mathsf P\big({\bf X}_t\in {\rm d}{\bf x},  T_{C}>t \,|\,{\bf X}_{\tau}={\bf y}\big), 
\qquad {\bf x} \in {\cal D},
\label{eq:deffA}
\end{equation}
for all $t>\tau \geq 0$, and ${\bf y}\in {\cal D}\setminus C(\tau)$, where, similarly as $\widehat T_{C}$ defined in Eq.\ (\ref{eq:defFCThatT}), $T_{C}$ denotes the first-crossing time of ${\bf X}_t$ through $C(t)$. The probabilistic analysis of $T_{C}$ 
relies on 
\begin{equation}
\begin{array}{l}
    g^*(h(z,t),z,t\,|\,{\bf y},\tau)\,{\rm d}t\,{\rm d}z  
    \\
    \qquad = \mathsf{P}\{t\leq T_C \leq t+ {\rm d}t, h(z,t)\leq X_1(t)<h(z+{\rm d}z,t),\; z\leq X_2(t)<z+{\rm d}z \,|\, {\bf X}_{\tau}={\bf y}\},  
\end{array}
    \label{gstar}
\end{equation}
for $z \in I_2$ and $t> \tau\geq 0$, since $\mathsf{P}(T_C \in {\rm d}t\,|\, {\bf X}_{\tau}={\bf y})/ {\rm d}t
=\int_{I_2}g^*(h(z,t),z,t\,|\,{\bf y},\tau)\,{\rm d}z$. 
From the Markovianity  and the continuity a.s.\ of the sample paths of  ${\bf X}_t$, 
the following integral equation holds:
\begin{equation} \label{relaz_f}
    f({\bf x},t\,|\, {\bf y},\tau)
    =\int_{\tau}^t {\rm d}\theta \int_{I_2}g^*(h(z,\theta),z,\theta\,|\,{\bf y},\tau)
    \,f({\bf x},t\,|\,h(z,\theta),z,\theta)\,{\rm d}z,
\end{equation}
for $({\bf x},{\bf y})\in  D^*_2(t)\times D_1(\tau)$ or $({\bf x},{\bf y})\in  D^*_1(t)\times D_2(\tau)$, where $D^*_i(t)=D_i(t) \cup C(t)$, $i=1,2$, as well as
\begin{equation} \label{relaz_fA}
    f^A({\bf x},t\,|\, {\bf y},\tau)
    = f({\bf x},t\,|\, {\bf y},\tau)-\int_{\tau}^t {\rm d}\theta \int_{I_2}g^*(h(z,\theta),z,\theta\,|\,{\bf y},\tau)
    \,f({\bf x},t\,|\,h(z,\theta),z,\theta)\,{\rm d}z,
\end{equation}
for all $t>\tau \geq 0$ and $({\bf x},{\bf y})\in  D_1(t)\times D_1(\tau)$ or $({\bf x},{\bf y})\in  D_2(t)\times D_2(\tau)$.
\par
We are now able to obtain an explicit result for the taboo transition p.d.f.\ of ${\bf X}_t$. 
We emphasize the limited availability of results of this type in the literature, that often force authors 
to focus mainly on numerical computations or asymptotic approximations of crossing-time densities 
of diffusion processes. 
\begin{proposition}\label{prop_taboo_pdf}
Under the assumptions of Theorem \ref{teo_1}, let ${\bf X}_t$ be the diffusion process having infinitesimal moments (\ref{2.5bc}), and let $C(t)$ be the curve defined in \eqref{eq:defCt}. 
Let the underlying diffusion process $\widehat{\bf X}_t$ have symmetric transition p.d.f.\ satisfying Eq.\ \eqref{cond_simm}, and let the transformation ${\bf\Psi}({\bf x},t)$ and the function $\varphi({\bf x},t)$ satisfy conditions (\ref{cond_teo_1995}). 
Then, the transition p.d.f.\ in the presence of the absorbing boundary $C(t)$, defined in \eqref{eq:deffA}, can be expressed in terms of the transition p.d.f.\ of ${\bf X}_t$ as 
%
\begin{equation}
{f}^A({\bf x},t\,|\,{\bf y},\tau)
={f}({\bf x},t\,|\,{\bf y},\tau)
-\frac{h({\bf x})}{h({\bf \Psi}({\bf x},t))}\,
\varphi({\bf x},t)\,\lvert J({\bf x},t)\rvert\, 
f({\bf \Psi}({\bf x},t),t\,|\,{\bf y},\tau), 
\label{fA_proc_new}
\end{equation}  
for all $t>\tau \geq 0$ and $({\bf x},{\bf y})\in D_1(t) \times D_1(\tau)$ or 
$({\bf x},{\bf y})\in D_2(t)\times D_2(\tau)$. 
\end{proposition}
\begin{proof}
Let us evaluate the p.d.f.\ of ${\bf X}_t$ when the initial state of the process is on the curve \eqref{eq:defCt}. By taking  
$({\bf y},\tau)=(h(z,\theta),z,\theta)$, the identities given in \eqref{cond_teo_1995} yield  
$$
h({\bf \Psi}({\bf y},\tau))=h({\bf \Psi}(h(z,\theta),z,\theta))
=h(h(z,\theta),z)=h({\bf y}), \qquad 
\varphi({\bf y},\tau)=\varphi(h(z,\theta),z,\theta)=1.
$$
Hence, applying Lemma \ref{lemma_simm}, from 
Eq.\ \eqref{eq:symmperf} one has  
\begin{equation} \label{rel_simm_f}
  f({\bf x},t\,|\,h(z,\theta),z,\theta)=  \frac{h({\bf x})}{h({\bf \Psi} ({\bf x},t))}
 \,
 \varphi({\bf x},t)\,|J({\bf x},t)|\,
    {f}({\bf \Psi} ({\bf x},t),t\,|\,\,h(z,\theta),z, \theta).
\end{equation}


Recalling the relation \eqref{relaz_fA}, from \eqref{rel_simm_f} we get
\begin{align*}
   {f}^A({\bf x},t\,|\,{\bf y},\tau)
&={f}({\bf x},t\,|\,{\bf y},\tau)\\
&-\frac{h({\bf x})}{h({\bf \Psi}({\bf x},t))}\,
\varphi({\bf x},t)\,\lvert J({\bf x},t)\rvert\, 
 \int_{\tau}^t {\rm d}\theta \int_{I_2}g^*(h(z,\theta),z,\theta\,|\,{\bf y},\tau)
    \,f({\bf \Psi}({\bf x},t),t\,|\,h(z,\theta),z,\theta) \,{\rm d}z.
\end{align*}
The result (\ref{fA_proc_new}) thus follows by making use of Eq.\ \eqref{relaz_f}.
\end{proof}
\par
Let us now obtain a product-form identity for the taboo transition p.d.f.'s defined in \eqref{eq:defhatfA} and \eqref{eq:deffA}. 
\begin{corollary} \label{cor_taboo_pdf}
Under the assumptions of Proposition \ref{prop_taboo_pdf}, the taboo transition p.d.f.'s of processes ${\bf X}_t$ and $\widehat{\bf X}_t$ satisfy
$$
 f^A({\bf x},t\,|\,{\bf y},\tau)
 =\frac{h({\bf x})}{h({\bf y})}\,
 \widehat{f}^A({\bf x},t\,|\,{\bf y},\tau),
$$ 
for all $t>\tau \geq 0$ and $({\bf x},{\bf y})\in D_1(t) \times D_1(\tau)$ or $({\bf x},{\bf y})\in D_2(t)\times D_2(\tau)$. 
\end{corollary}
\begin{proof}
The proof follows by making use of the product-form relation \eqref{2.1} in both terms on the right-hand-side of \eqref{fA_proc_new}, and recalling Eq.\ \eqref{fA_proc_orig}. 
\end{proof}
In Section $4$ of \cite{Di_Crescenzo95}, the transition p.d.f. $\widehat{f}^A({\bf x},t|{\bf y},\tau)$ has been specified for the two-dimensional linear process $\widehat{\bf X}_t$, that will be considered in the following section.
\section{Two-dimensional transformed linear processes}\label{sect:5}
In this section, we analyze in detail the case of a diffusion in $\mathbb{R}^2$. 
Specifically, with reference to (\ref{inf_mom_X_t_hat}), let the process 
$\widehat{\bf X}_t$ be a two-dimensional linear diffusion process with infinitesimal moments
\begin{equation}
\widehat{b}_{i}(\textbf{x})
=\theta_{i1}x_1+\theta_{i2}x_2+\theta_{i3} \hspace{0.5 cm} (i=1,2),
\qquad 
 {\widehat{ A}({\bf x})}=A := 
\begin{bmatrix}
	a_1^2 & \rho a_1 a_2 \\
	\rho a_1 a_2 & a_2^2
    \end{bmatrix},   
    \label{inf_mom_lin_bid}
\end{equation}
where $\theta_{ij} \in \mathbb{R}$ $(i=1,2;\,j=1,2,3)$, 
$a_{i} >0$ $(i=1,2),$ with $-1 < \rho<1$. Moreover, the matrix 
$A$ is positive definite, so $\Delta:=   |A| = a_1^2 a_2^2(1-\rho^2)>0$. 
\par
The function $Z({\bf x}):=A^{-1}\cdot M \cdot {\bf x}^\top$ must verify the "potential condition"
$\partial_{x_2}Z_1(x_1,x_2)=\partial_{x_1} Z_2(x_1,x_2)$ 
(cf.\ Gardiner \cite{Gardiner}), that is
\begin{equation} \label{cond_staz_Y}
a_1^2 \theta_{21}-a_2^2 \theta_{12}+\rho a_{1}a_2(\theta_{22}-\theta_{11})=0. 
\end{equation}
Let us also consider 
\begin{equation}
M = \begin{bmatrix}
	\theta_{11} & \theta_{12} \\
	\theta_{21} & \theta_{22}
\end{bmatrix} ,
\qquad
{\bf m} =\left(\theta_{13},\,\theta_{23}
\right).
\label{def_M_m}
\end{equation}
We denote the eigenvalues of $M$ as 
\begin{equation} \label{autovalori_m}
    \lambda_{1,2}=\frac{1}{2}\left({\rm Tr}(M)\pm {\xi^{1/2}}\right), 
    \qquad \lambda_{1}<\lambda_{2},
\end{equation}
where 
\begin{equation}
\xi:=\left({\rm Tr}(M)\right)^2-4|M|=(\theta_{22}-\theta_{11})^2+4 \theta_{12}\theta_{21}, 
\label{sigmaM}
 \end{equation}
and
$$
|M|=\lambda_1 \lambda_2=\theta_{11}
\theta_{22}-\theta_{12}\theta_{21}, 
\qquad
{\rm Tr}(M)=\lambda_1+\lambda_2=\theta_{11}+\theta_{22}.
$$
Note that condition (\ref{cond_staz_Y}) assures that $\xi \geq 0$, so that the eigenvalues $\lambda_{1,2}$ are real numbers.\par
Moreover, due to condition (\ref{cond_staz_Y}), the transition p.d.f.\ 
of ${\bf \widehat{X}}_t$ is normal, with 
\begin{equation}\label{f_norm}
\widehat{f}({\bf x},t\,|\,{\bf y},\tau)=
{\cal N}\Big(\left(M_i(t\,|\,{\bf y},\tau)\right),\left(D_{ij}(t\,|\,\tau)\right)\Big),
\end{equation}
where the mean values $M_i(t\,|\,{\bf y},\tau)$ and the terms of the covariance matrix $D_{ij}(t\,|\,\tau)$ can be explicitly given in the following three cases (for instance, see Section 4 of \cite{Di_Crescenzo95} for details): 
\\
(i) $\lambda_1=\lambda_2=0$, in this case the only possible occurrence is $\theta_{11}=\theta_{22}=\theta_{12}=\theta_{21}=0$; the process $\widehat{\bf X}_t$ is a Wiener process; 
\\
%
(ii)  $\lambda_1=\lambda_2\neq 0$, in this case the only possible occurrence is $\theta_{11}=\theta_{22}\neq 0$ and $\theta_{12}=\theta_{21}=0$; the process $\widehat{\bf X}_t$ is a OU process;
\\
(iii) $\lambda_1 \neq \lambda_2$. 
%
\par
For the process $\widehat{\bf X}_t$ having infinitesimal moments (\ref{inf_mom_lin_bid}) let us now define the following index:
\begin{equation}
    \delta= \delta\left(\widehat{b}_{i}(\textbf{x})\right):=\left\{
    \begin{array}{ll}
        0,& \text{if}\; M= ({\bf 0}\; {\bf 0})^\top\;
        \text{and} \;{\bf m}\neq {\bf 0},
    \\   
    1,& \text{if}\; M\neq ({\bf 0}\; {\bf 0})^\top
    \; \text{and}\;{\bf m}= {\bf 0}.
\end{array}
\right.
\label{delta}
\end{equation}
%
The following Proposition \ref{prop_w_dim2} focuses on  
$h_c({\bf x})$ for the two-dimensional linear diffusion process. 
Preliminarily, hereafter we investigate some properties of $M$ and its eigenvalues (\ref{autovalori_m}), recalling that the expressions of ${\rm Tr}(M)$ and $|M|$ are given respectively in (\ref{def_M_m}) and (\ref{sigmaM}) (see also Section 4 of \cite{Di_Crescenzo95}). 
\par
%
\begin{table}[t] 
\begin{tabular}{l  |l   |c|c|c|c}
\hline
&     Cases   & $\delta$ &  $|M|$ & ${\rm Tr}(M)$ & $\lambda_{1,2}$   \\
\hline 
$\lambda_1=\lambda_2$&(a)   $\lambda_1=0$  &  0  & 0  & 0 & $0$\\
$(\xi=0)$    &(b)   $\lambda_1\neq 0$ ($\theta_{11}=\theta_{22}\neq 0,\; \theta_{12}=\theta_{21}=0$)  &  1  & $\theta_{11}^2$  & $2 \theta_{11}$ & $\theta_{11}$ \\
\hline
$\lambda_1 \neq \lambda_2$ & (c) $\lambda_1+\lambda_2=0$  &  1  & $-\theta_{11}^2-\theta_{12}\theta_{21}$   & 0 & $\lambda_{2}=\frac{{\xi^{1/2}}}{2}$\\     
$(\xi>0)$ &(d) $\theta_{21}=0,\theta_{11}=0,\lambda_1+\lambda_2 \neq 0$  &  1    & $0$   & $\theta_{22}$ & $\lambda_{2}=\theta_{22}$\\   
&(e) $\theta_{21}=0,\theta_{11}\neq 0,\lambda_1+\lambda_2 \neq 0$  &  1    &$\theta_{11}\theta_{22}$  & $\theta_{11}+\theta_{22}$ & $\frac{\theta_{11}+\theta_{22}\pm |\theta_{22}-\theta_{11}|}{2}$ \\  
&(f) $\theta_{21}\neq 0,\lambda_1\lambda_2=0,\lambda_1+\lambda_2 \neq 0$ &  1    & 0  & $\theta_{11}+\theta_{22}$ & $\frac{\theta_{11}+\theta_{22}\pm |\theta_{11}+\theta_{22}|}{2}$ \\ 
&(g) $\theta_{21}\neq 0,\lambda_1\lambda_2\neq 0,\lambda_1+\lambda_2 \neq 0$  &  1    & $\theta_{11}\theta_{22}-\theta_{12}\theta_{21}$  & $\theta_{11}+\theta_{22}$ &  cf.\ (\ref{autovalori_m}) \\ 
\hline
\end{tabular}
\caption{The suitable instances for the eigenvalues of $M$ in (\ref{def_M_m}). 
The last column, when $\delta=1$, gives the possible occurrences 
of positive eigenvalues, under the conditions specified in 
Lemma \ref{pos_eigenvalue}.}
\label{table_cases}
\end{table}
\begin{lemma}\label{pos_eigenvalue}
Under the potential condition (\ref{cond_staz_Y}), if 
$M\neq ({\bf 0}\;{\bf 0})^\top$  and ${\bf m}={\bf 0}$, that is if $\delta=1$, 
then at least one eigenvalue $\lambda_{1,2}$ is positive under the following conditions, based on the cases of Table \ref{table_cases}:
\begin{itemize}
    \item ${\rm Tr}(M)>0$ in the cases (b), (d), (f);
    \item $|M|<0$ in the case (e);
    \item ${\rm Tr}(M)>0$ or $\{{\rm Tr}(M)<0,|M|<0\}$ in the case (g).
\end{itemize}
Moreover, in the case (c) one of the eigenvalues, 
i.e.\ ${\xi^{1/2}}/2$, is always positive. 
\end{lemma}
\begin{proposition} \label{prop_w_dim2}
Let the assumptions of Lemma \ref{pos_eigenvalue} hold, and let the vector ${\bf r}=(r_1,r_2)$ satisfy 
\begin{equation}
 {\bf r}\cdot   A \cdot{\bf r}^\top=\left\{
    \begin{array}{ll}
      {\bf r} \cdot {\bf m}^\top,   & \text{if}\; \delta=0,\\
       \lambda,  & \text{if}\; \delta=1,
    \end{array}
    \right.
    \label{lambda}
\end{equation}
where $\lambda$ is any of the positive eigenvalues of $M$. Then, 
a solution of Eq.\ (\ref{2.5}) for the process $\widehat{\bf X}_t$ with infinitesimal moments (\ref{inf_mom_lin_bid}) is:
\begin{equation}
    h_c({\bf x})=1+c \int_0^{{\bf r}\cdot{\bf x}^\top} \exp \{-(2-\delta)z^{\delta+1}\}\ {\rm d}z= 
\left\{
\begin{array}{ll}
 1+c\exp\bigl\{-2\,{\bf r} \cdot {\bf x}^\top\bigr\},  & \text{if}\; \delta=0, 
 \\[2mm]
 1+c\,{\rm Erf}\bigl\{{\bf r}\cdot{\bf x}^\top\bigr\}, & \text{if}\; \delta=1,
\end{array}
\right. 
    \label{k_lin_proc}
\end{equation}
where $\delta$ is given in Eq.\ (\ref{delta}), and where $c\in (0,+\infty)$ if $\delta=0$, and $c\in (-1,0)\cup (0,1)$ if $\delta=1$. 
\end{proposition}
\begin{proof}
When $\delta=0$, the expression of $h_c({\bf x})$ in (\ref{k_lin_proc}) is equal to that given in (\ref{3.3generica}) for the multidimensional case.   
For $\delta=1$, since ${\bf m}= {\bf 0}$, the function $h_c({\bf x})$ in (\ref{k_lin_proc}) is solution of Eq.\ (\ref{2.5}) for the process $\widehat{\bf X}_t$ with infinitesimal moments (\ref{inf_mom_lin_bid}) provided that 
%
%
\begin{equation}
{\bf r}\cdot M \cdot {\bf x}^\top-({\bf r}\cdot  A  \cdot{\bf r}^\top) 
\; {\bf r}\cdot{\bf x}^\top=0,
\qquad {\bf x}\in \mathbb{R}^2,
\label{cond_delta1}
\end{equation}
with $M$ given in (\ref{def_M_m}).
A particular solution of (\ref{cond_delta1}) is:
$$
M\cdot {\bf x}^\top=({\bf r}\cdot  A \cdot{\bf r}^\top)\; {\bf x},
$$
so that  ${\bf r}\cdot  A \cdot{\bf r}^\top$ is an eigenvalue of $M$. Since $A$ is positive definite, the conditions in Lemma \ref{pos_eigenvalue} ensure that there exists at least one positive eigenvalue.
\end{proof}
\par
Clearly, the condition (\ref{lambda}) reduces to Eqs.\ (\ref{3.4generica}) and (\ref{3.15}) 
in the Wiener and OU process case, respectively. 
As a consequence of Proposition \ref{prop_w_dim2}, the approach followed in Section \ref{sect:OUtransf} for the transformed Ornstein--Uhlenbeck process can also be adopted for more general instances in which $M\neq ({\bf 0}\;{\bf 0})^\top$. The details are omitted since the procedure is analogous. 
\par
We now present an interpretation of $\widehat{\bf X}_t$ as a diffusion in a potential field. 
A similar procedure has been given in Section 2.3 of \cite{Di_Crescenzo_Giorno_Nobile} for a one-dimensional diffusion.
\begin{theorem}
The two-dimensional linear diffusion process $\widehat{\bf X}_t$  
    having infinitesimal moments \eqref{inf_mom_lin_bid} satisfies a SDE analogous to (\ref{NEWsde_x_trasf}) with potential  
    \begin{align*}
        \widehat{U}({\bf x})&=-\left({\bf x}\cdot M \cdot {\bf x}^\top+{\bf m}\cdot{\bf x}^\top\right)+k \nonumber \\
        &=-\left(\frac{\theta_{11}}{2}x_1^2+\frac{\theta_{22}}{2}x_2^2+\theta_{12}x_1x_2+\theta_{13}x_1+\theta_{23}x_2\right)+k, 
        \qquad {\bf x}\in \mathbb{R}^2, \quad k\in \mathbb{R}.
    \end{align*}
    Hence, the critical point ${\bf x}^*$ of $\widehat{U}({\bf x})$ is given by:
    \begin{equation*}
        {\bf x}^*= \frac{1}{|M|} \left(  \theta_{12}\theta_{23}-\theta_{22}\theta_{13} , \; 
         \theta_{12}\theta_{13}-\theta_{11}\theta_{23}\right).
    \end{equation*}
\end{theorem}
\begin{proof}
The proof follows by direct calculations, considering the differential 1-form 
defined as 
$\widehat{\omega}({\bf x})=-\sum_{i=1}^2\widehat{b_i}({\bf x})\ {\rm d}x_i$,
${\bf x}\in \mathbb{R}^2$, with $\widehat{b_i}({\bf x})$ given in \eqref{inf_mom_lin_bid}. 
\end{proof}
%
\begin{remark}\label{rem:abslinepr}
For the process $\widehat{\bf X}_t$ having infinitesimal moments (\ref{inf_mom_lin_bid}), subject to the potential condition 
(\ref{cond_staz_Y}), we consider the dynamics in the presence of the time-varying linear absorbing boundary 
\begin{equation}
    \label{C(t)}
    C(t)=\{{\bf x}\in \mathbb{R}^2: x_1=h(x_2,t)
    \equiv c^*  x_2+S(t)\},
\end{equation}
with
$$
S(t)=\left\{
\begin{array}{ll}
at+b,  & \hbox{(linear case),} 
\\[2mm]
\displaystyle{\frac{\theta_{13}(\theta_{21}c^*-\theta_{11})+\theta_{23}(\theta_{22}c^*-\theta_{12})}{\nu^2}}
+a\,e^{-\nu t}+b\,e^{\nu t},  & \hbox{(exponential case),}
\end{array}
\right.
$$
for $a,b\in \mathbb{R}$, and where the constants $c^*$ and $\nu$, and the details on the two cases, are given in Table 1 of \cite{Di_Crescenzo95} (where $c^*$ is denoted as $c$). 
The transition p.d.f.\ of $\widehat{\bf X}_t$ satisfies  
the symmetry condition (\ref{cond_simm}), where the transformation 
${\bf \Psi} ({\bf x},t)$ and the function 
$\varphi({\bf x},t)$ are specified in Eqs.\ (4.3) and (4.4) of \cite{Di_Crescenzo95}. 
Moreover, the taboo transition p.d.f.\ of $\widehat{\bf X}_t$ in the presence of the boundary (\ref{C(t)}) is given by (cf.\ Eq.\ (4.12) of \cite{Di_Crescenzo95}) 
\begin{equation}
\widehat{f}^A({\bf x},t\,|\,{\bf y},\tau)
=\widehat{f}({\bf x},t\,|\,{\bf y},\tau)
-\varphi({\bf x},t)\,\widehat f({\bf \Psi}({\bf x},t),t\,|\,{\bf y},\tau),
\label{eq:hatfAWOU}
\end{equation}
for all $t>\tau \geq 0$ and $({\bf x},{\bf y})\in D_1(t) \times D_1(\tau)$ or $({\bf x},{\bf y})\in D_2(t)\times D_2(\tau)$,
where $\widehat{f}$ is specified in Eq.\ (\ref{f_norm}). 
\end{remark}
\par
In the following sections, we analyze in detail some features of the transformed process ${\bf X}_t$ in the special cases when $\widehat{\bf X}_t$ is (i) the two-dimensional Wiener process, and (ii) the two-dimensional OU process, since the analysis of the case (iii), i.e.\ when $\lambda_1 \neq \lambda_2$, can be referred to the previous cases.
Thanks to Corollary \ref{cor_taboo_pdf} and Remark \ref{rem:abslinepr}, the study will include also the determination of the taboo transition p.d.f.\ of the transformed process ${\bf X}_t$. 
\subsection{Two-dimensional transformed Wiener process} \label{wiener_bidim}
For the special case $\delta=0$, i.e.\  $M= ({\bf 0}\; {\bf 0})^\top$ and ${\bf m}\neq {\bf 0}$, from (\ref{inf_mom_lin_bid}) we have that $\widehat{\bf X}_t$ is the two-dimensional Wiener process with infinitesimal moments 
\begin{equation*}
\widehat{b}={\bf m}=(m_1,m_2), \hspace{0.5 cm}  \qquad 
A = \begin{bmatrix}
	a_1^2 & \rho a_1 a_2 \\
	\rho a_1 a_2 & a_2^2
    \end{bmatrix},   
\end{equation*}
where $a_{i} >0$ $(i=1,2),$  $-1 < \rho<1$ and
$\Delta = |A|=a_1^2a_2^2(1-\rho^2)>0$. 
Moreover, since $m_i \in \mathbb{R}$,  $i=1,2$, with $(m_1,m_2) \neq (0,0)$, due to Proposition \ref{prop_w_dim2}, the weight function becomes
$h_c({\bf x})=1+c\exp\bigl\{-2\,{\bf r} \cdot {\bf x}^\top\bigr\}$, for $c\in (0,+\infty)$, 
and the condition (\ref{lambda}) on $\textbf{r}=(r_1,r_2)\in \mathbb{R}^2$ can be given explicitly in the form of the equation of an ellipse, as stated in the following remark. 
%
\begin{remark}\label{rem:ellWiener}
The condition (\ref{lambda}), for $\delta=0$, is verified by any $\textbf{r}\in \mathbb{R}^2$ such that 
\begin{equation}
	r_1^2  a_1^2+2r_1r_2\rho  a_1  a_2 + r_2^2  a_2^2 - m_1r_1-m_2r_2=0.
	\label{equazione_r}
\end{equation}
Furthermore, taking into account that 
\begin{equation}\label{ang_rot}
\gamma
=\left\{
\begin{array}{ll}
\displaystyle  
\frac{ a_1^2-a_2^2+\left({[{\rm Tr}(A)]^2-4|{A}|}\right)^{1/2}}{2 \rho  a_1  a_2}, & \hbox{if}\; \rho\neq 0,\\
0, & \hbox{if}\; \rho= 0
\end{array} 
\right.
\end{equation}
is the rotation angle which defines Eq.\ \eqref{equazione_r}, its solutions can be given in parametric form as 
%
\begin{equation} \label{parametriche_rhodivzero}
    \boldsymbol{r}=(r_1(\upsilon),r_2(\upsilon))=(x_0,y_0)+\frac{1}{\left(1+\gamma^2\right)^{1/2}}\,(s_1,s_2)\cdot \begin{bmatrix}
	\cos \upsilon & -\gamma \cos \upsilon \\
	\gamma \sin \upsilon & \sin \upsilon
    \end{bmatrix},
    \qquad \upsilon\in [0, 2\pi),
\end{equation}
	where $(x_0,y_0)$ is the center
	\begin{equation*}
	(x_0,y_0)= \frac{1}{2}
	\Bigl ( {\frac{m_1  a_2 - m_2 \rho  a_1 }{\Delta} a_2, \frac{m_2  a_1 - m_1 \rho  a_2}{\Delta} a_1 }\Bigr),
	\end{equation*}

and where the amplitudes of the semi-major and semi-minor axes $s_1,\,s_2$ are 
$$
 (s_1,s_2) 
 =\left\{
 \begin{array}{ll}
  \left(\displaystyle\frac{1}{2}\left({\frac{\eta_1\  \boldsymbol{m}\, {A^{-1}}\,\boldsymbol{m}^\top}{\Delta}}\right)^{1/2}, 
  \displaystyle\frac{1}{2}\left(\frac{\eta_2\  \boldsymbol{m}\, {A^{-1}}\,\boldsymbol{m}^\top}{\Delta}\right)^{1/2}\right),    
  &  \hbox{if}\; \rho\neq 0, 
  \\[5mm]
  \Bigg(
  \displaystyle  
  \frac{\left(m_2^2\  a_1^2+m_1^2\  a_2^2\right)^{1/2}}{2 a_1^2 a_2},
  \displaystyle  \frac{\left(m_2^2\  a_1^2+m_1^2\  a_2^2\right)^{1/2}}{2 a_1 a_2^2}
  \Bigg),    
  & \hbox{if}\; \rho= 0,  
 \end{array}
 \right.
$$
	with $\boldsymbol{m}=(m_1\;m_2)$ and the matrix  $A$ given  in \eqref{inf_mom_lin_bid}. 
    Moreover, $\eta_{i}$ for $i=1,2$ are the eigenvalues of $A$, obtained as 
    \begin{equation}\label{eta12}
    \eta_{1,2}
    =\frac{1}{2}\left({\rm Tr}( A)\pm \left([{\rm Tr}(A)]^2-4 | A|\right)^{1/2}\right), 
    \qquad \eta_1>\eta_2,
    \end{equation}
due to Eq.\ \eqref{autovalori_m}. 
\end{remark}
	%
	%
	%
	%

\begin{remark}\label{rem:W2ii}
The two cases treated in Remark \ref{rem:r2choices}, under the assumptions specified in this section, become  
\\
(i) ${\bf r}={\bf m} \cdot A^{-1}
=\displaystyle\frac{1}{ a_1 a_2(1-\rho^2)}\left(\frac{m_1   a_2-m_2 \rho  a_1}{ a_1},\frac{m_2   a_1-m_1 \rho  a_2}{ a_2}\right)$, 
\\
(ii) ${\bf r}= {\frac{m_i}{ a_{i}^2}}\cdot {\bf e}_i$, $i=1,2$.
\\
When $\rho=0$, in addition to the two cases given in (ii), the choice 
${\bf r}= (\frac{m_1}{ a_{1}^2},\frac{m_2}{ a_{2}^2})$ 
is also admissible. The corresponding values of $\upsilon$, using Eqs.\ \eqref{ang_rot} and  (\ref{parametriche_rhodivzero}), are provided in the Table \ref{casi_spec_rho_zero}. 
\begin{table}
    \centering
    \begin{tabular}{c|c|c}
        \hline
        $r_1$ & $r_2$ &  $\upsilon$ \\
        \hline
        ${ \frac{m_1} { a_1^2}}$ & 0 & 
        $\begin{cases}
            -\sign(m_2)\,\arccos\left(
            {
            \frac{m_1 a_2}{\left({m_1^2 a_2^2+m_2^2 a_1^2}\right)^{1/2}}
            }
            \right) 
            \qquad & \textnormal{if } {\bf m} \neq {\bf 0} \\
            0 & \textnormal{if } m_2= 0,m_1>0 \\
            \pi & \textnormal{if } m_2= 0,m_1<0
        \end{cases}$\\
        \hline
        0 & $\frac{m_2}{a_2^2}$ & 
        $\begin{cases}
        \sign(m_2)\,\arccos\left(-
        \frac{m_1 a_2}{\left({m_1^2 a_2^2+m_2^2 a_1^2}\right)^{1/2}}
        \right) 
        \qquad & \textnormal{if } {\bf m} \neq {\bf 0} \\
            \frac{\pi}{2} & \textnormal{if } m_1= 0,m_2>0 \\[5pt]
            \frac{3}{2}\pi & \textnormal{if } m_1= 0,m_2<0
        \end{cases}$\\
        \hline
         $\frac{m_1}{a_1^2}$ 
         & $\frac{m_2}{a_2^2}$ & $\sign(m_2)\,
         \arccos\left( 
         \frac{m_1 a_2}{\left({m_1^2 a_2^2+m_2^2 a_1^2}\right)^{1/2}}
         \right)$;\ ${\bf m}\neq {\bf 0}$ \\
        \hline
    \end{tabular}
    \caption{
    The values of $\upsilon$ for ${\bf r}= {\bf r}(\upsilon)$ in case (ii) of Remark \ref{rem:W2ii}, under the condition $\rho=0$.}
    \label{casi_spec_rho_zero}
\end{table}
\end{remark}

%
\begin{remark}
The transformed process ${\bf X}_t$, derived from the two-dimensional Wiener diffusion, has infinitesimal moments given in Eq.\ \eqref{mom_infinitesimali_fx} for $n=2$, i.e., for ${\bf x}\in \mathbb{R}^2$, 
\begin{equation}
\begin{array}{ll}
	b_i({\bf x}) =
	m_i - \displaystyle\frac{2c \,{a_i}}{\exp\!\bigl\{2\,{\bf r} \cdot {\bf x}^\top \bigr\} + c}
	\left(  a_i r_i + \rho  a_j r_j \right), 
	& \; i,j=1,2, \;\; i\neq j 
    \\[4mm]
	 a_{ij}({\bf x}) =
	\begin{cases}
		 a_{i}^2, & i=j, \\
		\rho a_i a_j, & i \neq j,
	\end{cases}
    & {}
\end{array}
\label{eq:momXtWien}
\end{equation}
where $c\in (0,+\infty)$,   and $\textbf{r}\in \mathbb{R}^2$ satisfies Eq.\ (\ref{equazione_r}).
With reference to Theorem \ref{theor_pot_wiener}, the potential \eqref{potential_wiener_n} of the process ${\bf X}_t$ now becomes 
   \begin{equation}
   U({\bf x})=-m_1 x_1-m_2 x_2-\log(h({\bf x}))\,
   \frac{ { a_1^2 r_1+\rho  a_1  a_2 r_2}}{r_1}, 
   \qquad {\bf x} \in \mathbb{R}^2,
   \label{potential_wiener}
   \end{equation}
where $(r_1,r_2)\neq (0,0)$. Hence, the suitable non-zero values of ${\bf r}$ for the potential (\ref{potential_wiener}) satisfy the  condition (\ref{equazione_r}), which defines an ellipse, and also Eq.\ (\ref{cond_pot}) for $n=2$, i.e. 
   \begin{equation}
       {\rho  a_1  a_2 (r_1^2-r_2^2)+( a_2^2- a_1^2)
      \,r_1 r_2=0}.
      \label{iperbole_deg}
   \end{equation}
The latter defines a degenerate hyperbole, since Eq.\ (\ref{iperbole_deg}) becomes $r_1 r_2=0$ if $\rho=0$, and 
\begin{equation*}
\left(r_1-\gamma\,r_2 \right)\left[r_1-\left(\gamma-
\frac{1}{\rho a_1  a_2} \, 
\left(( a_2^2- a_1^2)^2+4 \rho^2 a_1^2  a_2^2\right)^{1/2}
\right)r_2 \right]=0 
\qquad 
\hbox{if}\; \rho \neq 0,
\end{equation*}
with $\gamma$ given in (\ref{ang_rot}).
\end{remark}
%
%
%
\begin{figure}[t]
    \centering
    \includegraphics[width=0.47\linewidth]{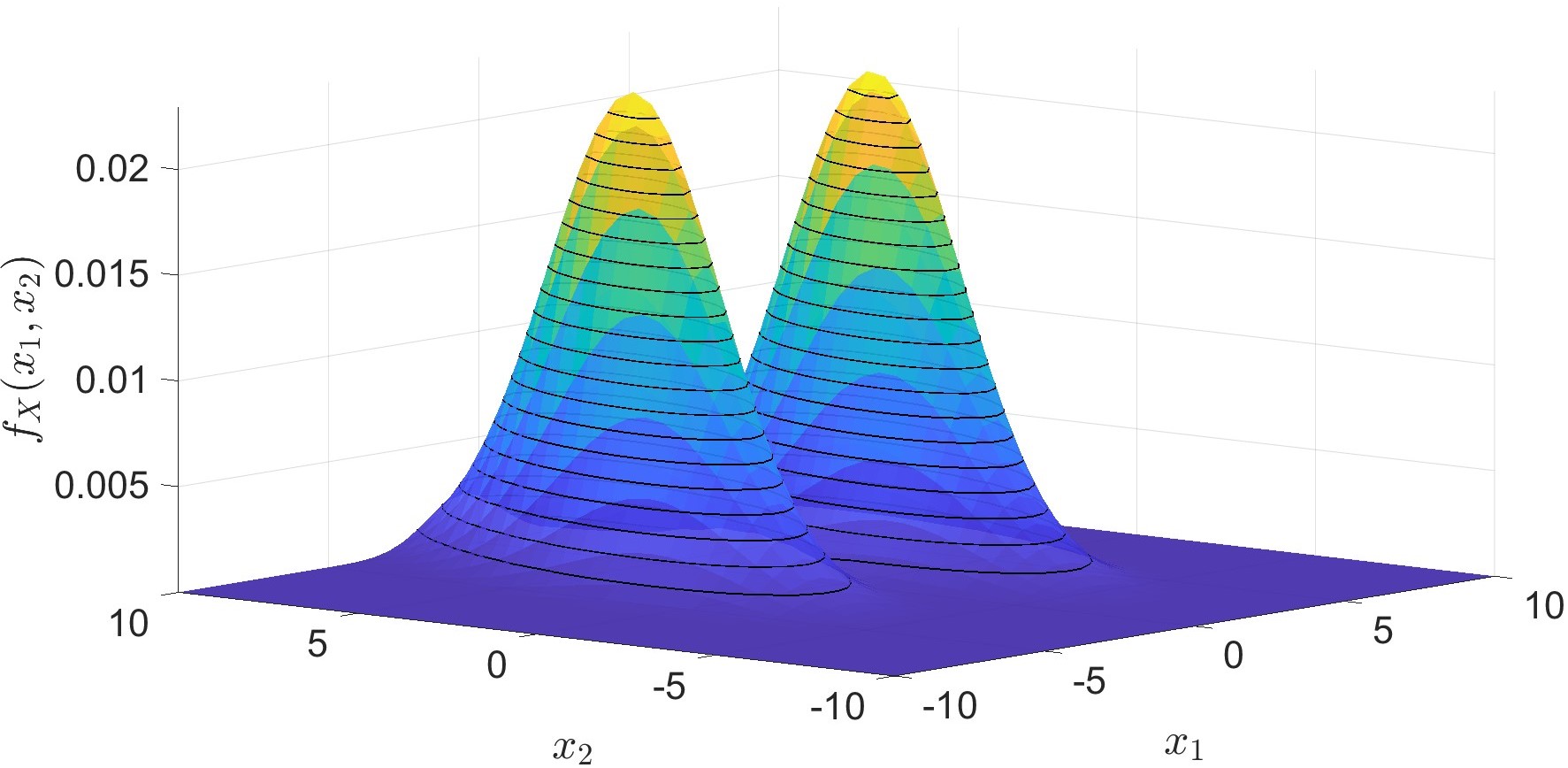}
    \includegraphics[width=0.47\linewidth]{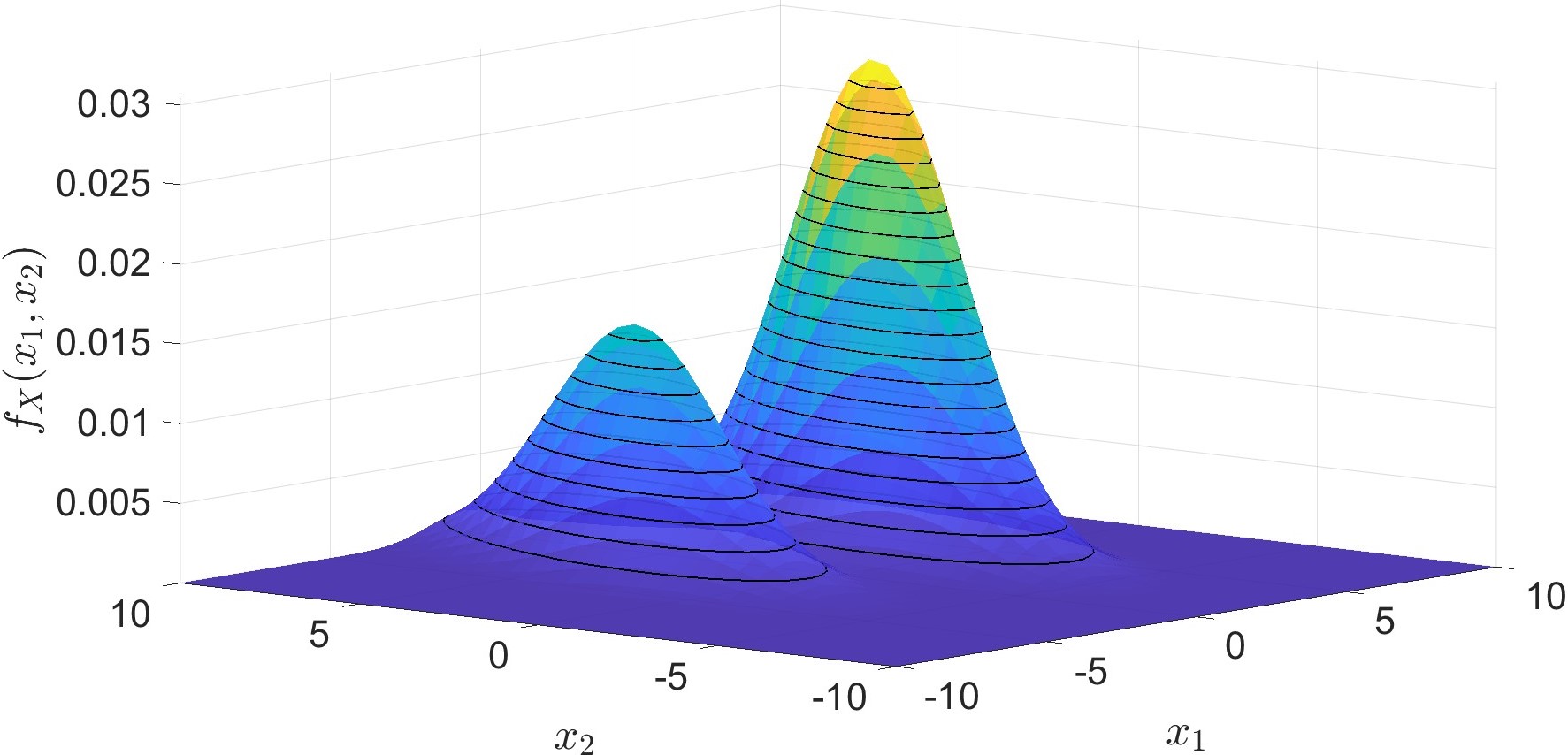}
    \includegraphics[width=0.47\linewidth]{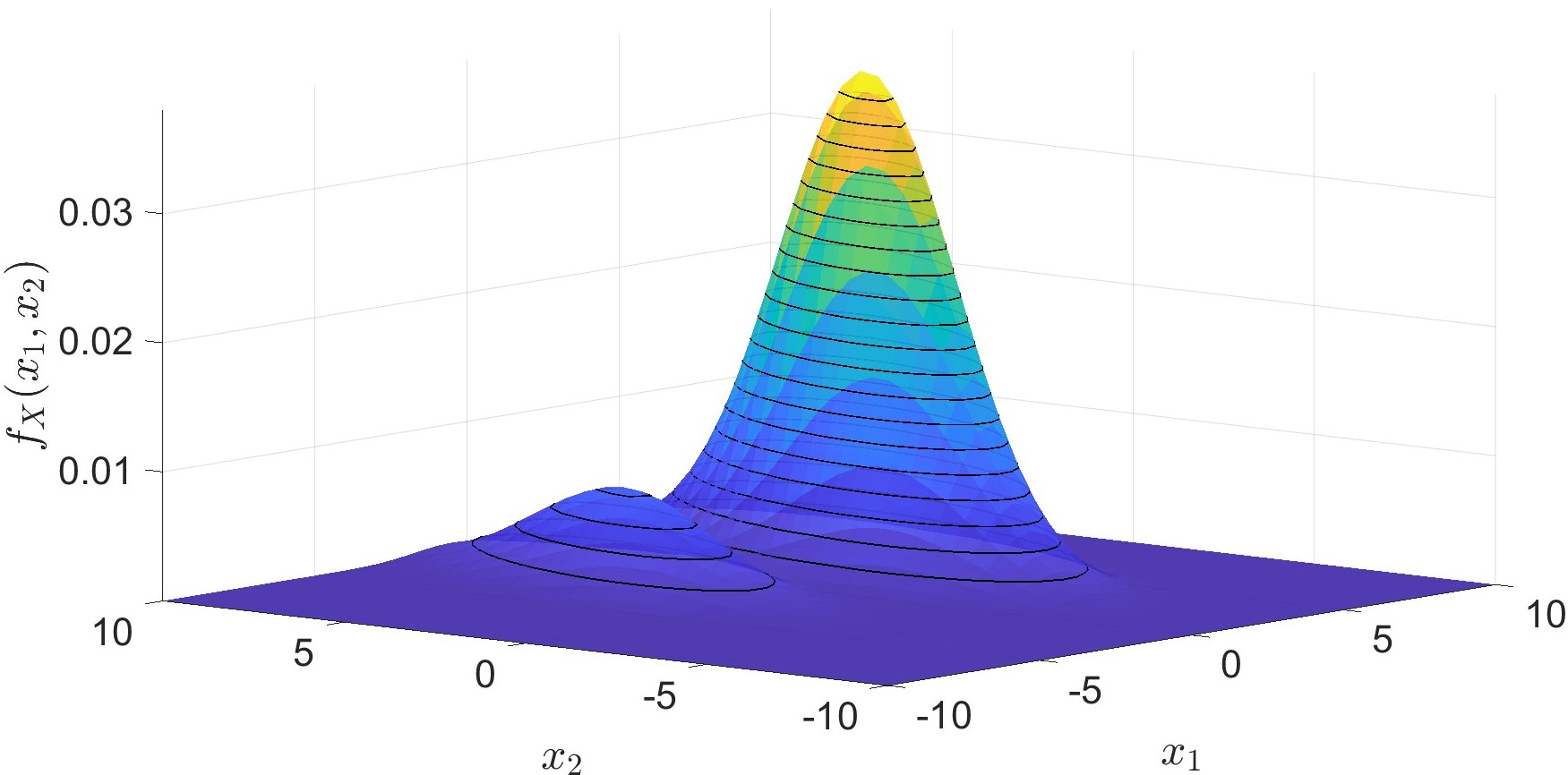}
    \includegraphics[width=0.47\linewidth]{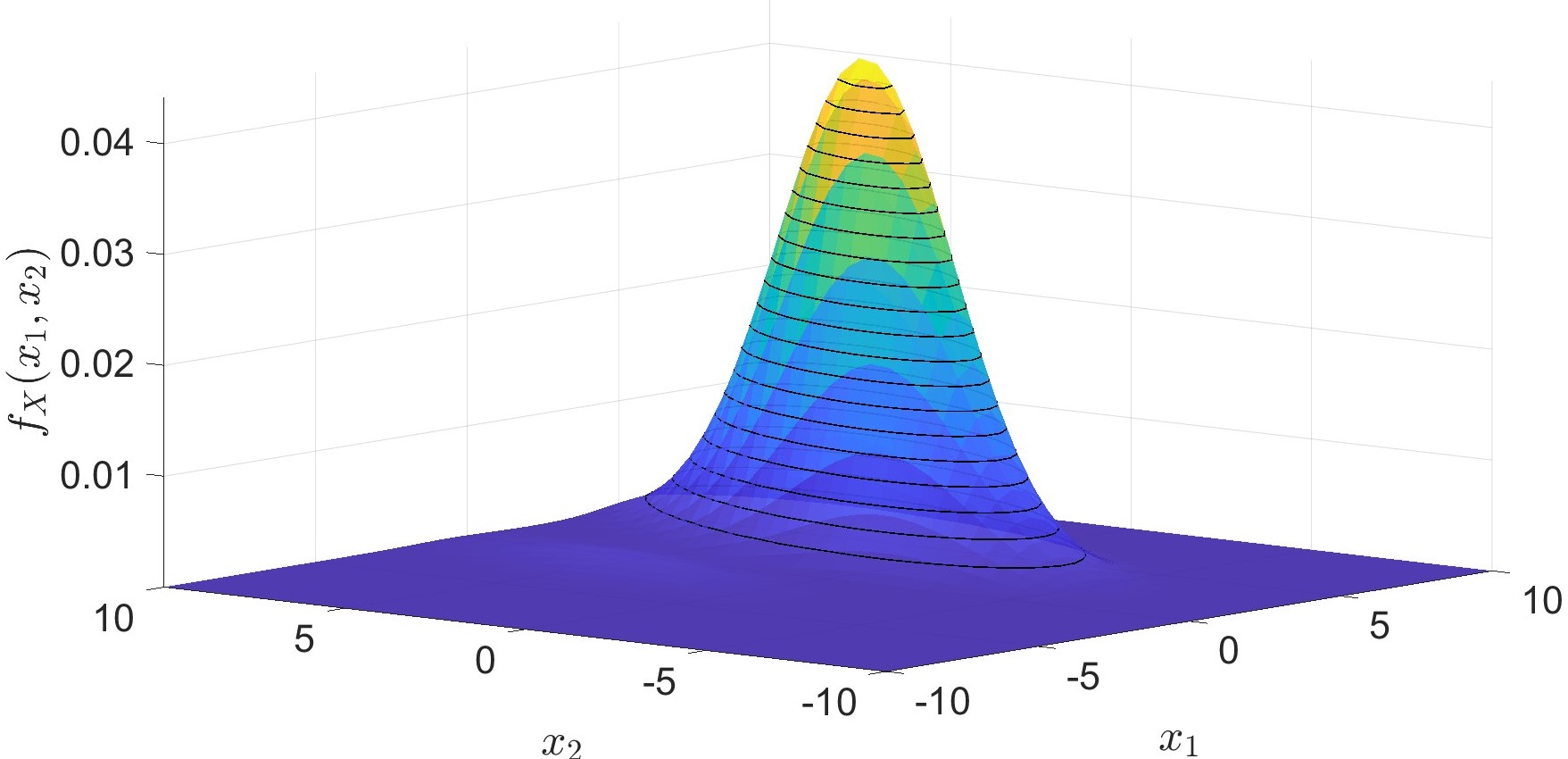}
    \caption{The transition p.d.f.\ of the transformed Wiener process ${\bf X}_t$, for $m_1=-1$, $m_2=2$, $a_1=1$, $a_2=2$, $\rho=0.5$,  ${\bf y}=(0,0)$, $\tau=0$, $t=2$, $(r_1,r_2)=(-2.1094,0.6386)$, with $c=1$ (upper left), $c=2$ (upper right), $c=5$ (lower left) and $c=30$ (lower right).}
    \label{Wiener_case1_dens}
\end{figure}
\begin{figure}[t]
    \centering
     \includegraphics[width=0.49\linewidth]{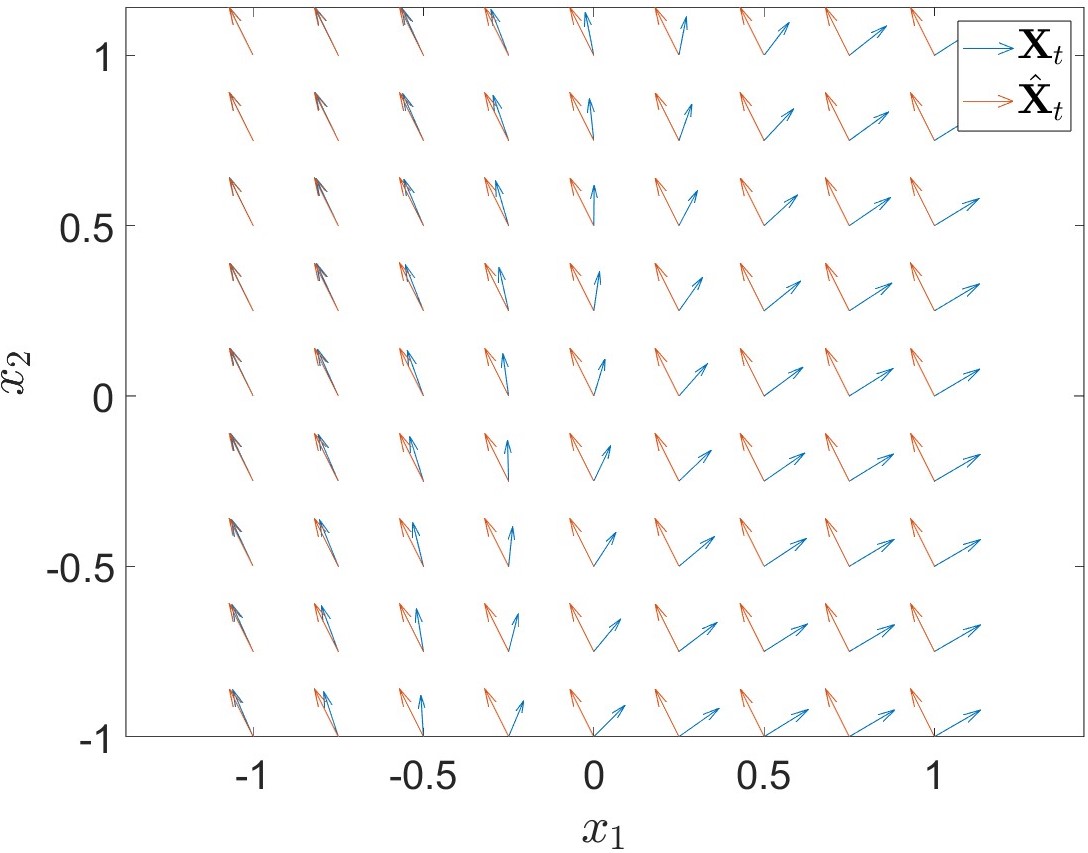}
     \hspace{0.1cm}
     \includegraphics[width=0.49\linewidth]{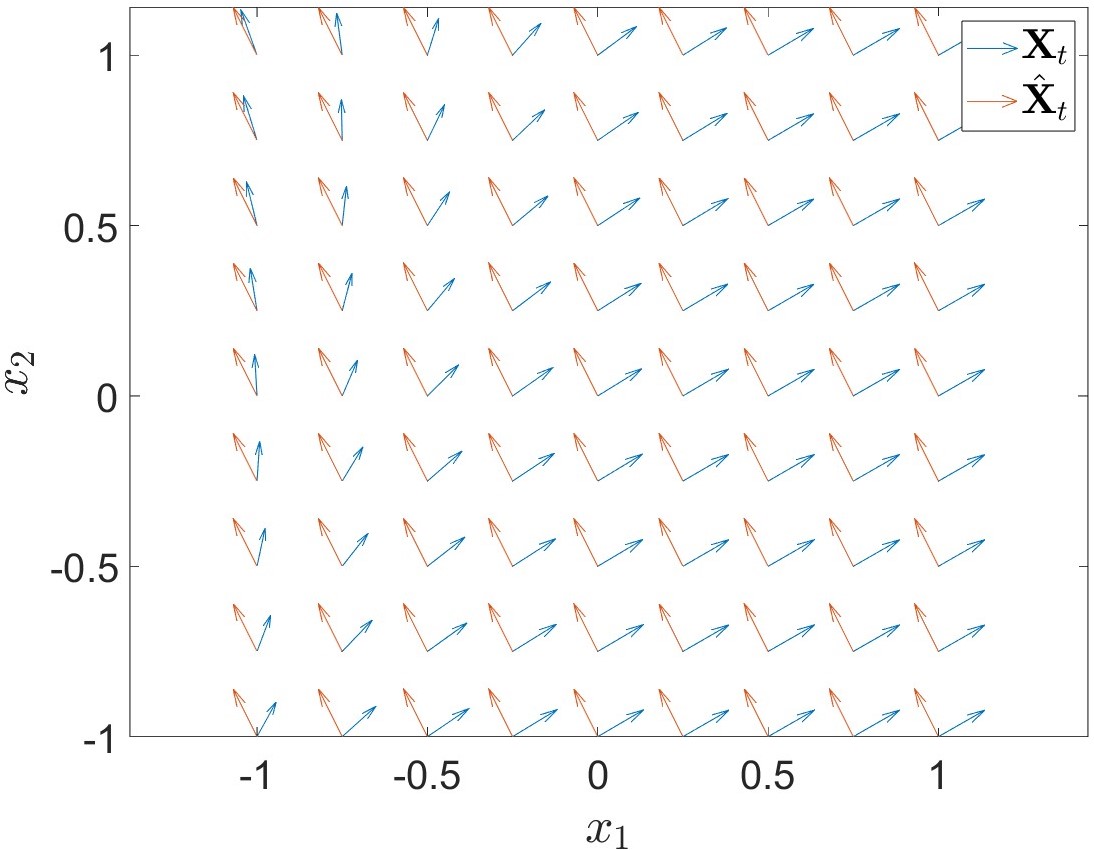}
    \includegraphics[width=0.5\linewidth]{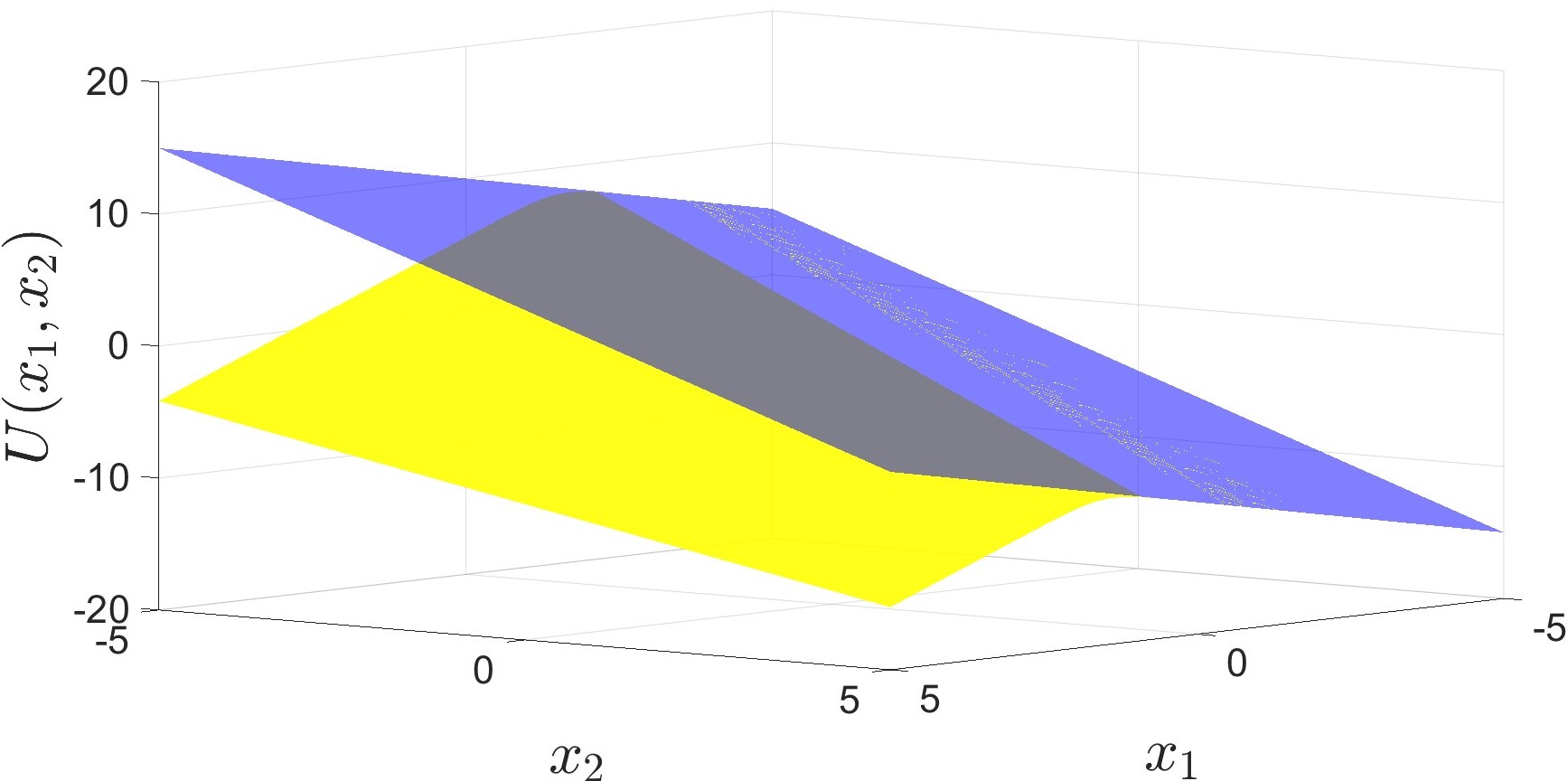}
     \hspace*{-0.4cm}
     \includegraphics[width=0.5\linewidth]{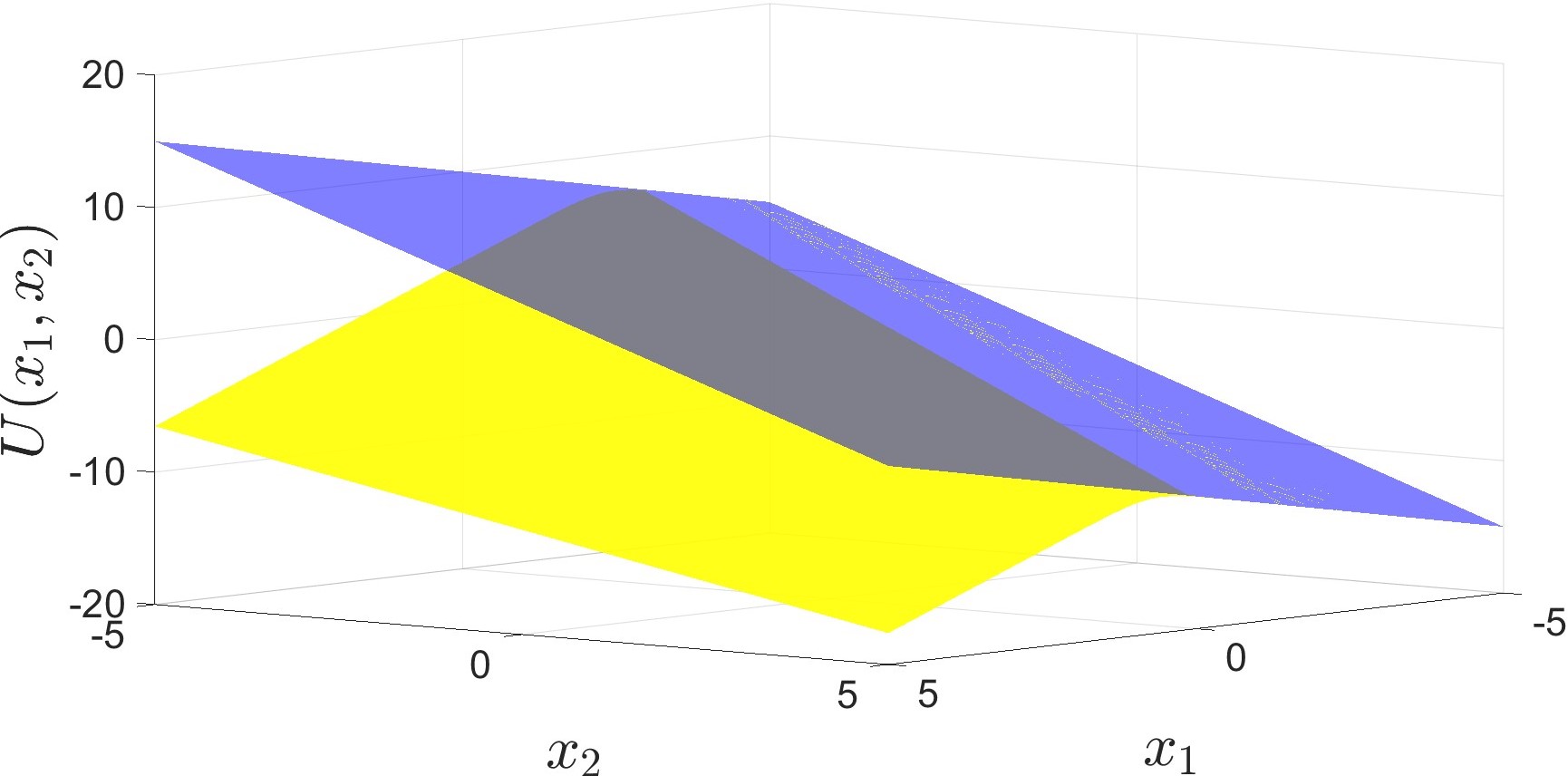}
  \caption{
  The drifts of  ${\bf \widehat{X}}_t$ and ${\bf X}_t$ 
  represented as vector field 
  (on top) and  the potential of ${\bf \widehat{X}}_t$ in blue and of ${\bf X}_t$ in yellow (bottom). The choices of the parameters are $\rho=0.5$,   $ a_1=1$, $ a_2=2$, 
  $m_1=-1$, $m_2=2$, $(r_1,r_2)=(-2.1094,0.6386)$ and $c=1$ (left) or $c=30$ (right).}
    \label{Wiener_case1}
\end{figure}
\par
Figure \ref{Wiener_case1_dens} shows four examples of the transition p.d.f.\ of the two-dimensional transformed Wiener process ${\bf X}_t$, having infinitesimal moments \eqref{eq:momXtWien}. 
We recall that such a density can be expressed as the 2-term mixture (\ref{rel_mixture}) with mixing parameter $\theta_c({\bf y})$ (see Eq.\ \eqref{h_teta_y} for $n=2$). 
In particular, in the first case of Figure \ref{Wiener_case1_dens} we have $\theta_c({\bf y})=1/2$, so $f({\bf x},t\,|\,{\bf y},\tau)$ is bimodal with equal peaks (cf.\ point (iii) of Remark \ref{remark_picchi}). The other cases, for different choices of $c$, confirm that the bimodality of $f$ vanishes when $c$ if sufficiently large. 
Moreover, Figure \ref{Wiener_case1} shows the field representation of the drift vectors of ${\bf X}_t$ for ${\bf x}\in [-1,1]^2$ (in the top), and the potential function $U({\bf x})$ for ${\bf x}\in [-5,5]^2$ and two parameters options (in the bottom). 
In particular, in the cases given in the first column of Figure \ref{Wiener_case1} we have $c=e^{2{\bf r} \cdot {\bf y}^\top}=1$, so $f$ is again bimodal with equal peaks. The plot of the directions of the drifts highlights that the transformation involving $h_c({\bf x})$ only significantly modifies the behavior of the process for some choices of ${\bf x}$. More specifically, the plot of the potential function of both the original and transformed process suggests that the variation largely affects the drift $\bf m$ within an half-plane, and less on the other side. Similar remarks hold for the right-hand side column of the figure. In this case, the choice $c=30$ leads to a significant variation of the drift vectors throughout the whole depicted domain. Due to the weak contribution of $c$ to $U({\bf x})$, the potential function is only slightly different from the $c=1$ case. 
Furthermore, the roof-shaped drift of ${\bf X}_t$, which is visible in both cases of Figure \ref{Wiener_case1}, further highlights another way to assess the bimodality of the density $f$. 
\par
Notice that in both Figures \ref{Wiener_case1_dens} and \ref{Wiener_case1} the choice of the vector $\bf r$ has been made to verify both Eqs.\ \eqref{equazione_r} and \eqref{iperbole_deg}, as required to give a complete description of the diffusion and the potential.
\begin{remark}\label{rem:Wabs2d}
In conclusion, we can apply Corollary \ref{cor_taboo_pdf} to the transformed process ${\bf X}_t$ having infinitesimal moments (\ref{eq:momXtWien}). In this case, the taboo transition p.d.f.\ of ${\bf X}_t$ in the presence of the absorbing boundary (\ref{C(t)}), for $S(t)=at+b$, $a,b\in \mathbb{R}$, can be expressed as 
\begin{equation}
 f^A({\bf x},t\,|\,{\bf y},\tau)
 =\frac{h_c({\bf x})}{h_c({\bf y})}\,
 \widehat{f}^A({\bf x},t\,|\,{\bf y},\tau),
 \label{eq:relfAwhfA}
\end{equation}
for all $t>\tau \geq 0$ and $({\bf x},{\bf y})\in D_1(t) \times D_1(\tau)$ or $({\bf x},{\bf y})\in D_2(t)\times D_2(\tau)$, 
where $h_c$ is given in (\ref{k_lin_proc}) for $\delta =0$, and 
the p.d.f.\ $\widehat{f}^A$ is specified in (\ref{eq:hatfAWOU}). 
\end{remark}
\subsection{Two-dimensional transformed Ornstein--Uhlenbeck process}\label{subs:tOU}
With reference to the two-dimensional diffusion process $\widehat{\bf X}_t$ having infinitesimal moments (\ref{inf_mom_lin_bid}), hereafter we focus on a special case where $\delta=1$, cf.\ (\ref{delta}). Specifically, we refer to the case (b) of Table \ref{table_cases} so we deal with the Ornstein--Uhlenbeck process with infinitesimal moments
\begin{equation*}
\widehat{{\bf b}}=\theta_{11}{\bf x}=(\theta_{11}x_1,\theta_{11}x_2), \hspace{0.5 cm}  \qquad 
A = \begin{bmatrix}
	a_1^2 & \rho  a_1  a_2 \\
	\rho  a_1  a_2 &  a_2^2
    \end{bmatrix},  
\end{equation*}
where $\theta_{11}>0$, $-1 < \rho<1$ and
$\Delta =|A|=a_1^2a_2^2(1-\rho^2)>0$. 
Hence, recalling Proposition \ref{prop_w_dim2}, in the present case the weight function is given by $h_c({\bf x})=1+c\,{\rm Erf}\bigl\{{\bf r}\cdot{\bf x}^\top\bigr\}$, 
where $c\in (-1,0)\cup (0,1)$. Similarly to the transformed Wiener case, in the following remark we note that the vectors $\textbf{r}=(r_1,r_2)\in \mathbb{R}^2$ verifying the condition (\ref{lambda}) are the coordinates of an ellipse.
\begin{remark}
The condition (\ref{lambda}), for $\delta=1$, is verified by any $\textbf{r}\in \mathbb{R}^2$ such that
\begin{equation}
 r_1^2  a_1^2+2r_1r_2\rho  a_1  a_2 + r_2^2  a_2^2 - \theta_{11}=0.
	\label{ellisse_OU}
\end{equation}
Moreover, recalling that the rotation angle of the ellipse in Eq.\ (\ref{ellisse_OU}) is the same expressed in Eq.\ \eqref{ang_rot}, the parametric form of the ellipse can be obtained from Eq.\ \eqref{parametriche_rhodivzero} by taking $(x_0,y_0)=\boldsymbol{0}$ and
%
%
$$
 (s_{1},s_2)=\begin{cases}
     \left(\left({\frac{ \theta_{11}}{\eta_2}}\right)^{1/2},\left({\frac{ \theta_{11}}{\eta_1}}\right)^{1/2}\right), & {\rm if }\; \rho\neq0 \\[10pt]
     \left(\frac{{ \theta_{11}^{1/2}}}{a_{1}},\frac{{ \theta_{11}^{1/2}}}{a_{2}}\right), & {\rm if }\;\rho=0,
 \end{cases},
$$
where $\eta_{1,2}$ are given in Eq.\ \eqref{eta12}. 
\end{remark}
%
\begin{remark}
Due to Eq.\ (\ref{drift_OU}), the process ${\bf X}_t$ obtained by transforming 
the OU process for $n=2$ has infinitesimal moments, 
for ${\bf x}\in \mathbb{R}^2$,
\begin{equation}
\begin{array}{ll}
b_{i}(\textbf{x})=\theta_{11} x_i
    +
    {\displaystyle\frac{c\  a_i \, e^{-({\bf r}\cdot{\bf x}^\top)^2}}{{\pi^{1/2}}
    \left(1+c\,{\rm Erf}[{\bf r}\cdot{\bf x}^\top]\right)}
    \,    (a_i r_i+\rho  a_j r_j)}, 
    & i,j=1,2,\; i\neq j,
 \\[4mm]
 a_{ij}(\textbf{x})=
	\left\{\begin{array}{ll}
		 a_{i}^2,  & i=j, \\
		\rho a_i a_j,  & i \neq j,
	\end{array}\right.
    & {}
\end{array}
\label{eq:mompOUtrans2d}
\end{equation}
with $c\in (-1,0)\cup (0,1)$ and ${\bf r} \in \mathbb{R}^2$ satisfying condition \eqref{ellisse_OU}. In addition, the potential of the transformed OU process ${\bf X}_t$ is:
   \begin{equation*}
   U({\bf x})=-\frac{\theta_{11}(x_1^2+x_2^2)}{2}-\log(h({\bf x}))\,\frac{{ a_1^2 r_1+\rho a_1  a_2 r_2}}{r_1}, 
   \qquad {\bf x} \in \mathbb{R}^2,
   \end{equation*}
provided that ${\bf r}\neq {\bf 0}$ satisfies also the condition (\ref{iperbole_deg}), due to Theorem \ref{theor_pot_OU}. 
\end{remark}
\begin{figure}
    \centering
    \includegraphics[width=0.47\linewidth]{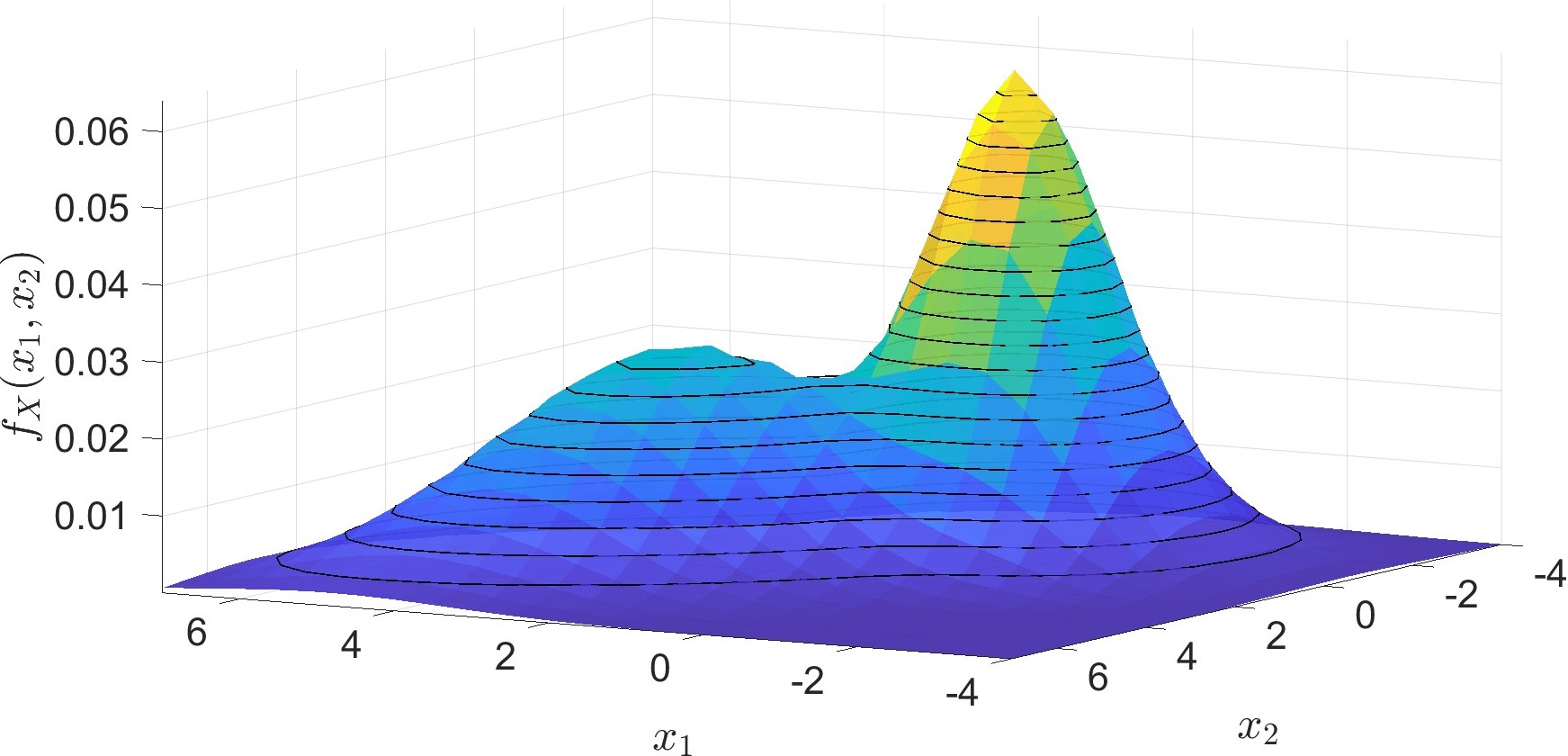}
    \hspace{-1cm} 
    \includegraphics[width=0.47\linewidth]{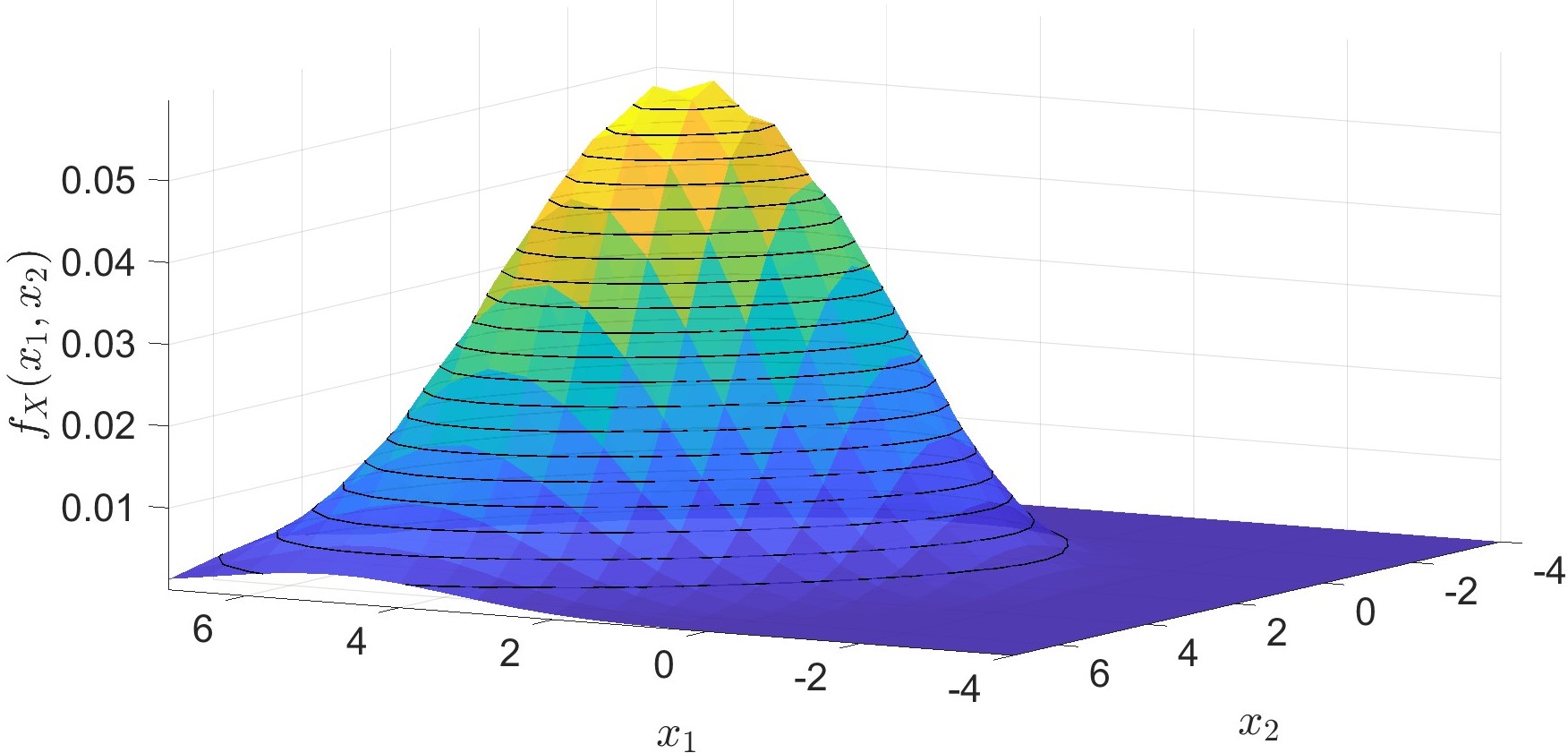}
    \caption{Plots of the density functions of the transformed OU process ${\bf X}_t$ for $n=2$. 
    The choices of the parameters are $\rho=0.5$, $ a_1= a_2=\sqrt{0.5}$, 
    $\theta_{11}=0.5$, $y=(1,1)$, $\tau=0$, $t=2$, $(r_1,r_2)=(0.577350,0.577350)$, with $c=-0.9$ (left) and $c=0.9$ (right).}
    \label{OU_dens}
\end{figure}
\begin{figure}[!h] 
    \centering
    \includegraphics[width=0.45\linewidth]{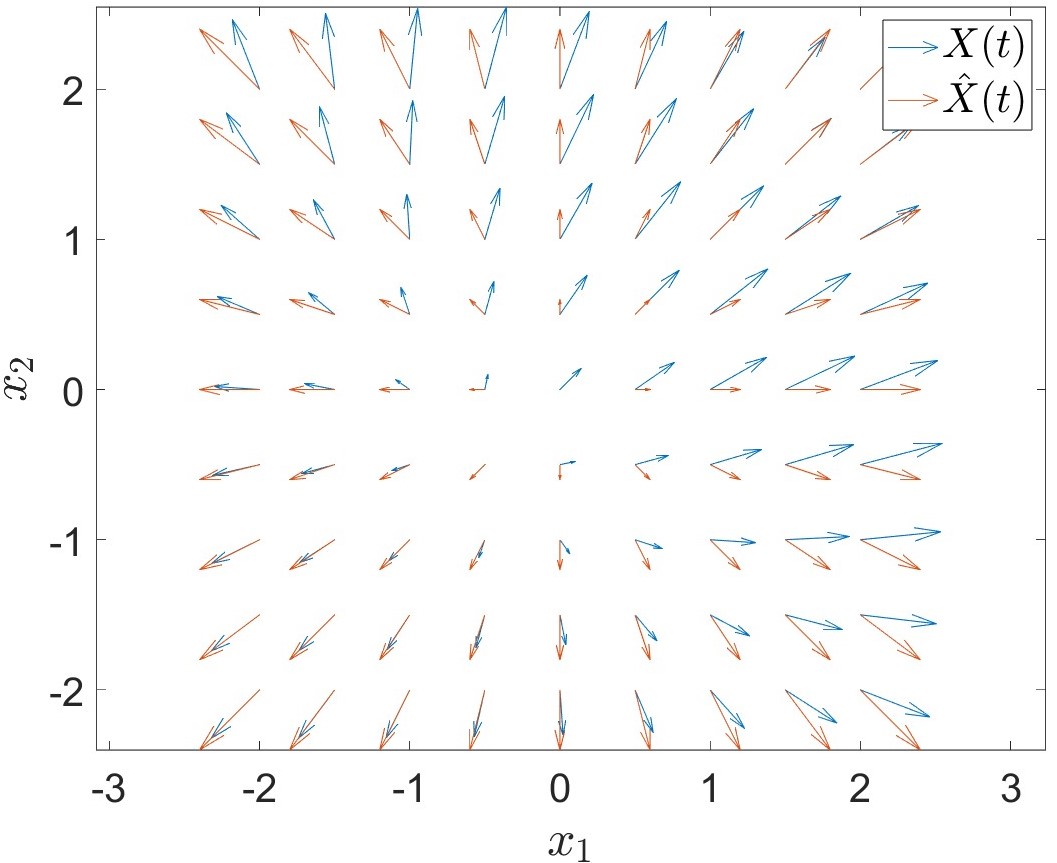}
    \includegraphics[width=0.45\linewidth]{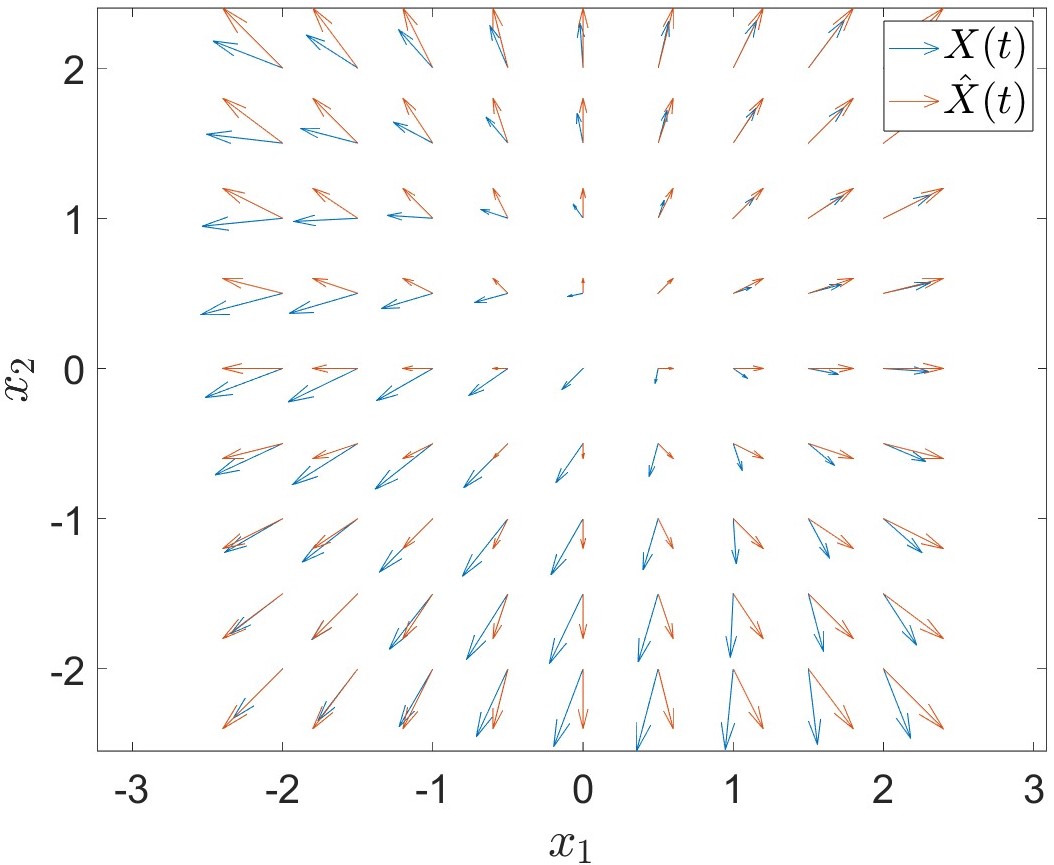}
       \includegraphics[width=0.5\linewidth]{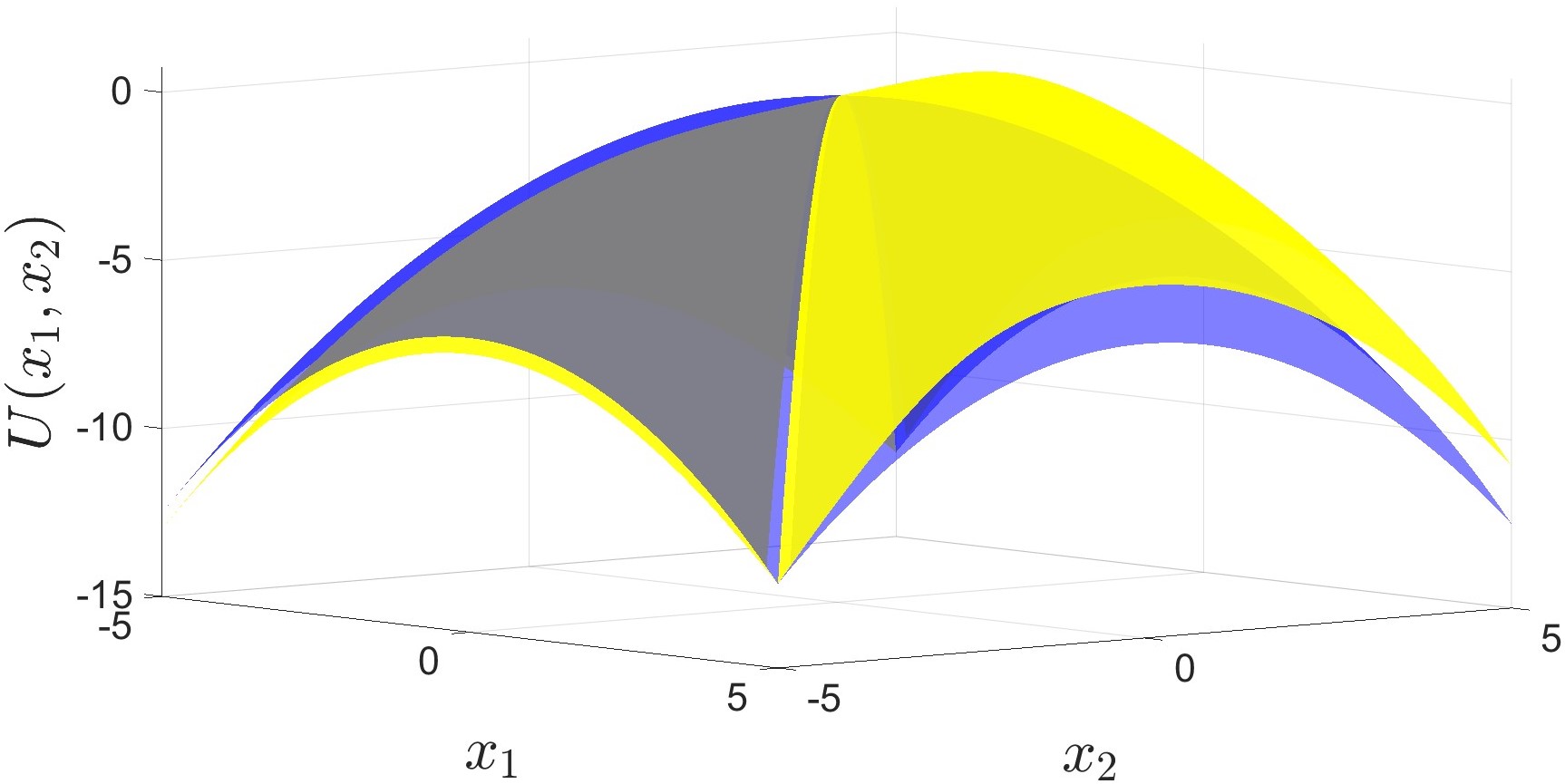}
    \hspace{-0.5cm}
    \includegraphics[width=0.5\linewidth]{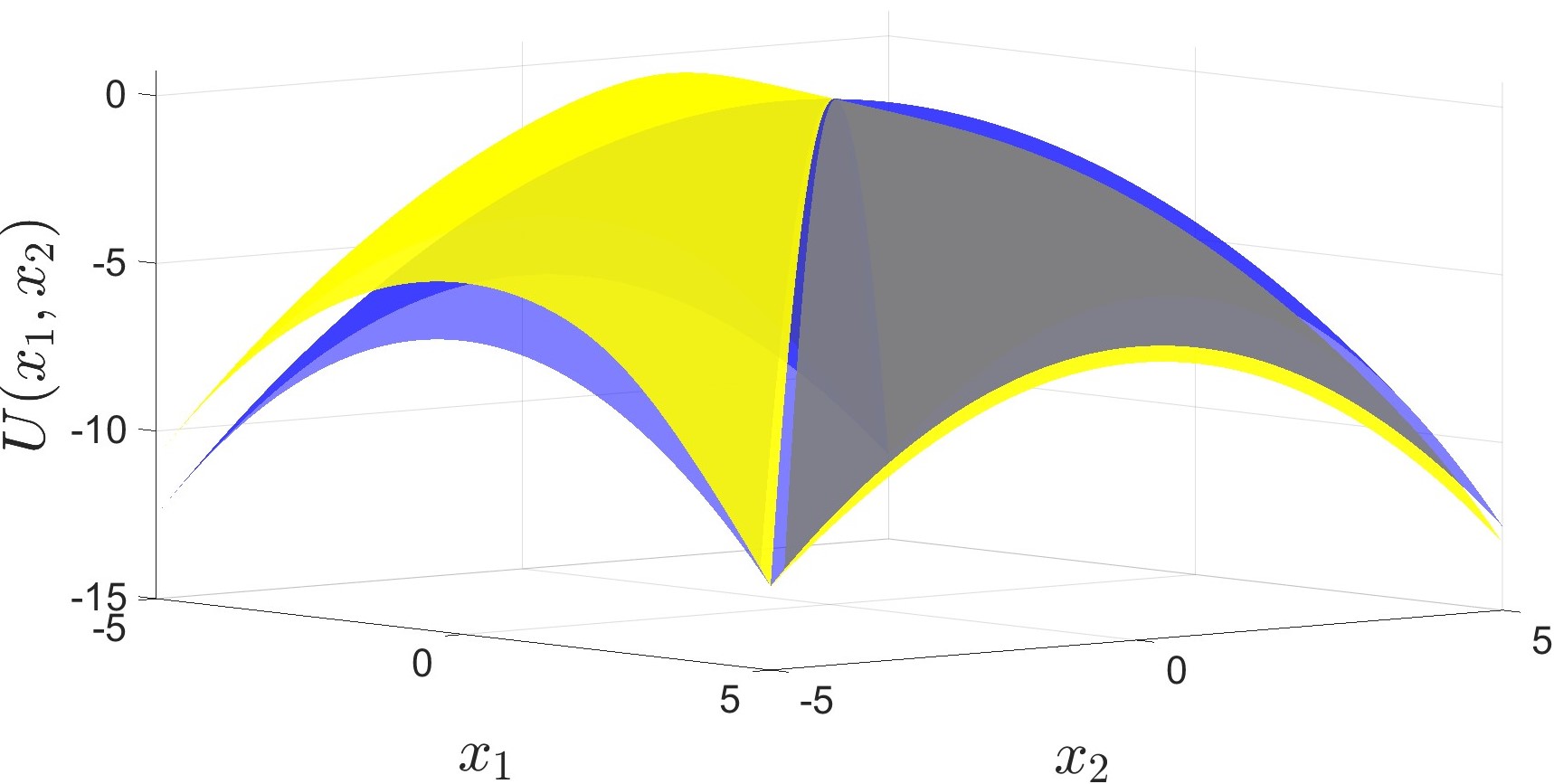}
    \caption{
The drifts of the OU process ${\bf X}_t$ and of the transformed process $\widehat{{\bf X}}_t$ represented as vector fields (on top), and the potentials of ${\bf X}_t$ in yellow and of  ${\bf \widehat{X}}_t$ in blue (bottom). The choices of the parameters are specified in Figure \ref{OU_dens}.}
    \label{OU_case1}
\end{figure}
\par
Two instances illustrating the behavior of the density function for the two-dimensional transformed OU process are considered in Figure \ref{OU_dens}. Specifically, in the first case (on the left), the choice of the parameters shows a bimodal density. Unlike the instance of the transformed Wiener process treated in (iii) of Remark \ref{remark_picchi}, in the transformed OU framework no choice of $c$ leads to  equal peaks, but different maxima values are exhibited by ${f}({\bf x},t\,|\,{\bf y},\tau)$. In the case displayed on the right of Figure \ref{OU_dens}, a different choice of parameter $c$ yields a different behavior. The same sets of parameters are used in Figure \ref{OU_case1}, which shows a comparison of the drift vectors and potential functions of the processes ${\bf X}_t$ and $\widehat{\bf X}_t$. The plot of the drift vectors shows that the transformation varies the intensity of the drift values, although only slightly modifying the direction on one half-plane. Conversely, the drift vectors change significantly in both intensity and direction on the other half-plane. Similarly to the Wiener case, the drift vectors are not drastically affected by the variation of $c$. However, the opposite values of $c$ in the two examples yield a symmetrical transformation of the potential function.
\par
\begin{remark} 
As for the case treated in Remark \ref{rem:Wabs2d}, 
Corollary \ref{cor_taboo_pdf} can be applied to the transformed process ${\bf X}_t$ having infinitesimal moments (\ref{eq:mompOUtrans2d}). 
Now, the taboo transition p.d.f.\ of ${\bf X}_t$ in the presence of the absorbing boundary (\ref{C(t)}), for 
$S(t)= \bar a \,e^{-\theta_{11} t}+\bar b\,e^{\theta_{11} t}$, $\bar a,\bar b\in \mathbb{R}$, 
can be expressed formally as in (\ref{eq:relfAwhfA}), 
where $h_c$ is given in (\ref{k_lin_proc}) for $\delta =1$, and 
the p.d.f.\ $\widehat{f}^A$ is specified in (\ref{eq:hatfAWOU}). 
\end{remark}
\par
We conclude this section noting that the analysis performed so far on the linear diffusion processes with infinitesimal moments of the form (\ref{inf_mom_lin_bid}) allows us also to deal with mixed types, as specified hereafter. 
\begin{remark}
Let us now consider a two-dimensional mixed process $\widehat{\bf X}_t$ with infinitesimal moments:
\begin{equation*}
\widehat{b}(\textbf{x})=(m,\;\eta\, x_2),
\qquad 
\widehat{ a}_{ij}(\textbf{x})=
		\begin{cases}
			 a_{i}^2,  & i=j, \\
			\rho a_i a_j,  & i \neq j,
		\end{cases}
\end{equation*}
where  $m \neq 0$, $\eta>0$, $a_{i} >0$ $(i=1,2)$ and $-1 < \rho<1$, i.e. a mixed diffusion process in which one of the components behaves like a Wiener process while the other one behaves like a OU process. It is not hard to see that a solution of Eq.\ (\ref{2.5}) for such a process is
\begin{equation} \label{kmista}
    h({\bf x})=1+c_1\, {\rm exp} \left\{-\frac{2m}{ a_1^2}\ x_1\right\}
    +c_2 \,{\rm Erf}\left(\frac{{\eta^{1/2}}}{ a_2}\,x_2\right), 
    \qquad {\bf x}\in \mathbb{R}^2,
\end{equation}
with $c_1,c_2>0$.
Making use of (\ref{kmista}), from (\ref{2.5bc}) one then obtains 
the infinitesimal moments of the new process ${\bf X}_t$: 
\begin{eqnarray*}
	&&b_1({\bf x})=\frac{m\left[1-c_1\, {\rm exp}\left\{-\frac{2m}{ { a_1^2}}\,x_1\right\}
	+c_2 \,{\rm Erf}\left(\frac{{\eta^{1/2}}}{{ a_2}}\,x_2\right) \right]
    +2\rho { a_1} c_2 \left({\displaystyle \frac{\eta}{\pi}}\right)^{1/2}\, 
    {\rm exp}\left\{ -\frac{\eta x_2^2}{{ a_2^2} }\right\} }{1+c_1\, {\rm exp}\left\{-\frac{2m}{a_1^2}\,x_1\right\}
    +c_2 \,{\rm Erf}\left(\frac{{\eta^{1/2}}}{a_2}\,x_2\right)},  
    \nonumber\\[10pt]
    &&b_2({\bf x})=\eta x_2
    +2\frac{c_2 { a_2}\, \left({\displaystyle \frac{\eta}{\pi}}\right)^{1/2}\,
    {\rm exp}\left\{ -\frac{\eta x_2^2}{ {a_2^2}}\right\}
    -2{\displaystyle \frac{m\ \rho \ c_1\ { a_2}}{ {a_1}}}  {\rm exp}\left\{-\frac{2m}{{ a_1^2}}\,x_1\right\}}
    {1+c_1\, {\rm exp}\left\{-\frac{2m}{{ a_1^2}}\,x_1\right\}
    +c_2 \,{\rm Erf}\left(\frac{{\eta^{1/2}}}{{ a_2}}\,x_2\right)},  \nonumber
    \\[10pt]
	&&
	{ a_{ij}({\bf x})=\widehat a_{ij}({\bf x})}
    \qquad \qquad (i,j=1,2).
    \nonumber
\end{eqnarray*}
In this case the analysis can be performed by  the same reasoning adopted in Sections \ref{wiener_bidim} and \ref{subs:tOU}. 
\end{remark}

\section{Concluding remarks and perspectives}\label{sect:6}
We have developed a class of exactly solvable multidimensional diffusion models generated through Doob $h$-transforms. 
The proposed construction provides a flexible mechanism for producing transformed diffusions whose transition densities satisfy a product-form relation involving a suitable weight function. 
These explicit relations allow the analysis of several properties of the transformed dynamics, including stochastic ordering, diffusion in potential fields, and Poissonian resetting.
\par
The comparison based on the usual stochastic order, developed in Section \ref{sect:usord}, provides a useful characterization of the qualitative differences between the original and transformed diffusions. 
In particular, the transformed Wiener process admits a representation as a mixture of two Wiener processes, resulting in transition densities with at most bimodal behavior. 
Such multimodal diffusion models may be relevant in applications involving hidden random drifts; for instance, mixtures of Wiener processes have been considered for detecting random drift effects (see Johnson et al.\ \cite{Johnson_etal}).
\par
The spatial transformation introduced in this work can also be interpreted as a transport mechanism for probability mass among different regions of the $n$-dimensional state space. 
This observation suggests possible connections with maximal couplings and optimal transport techniques, including the derivation of bounds for Wasserstein distances between the probability measures of the original and transformed processes. 
Related ideas involving bimodal densities in optimal transport have been discussed, for example, by Lavenant et al.\ \cite{Lavenant}.
\par
In addition to transformations based on the Wiener and OU processes, 
further extensions can be developed toward other relevant diffusion models, such as the lognormal diffusion process, extending the case studied in Section 5.1 of Ricciardi et al.\ \cite{Ricciardi_etal99}). 
Moreover, the proposed framework may provide a basis for constructing denoising Markov models, in analogy with approaches based on Doob $h$-transforms (see Ren et al.\ \cite{Ren}). 
\par
Proposition \ref{propWienerreset} provides the stationary density of the transformed Wiener process under Poissonian resetting. 
However, due to the lack of analogous analytical results, obtaining the stationary distribution of the transformed OU process with resetting remains an open problem. 
\par
Finally, Eq.\ (\ref{2.1}) suggests various additional research lines through the  time-independent ratio of transition densities 
\begin{equation}
\frac{f({\bf x},t\,|\,{\bf y},\tau)}{\widehat{f}({\bf x},t\,|\,{\bf y},\tau)}
={h({\bf x})\over h({\bf y})}. 
\label{eq:rappdensf}
\end{equation}
%
{\em (i) Divergence measures.} 
Eq.\ (\ref{eq:rappdensf}) can be used to derive explicit expressions for measures of relative information between processes $\widehat{\bf X}_{t}$ and ${\bf X}_{t}$, since various divergences or distances can be expressed in terms of the above ratio, such as the Kullback--Leibler divergence, Jeffreys--Kullback--Leibler distance, Cressie--Read power divergence (see, for instance, Nielsen and Okamura \cite{NielsenOkamura} and Vuong et al.\ \cite{Vuong}). 
\par
{\em (ii) Likelihood ratio test.} 
The representation in Eq.\ (\ref{eq:rappdensf}) also provides a natural framework for 
likelihood-based inference, since likelihood ratios as $f/\widehat{f}$ can be expressed directly in terms 
of the weight function. 
\par
{\em (iii) Rejection sampling method.} 
Eq.\ \eqref{eq:rappdensf} is potentially useful for rejection sampling algorithms (for instance, see Maatouk and Bay \cite{Maatouk} and references therein), where the original diffusion ${\bf \widehat{X}}_t$ may serve as an efficient proposal distribution for simulating the transformed process ${\bf X}_t$.
\section*{Data Availability}
 All data generated or analysed during this study are included in this article.

\section*{Declaration of competing interest }

All authors contributed equally to this paper. 
The authors declare that they have no known competing financial interests or personal relationships that could have appeared to influence the work reported in this paper.

\section*{Acknowledgments}
The authors are members of the INdAM (Istituto Nazionale di Alta Matematica) research group GNCS,   
and acknowledge support received from GNCS project ``Metodi analitici e numerici per la modellizzazione di processi stocastici complessi'' (CUP E53C25002010001).

%

%
\end{document}